\documentclass[10pt,a4paper]{amsart}
\usepackage[subsec,skipmbb,skipfonts]{rune}
\usepackage{matestyle}

\title[Straightening for lax transformations and adjunctions of
$(\infty,2)$-categories]{Straightening for lax transformations\\ and
  adjunctions of $(\infty,2)$-categories}

\author{Fernando Abell\'an}
\address{Norwegian University of Science and Technology (NTNU),
  Trondheim, Norway}
\urladdr{https://fernandoabellan.de}

\author{Andrea Gagna}
\address{Institute of Mathematics, Czech Academy of Sciences, Prague, Czech Republic}
\urladdr{https://sites.google.com/view/andreagagna/home}

\author{Rune Haugseng}
\address{Norwegian University of Science and Technology (NTNU),
  Trondheim, Norway}
\urladdr{http://folk.ntnu.no/runegha}

\date{\today}

\begin{document}

\begin{abstract}  
  We prove an unstraightening result for lax transformations between
  functors from an arbitrary $(\infty,2)$-category to that of
  $(\infty,2)$-categories. We apply this to 
  characterize adjoints in $(\infty,2)$-categories of functors and
  (op)lax transformations, and to prove a mate correspondence between
  lax transformations that are componentwise right adjoints and oplax
  transformations that are componentwise left adjoints, for such
  transformations among functors between arbitrary
  $(\infty,2)$-categories.    
\end{abstract}

\maketitle 

\tableofcontents

\section{Introduction}\label{sec:intro}
As the applications of \icats{} have multiplied over the last years,
the importance of \itcatl{} methods for working with \icats{}
themselves has also become clear: in the higher-categorical world it
is almost always necessary to find a conceptual reason for each
construction to exist, since it is impossible to ``just write down''
what it does, and when we want to define functors between various
\icatl{} structures these conceptual reasons often end up involving
\itcatl{} universal properties. In particular, the working \icat{}
theorist surprisingly often finds themself confronted with \emph{lax
  natural transformations}, which are one of of the key new features
that appear when we pass from \icats{} to \itcats{}.

If we have two functors of \itcats{} $F, G \colon \tA \to \tB$, then a
\emph{lax natural transformation} $\phi$ from $F$ to $G$ assigns to
every object $a \in \tA$ a morphism $\phi_{a} \colon F(a) \to G(a)$,
and to every morphism $f \colon a \to a'$ in $\tA$ a \emph{lax} square
\[
  \begin{tikzcd}
   F(a) \ar[r, "\phi_{a}"] \ar[d, "F(f)"'] & G(a)  \ar[d, "G(f)"] \ar[dl, Rightarrow] \\
   F(a') \ar[r, "\phi_{a'}"'] & G(a').
  \end{tikzcd}
\]
More formally, $\phi$ can be defined as a functor
$[1] \otimes \tA \to \tB$, where $\otimes$ denotes the \emph{Gray
  tensor product} of \itcats{}.

In this paper we will prove a number of useful results related to lax
transformations. We will first study the relation between lax
transformations and fibrations of \itcats{}, and then apply this to
study 
adjunctions and mates in
\itcats{}.

\subsubsection*{Lax transformations and fibrations}
Our first main result gives a description of lax transformations of
functors to the \itcat{} $\CATIT$ of \itcats{} in terms of fibrations:
A functor $F \colon \tB \to \CATIT$ corresponds, under the \itcatl{}
straightening equivalence of Nuiten~\cite{Nuiten} and
Abell\'an--Stern~\cite{AGS2}, to a fibration $\tF \to \tB$ where $\tF$
has \emph{cocartesian} morphisms and \emph{cartesian} 2-morphisms; as
we will also need to consider the three other variances of such
fibrations in this paper, we will refer to these as
\emph{$(0,1)$-fibrations}. We show that if $\tG \to \tB$ is the
$(0,1)$-fibration for another such functor $G$, then a lax
transformation from $F$ to $G$ corresponds to a commutative triangle
\[
  \begin{tikzcd}
   \tF \ar[rr, "\phi"] \ar[dr] & & \tG \ar[dl] \\
    & \tB
  \end{tikzcd}
\]
where $\phi$ preserves cartesian 2-morphisms (but does \emph{not} necessarily preserve cocartesian morphisms).
More precisely, we prove the following:
\begin{thmA}[see \ref{thm:fib01laxstr}]\label{thm:strlaxtr}
  For an \itcat{} $\tB$, there is a natural equivalence
  \[ \FUN(\tB, \CAT_{(\infty,2)})^{\lax} \simeq
    \FIB_{(0,1)}^{\lax}(\tB),\] where the left-hand side is the
  \itcat{} of functors and lax natural transformations, and the
  right-hand side is the locally full sub-\itcat{} of $\CATITsl{\tB}$ whose
  objects are $(0,1)$-fibrations and whose morphisms preserve
  cartesian 2-morphisms.
\end{thmA}
We also derive the analogues for other types of fibrations, and for
partially lax transformations. This result extends
\cite{HHLN1}*{Theorem E}, which covers the case of functors to the
\itcat{} $\CATI$ of \icats{} when $\tB$ is an \icat{}, and our proof
generalizes that of \cite{HHLN1}: The key ingredients are the first
author's straightening theorem for local fibrations of \itcats{}
\cite{Ab23}, and an explicit characterization of \emph{Gray
  fibrations} over $\tA \times \tB$ in \cref{propn:grayfib}, which
straighten to functors $\tA \otimes \tB \to \CATIT$; the latter
generalizes results from \cite{HHLN1}*{\S 2.4}.  We note that the key
step in the proof has already been proved by Ayala, Mazel-Gee and
Rozenblyum as \cite{AMGR}*{Theorem B.4.4}, though they do not
explicitly discuss how this gives rise to an equivalence with an
\itcat{} of functors and lax transformations. Moreover,
\cref{thm:strlaxtr} also appears as \cite{GR}*{Ch.~11, 1.1.8}, but the
proof given there may depend on some still-unproven claims. A version of
\cref{thm:strlaxtr} was also proved by Loubaton in \cite{Loubaton}*{\S 6.1.4}
in the more general setting of $(\infty,\infty)$-categories. However, Loubaton's argument
does not characterise fibrations over a Gray tensor product as a subclass of local 
fibrations over the corresponding cartesian product, which is the main input 
to our treatment of adjunctions of $(\infty,2)$-categories.

\subsubsection*{Adjunctions and mates}
One of the main reasons lax transformations come up in practice is
their connection to adjunctions: If we have a lax transformation
$\phi \colon F \to G$ such that the components
$\phi_{a} \colon F(a) \to G(a)$ have left adjoints $\lambda_{a}$, then
we can take the \emph{mates} of the lax naturality squares:
\[
  \begin{tikzcd}
   F(a) \ar[r, "\phi_{a}"] \ar[d, "F(f)"'] & G(a)  \ar[d, "G(f)"] \ar[dl, Rightarrow] \\
   F(a') \ar[r, "\phi_{a'}"'] & G(a').
 \end{tikzcd}
 \quad \mapsto \quad
  \begin{tikzcd}
   G(a) \ar[r, "\lambda_{a}"] \ar[d, "G(f)"'] & F(a)  \ar[d, "F(f)"]  \\
   G(a') \ar[r, "\lambda_{a'}"'] \ar[ur, Rightarrow] & F(a'),
 \end{tikzcd} 
\]
where the transformation in the right-hand square is the \emph{mate}
(or \emph{Beck--Chevalley}) transformation of that in the left-hand square, obtained as the composite
\[ \lambda_{a'}G(f) \to \lambda_{a'}G(f)\phi_{a}\lambda_{a} \to \lambda_{a'}\phi_{a'}F(f)\lambda_{a} \to F(f)\lambda_{a}.\]
We show that if these mate squares actually commute, then they give a left adjoint to $\phi$ in $\FUN(\tA, \tB)^{\lax}$. More precisely, we prove the following:
\begin{thmA}[see \ref{cor:adjinlaxtr}]\label{thm:laxtradj}
  Let $\tA, \tB$ be \itcats{}. If $\phi \colon F \to G$ is a lax transformation between functors $F,G \colon \tA \to \tB$, then:
  \begin{itemize}
  \item $\phi$ has a left adjoint in $\FUN(\tA,\tB)^{\lax}$ \IFF{} the components $\phi_{a} \colon F(a) \to G(a)$ have left adjoints for all $a \in \tA$ and the mate of the naturality square commutes for every morphism in $\tA$.
  \item $\phi$ has a right adjoint in $\FUN(\tA,\tB)^{\lax}$ \IFF{} the components $\phi_{a} \colon F(a) \to G(a)$ have right adjoints for all $a \in \tA$ and the naturality squares of $\phi$ commute.
  \end{itemize}
\end{thmA}
A fibrational version of the key special case where $\tB$ is $\CATI$
appears as \cite{AMGR}*{Lemma B.5.9}.

More generally, we show that taking mates gives an equivalence between lax transformations that are componentwise right adjoints and oplax transformations that are componentwise left adjoints:
\begin{thmA}[see \ref{thm:mategeneral}]\label{thm:matecorr}
  There is a natural equivalence
  \[ \FUN(\tA, \tB)^{\lax,\radj} \simeq (\FUN(\tA, \tB)^{\oplax,\ladj})^{\coop},\]
  given by taking mates of the naturality squares of lax transformations.
\end{thmA}
The case where $\tB$ is $\CATI$ was proved in \cite{HHLN1} when $\tA$
is an \icat{}, and more generally (in a fibrational form) in
\cite{AMGR}*{Lemma B.5.7}, but only as an equivalence of
\igpds{}. The general case of a closely related statement (our
\cref{thm:adjointtoafunctor}) appears as \cite{GR}*{Ch.~12, 3.1.9},
though the proof may again depend on some still-unproven statements;
note in particular that the strategy of the ``alternative proof'' in
\cite{GR}*{Ch.~12, \S 4.2} is close to that of our proof.

There are three main steps in our proofs of the last two theorems:
\begin{enumerate}[(1)]
\item We prove a version of the result for certain fibrations of \itcats{}.
\item We use the straightening equivalence of \cref{thm:strlaxtr} to deduce the case where the target is $\CATI$.
\item We use the Yoneda embedding to extend the result to the general case.
\end{enumerate}
For the first step for \cref{thm:laxtradj}, we extend Lurie's results
on relative adjunctions from \cite{HA}*{\S 7.3.2} in
\S\ref{sec:reladj}, while for \cref{thm:matecorr} we study the bivariant version of Gray fibrations (which we call \emph{Gray bifibrations}) in \S\ref{sec:graybifib}; this generalizes results from \cite{HHLN1}*{\S 3.1}.\footnote{There the \icatl{} analogue of our Gray bifibrations are called \emph{curved orthofibrations}.}

For the third step, we need to understand the compatibility of lax
transformations with the Yoneda embedding. In order to do this, we
extensively study various types of sub-\itcats{} in
\S\ref{sec:subcat} and their behaviour with respect
to lax transformations in \S\ref{sec:sublax}. Combining these results
with some computations of Gray tensor products, we obtain the
following:
\begin{thmA}[see \ref{funlaxemb}]\label{thm:funlaxyoneda}
  For \itcats{} $\tA, \tB$, the Yoneda embedding of $\tB$ induces a
  locally fully faithful functor
  \[ \FUN^{\lax}(\tA, \tB) \to \FUN^{\lax}(\tA \times \tB^{\op},
    \CATI),\]
  whose image consists of those functors $F \colon \tA \times \tB^{\op} \to
  \CATI$ such that $F(a, \blank)$ is a representable presheaf for all
  $a$, and those morphisms $\Phi \colon [1] \otimes (\tA \times
  \tB^{\op}) \to \CATI$ such that $\Phi(\blank, a, \blank)$ is an
  ordinary natural transformation for all $a \in \tA$.
\end{thmA}



\subsubsection*{Notation and conventions}
We now introduce our main notational conventions for \icats{} and \itcats{}:
\begin{itemize}
\item We will denote generic \icats{} by $\oA, \oB, \dots, \oZ$ and generic \itcats{} by $\tA,\tB,\dots,\tZ$.
\item We use \emph{space} and \emph{\igpd{}} synonymously, and write $\Spc$ for the \icat{} thereof.
\item We write $\CatI$ and $\CATI$ for the \icat{} and the \itcat{} of (small) \icats{}, and similarly $\CatIT$ and $\CATIT$ for those of (small) \itcats{}.
\item We denote the mapping spaces in an \icat{} $\oA$  for objects
  $x,y \in \oA$ by $\oA(x,y)$ or $\Map_{\oA}(x,y)$.
\item We denote the mapping \icats{} in an \itcat{} $\tA$  for objects
  $x,y \in \tA$ by $\tA(x,y)$ or $\MAP_{\tA}(x,y)$.
\item We denote the underlying \igpd{} of an \icat{} $\oA$ or an \itcat{} $\tA$ by $\oA^{\simeq}$ or $\tA^{\simeq}$, and the underlying \icat{} of an \itcat{} $\tA$ by $\tA^{\leq 1}$.
\item If $\tA$ and $\tB$ are \itcats{}, the \itcat{} of functors from $\tA$ to $\tB$ (\ie{} the internal Hom for the cartesian product in $\CatIT$) is $\FUN(\tA, \tB)$. Its mapping \icats{} are
  \[ \Nat_{\tA,\tB}(F,G) := \MAP_{\FUN(\tA,\tB)}(F,G).\]
  We write $\Fun(\tA,\tB) := \FUN(\tA,\tB)^{\leq 1}$ for the underlying \icat{}.
\item If $\tA$ is an \itcat{}, we write $\tA^{\op}$, $\tA^{\co}$ and
  $\tA^{\coop}$ for the \itcats{} obtained by reversing, respectively,
  the 1-morphisms in $\tA$, the 2-morphisms in $\tA$, and both.
\item We denote by $C_2$ the walking 2-morphism, \ie{}
  \[
    C_{2} :=
      \begin{tikzcd}[ampersand replacement=\&]
        0 \ar[r, bend left=40, ""{below, name=A}] \ar[r,
        bend right=40, ""{above, name=B}] \& 1.
        \ar[from=A, to=B, Rightarrow]
      \end{tikzcd}
  \]
\end{itemize}

\subsubsection*{Acknowledgments}
We thank Louis Martini for helpful conversations. The second author
was supported by the Praemium Academiae of M.~Markl and
RVO:67985840. 


\section{$(\infty,2)$-categories and lax transformations}\label{sec:twocats}
In this section we review some background material on
$(\infty,2)$-categories and lax transformations, before we study
various types of sub-\itcats{} and their interaction with lax
transformations. We start by recalling the description of \itcats{} as
scaled simplicial sets in \S\ref{sec:scaled} and then consider Gray
tensor products and lax transformations in \S\ref{subsec:gray}. In
\S\ref{sec:pushouts} we compute some pushouts that relate two types of
building blocks for \itcats{}, which we make use of in
\S\ref{sec:aropl} to study oplax arrow \itcats{}, where we in
particular identify their fibres, and various types of sub-\itcats{} in
\S\ref{sec:subcat}. We then show that these are preserved by taking
\itcats{} of lax transformations in \S\ref{sec:sublax}. In
\S\ref{sec:partiallax} we introduce \emph{partially} lax
transformations, and extend some of the preceding results to this
setting. Finally, in \S\ref{sec:laxyoneda} we look at the
compatibility of the Yoneda embedding with lax transformations and
prove \cref{thm:funlaxyoneda}.

\subsection{$(\infty,2)$-categories and scaled simplicial sets}\label{sec:scaled}
In this paper we will primarily work model-independently with \icats{}
and \itcats{}, but a few constructions (particularly related to the
Gray tensor product) will be implemented using scaled simplicial
sets. In this section we therefore briefly recall some definitions
related to this model and its relation to that of marked simplicial
categories.

\begin{defn}\label{defn:scsimplicialset}
  A scaled simplicial set $(X,T_X)$ consists of a simplicial set $X$
  together with a collection $T_X$ of $2$-simplices (or triangles)
  that contains every \emph{degenerate} triangle. We call the elements
  of $T_X$ the \emph{thin} triangles of $X$. A morphism of scaled
  simplicial sets $f \colon (X,T_X) \to (Y,T_Y)$ is a map of simplicial sets
  $f \colon X \to Y$ such that $f(T_X)\subseteq T_Y$. We denote the
  corresponding category of scaled simplicial sets by
  $\sSetsc$.
\end{defn}

\begin{notation}\label{not:flatsharpsc}
  Given a simplicial set $X$ there are two canonical ways of viewing
  it as a scaled simplicial set. First, we can declare only the
  degenerate triangles in $X$ to be thin, in which case we denote the
  resulting scaled simplicial set as $X^{\flat}$. We can
  also define a scaled simplicial set $X^{\sharp}$ by
  taking \emph{every} triangle to be thin.
\end{notation}

\begin{remark}
  There is a model structure on the category $\sSetsc$ of scaled
  simplicial sets that models $\CatIT$
  \cite{LurieGoodwillie,GagnaHarpazLanariScale}. Scaled simplicial
  sets can also be seen as a special case of Verity's \emph{complicial
    sets} \cite{VerityComplI}.
\end{remark}

For later use, we note the following fact about the model structure on
scaled simplicial sets:
\begin{lemma}\label{lem:sattriangles}
  Let $(\Delta^3,U_i)$ be the scaled simplicial set where $U_i$
  contains every triangle except the face that skips the vertex $i$,
  for $0<i<3$. Then the morphism
  $(\Delta^3,U_i) \to \Delta^{3,\sharp}$ is a trivial cofibration in
  $\sSetsc$.
\end{lemma}
\begin{proof}
  This is \cite[Remark 3.1.4]{LurieGoodwillie}.
\end{proof}

\begin{defn}
  A \emph{marked simplicial set} $(X,E_X)$ consists of a simplicial set $X$ together with a collection $E_X$ of $1$-simplices (or edges) that contains every \emph{degenerate} edge. We call the elements of $E_X$ the \emph{marked} edges of $X$. A morphism of marked simplicial sets $f \colon(X,E_X) \to (Y,E_Y)$ is a map of simplicial sets $f\colon X \to Y$ such that $f(E_X)\subseteq E_Y$. We denote the corresponding category of marked simplicial sets by $\sSetm$.
\end{defn}

\begin{defn}\label{defn:catsetmarked}
  We denote by $\sCatm$ the (ordinary) category of categories enriched in marked simplicial sets. There is a model structure on $\sCatm$ that models $\CatIT$, constructed in \cite[A.3.2]{HTT}.
\end{defn}

\begin{defn}\label{defn:orientalsincat+}
  Let $n \geq 0$ and define $\mathfrak{O}^n \in \sCatm$ as the following $\sSetm$-enriched category:
  \begin{itemize}
    \item The objects are the elements of $[n]$.
    \item Given $i,j \in [n]$, we define the mapping marked simplicial set $\mathfrak{O}^n(i,j)$
    as the (minimally marked) poset of subsets $S \subseteq [n]$ such that $\min(S)=i$ and $\max(S)=j$,
    ordered by inclusion.
    \item For a triple $i,j,k \in [n]$, the composition maps are given by taking unions.
  \end{itemize}
\end{defn}

\begin{remark}
  We can view strict $2$-categories as marked simplicial categories by
  taking the nerves of their mapping categories (marked by the
  isomorphisms). The $\sSetm$-enriched categories $\mathfrak{O}^n$
  correspond in this way to the (2-dimensional) \emph{orientals} defined
  by Street~\cite{StreetOrient}.
\end{remark}

\begin{defn}\label{defn:rigidification}
  There is a colimit-preserving functor $\Csc[-]:\sSetsc \to \sCatm$ that is uniquely determined by:
  \begin{itemize}
    \item The values at the minimally scaled simplices are given by
      the cosimplicial object $\Csc[\Delta^{n,\flat}]:=\mathfrak{O}^n$.
    \item The value at a maximally scaled 2-simplex is given by
      $\Csc[\Delta^{2,\sharp}]:=\mathfrak{O}^{2,\sharp}$, where
      $\mathfrak{O}^{2,\sharp}$ is obtained from $\mathfrak{O}^2$ by
      maximally marking $\mathfrak{O}^2(0,2)$ (the only non-trivial
      mapping object).
  \end{itemize}
  We call $\Csc[-]$ the (scaled) \emph{rigidification functor}. By
  \cite[Theorem 4.2.7]{LurieGoodwillie} it is a left Quillen
  equivalence from the model structure of scaled simplicial sets to
  the model structure of marked simplicial categories.
\end{defn}

\begin{observation}\label{obs:free2cat}
  For an \icat{} $\oC$, we can define an \itcat{} $[1](\oC)$
  with two objects $0,1$ and with mapping $\infty$-categories given by
  \[
    [1](\oC)(i,j)=
    \begin{cases}
      \emptyset, & \text{ if } i> j, \\
      [0], & \text{ if } i=j, \\
      \oC, & \text{ if } i=0,j=1.
    \end{cases}
  \]
  We claim that this defines a functor
  $[1](\blank) \colon \CatI \to \CatIT$; this takes the initial object
  to $\partial [1] = [1](\emptyset)$ and so also gives a functor
  $\CatI \to (\CatIT)_{\partial [1]/}$.  The functor $[1](\blank)$ can
  be defined either in the setting of enriched \icats{}, by taking the
  free $\CatI$-\icat{} on a graph (see \cite{adjmnd}*{Remark 2.17}),
  or in terms of models by implementing the same formula as a functor
  from marked simplicial sets to marked simplicial categories. In both
  approaches, we can show that
  $[1](\blank) \colon \CatI \to (\CatIT)_{\partial [1]/}$ is left
  adjoint to the functor that takes an \itcat{} $\tX$ with a pair of
  objects $x,y$ to the mapping \icat{} $\tX(x,y)$. In the model
  category case, we can use \cite{HTT}*{Proposition A.3.2.4} to see
  that $[1](\blank)$ is a left Quillen functor from the model
  structure on marked simplicial sets to the model structure on marked
  simplicial categories. Since $[1](\oC)(0,1) \simeq \oC$, we see that
  $[1](\blank)$ is fully faithful as a functor to
  $(\CatIT)_{\partial [1]/}$.
\end{observation}

A particularly useful aspect of the model of scaled simplicial sets for us is that they can be used to define ``partially
lax'' versions of \itcats{}, as follows:

\begin{defn}
  A \emph{scaled \itcat{}} is an \itcat{} $\tC$ together with
  a sub-\igpd{} $S \subseteq \Map([2],\tC)$, such that $S$
  contains the \igpd{} $S_{\tC}^{\flat}$ of path components corresponding to triangles
  \[
    \begin{tikzcd}[column sep=tiny,row sep=small]
      {} & y \arrow{dr}{g} \\
      x \arrow{ur}{f} \arrow{rr} & & z
    \end{tikzcd}
  \] where either $f$ or $g$ is an equivalence. (Note that such an $S$
  is uniquely determined by its path components as a subset of
  $\pi_{0}\Map([2], \tC)$.)
  We write
  $\tC^{\flat} := (\tC, S_{\tC}^{\flat})$ for $\tC$
  equipped with this minimal scaling, and
  $\tC^{\sharp} = (\tC, S_{\tC}^{\sharp})$ for $\tC$
  equipped with the maximal scaling consisting of all commuting
  triangles in $\tC$.
\end{defn}

\begin{construction}\label{constr:laxitcat}
  Let $(\tB, S)$ be a scaled \itcat{}, and suppose
  $\mathfrak{B}^{\natural} = (\mathfrak{B}, S^{\natural})$ is a
  fibrant scaled simplicial set that represents the \itcat{}
  $\tB$. Then we can consider the (non-fibrant) scaled simplicial
  set $(\mathfrak{B}, S)$, where we abusively write $S$ for the set of
  2-simplices in $\mathfrak{B}$ that correspond to the commuting
  triangles $S$ in $\tB$. We write $\tB_{S}$ for the
  \itcat{} represented by a fibrant replacement of
  $(\mathfrak{B}, S)$. If $S$ is the minimal scaling
  $S_{\tB}^{\flat}$, we will write
  \[ \tB_{\oplax} := \tB_{S_{\tB}^{\flat}}. \]
\end{construction}

Informally, the \itcat{} $\tB_{S}$ should be thought of as obtained
from $\tB$ by ``replacing the invertible 2-morphisms in the
commuting triangles that don't lie in $S$ with non-invertible
2-morphisms''. A functor of \itcats{} $\tB_{S} \to \tc{C}$
similarly represents an oplax functor from $\tB$ that is an ordinary
functor on the triangles in $S$; this idea is made precise in work of
the first author \cite{AbLax}.

\begin{observation}\label{obs:2foldseg}
  Finally, we briefly recall two other useful descriptions of the
  \icat{} $\CatIT$, which will play a small role in this paper:
  \begin{itemize}
  \item Let $[n]([i_{1}],\dots,[i_{n}])$ be the (strict) 2-category
    with objects $0,\ldots,n$ and mapping categories
    \[ \MAP(s,t) =
      \begin{cases}
        \emptyset, & s > t, \\
        [i_{s+1}] \times \cdots \times [i_{t}], & s \leq t,                   
      \end{cases}
    \]
    and composition given by the obvious isomorphisms
    \[\MAP(s,t) \times \MAP(t,u) \cong \MAP(s,u).\] The category
    $\Theta_{2}$ is the full subcategory of $\CatIT$ on these objects,
    and the restricted Yoneda embedding
    \[ \CatIT \to \Fun(\Theta_{2}^{\op}, \Spc)\] identifies $\CatIT$
    with the full subcategory satisfying certain Segal and
    completeness conditions introduced by Rezk~\cite{RezkThetaN}.
  \item There is a functor
    $\tau \colon \simp \times \simp \to \Theta_{2}$ that takes
    $([n],[m])$ to $[n]([m],\ldots,[m])$. Restriction (on presheaves
    of \igpds{}) along this gives an equivalence between Rezk's
    $\Theta_{2}$-spaces and Barwick's 2-fold complete Segal spaces
    \cite{BarwickThesis,BSPUnicity,BergnerRezk,thetan}. This can be
    rephrased as an identification between $\CatIT$ and the full
    subcategory of $\Fun(\Dop, \CatI)$ containing the functors
    $X_{\bullet} \colon \Dop \to \CatI$ such that:
    \begin{enumerate}[(i)]
    \item $X$ satisfies the Segal condition, \ie{}
      \[ X_{n} \simeq X_{1} \times_{X_{0}} \cdots \times_{X_{0}}
        X_{1}. \]
    \item The \icat{} $X_{0}$ is an \igpd{}.
    \item The underlying Segal space $X_{\bullet}^{\simeq}$ is
      complete.
    \end{enumerate}
  \end{itemize}
\end{observation}

\subsection{Gray tensor products and lax transformations}\label{subsec:gray}
In this section we recall the construction of \emph{Gray tensor products} of \itcats{}. Roughly speaking, the idea is that where the cartesian product $\tA \times \tB$ of two \itcats{} $\tA, \tB$ has a commuting square
\[
  \begin{tikzcd}
    (a,b) \arrow{d}[swap]{(f,\id_{b})} \arrow{r}{(\id_{a},g)} & (a,b') \arrow{d}{(f,\id_{b'})} \\
    (a',b) \arrow{r}[swap]{(\id_{a'},g)} & (a',b')
  \end{tikzcd}
\]
for every pair of morphisms $f \colon a \to a'$ in $\tA$ and $g\colon b \to b'$ in $\tB$, the Gray tensor product has an oplax square
\[
  \begin{tikzcd}
    (a,b) \arrow{d}[swap]{(f,\id_{b})} \arrow{r}{(\id_{a},g)} & (a,b') \arrow{d}{(f,\id_{b'})} \\
    (a',b) \arrow[Rightarrow]{ur} \arrow{r}[swap]{(\id_{a'},g)} & (a',b').
  \end{tikzcd}
\]
There are several ways to make this idea precise for \itcats{}; see
for instance
\cite{JohnsonFreydScheimbauer,MaeharaGray,CampionMaehara,CampionGray}.
Here we will use the construction via scaled simplicial sets,
following Gagna, Harpaz, and Lanari~\cite{GagnaHarpazLanariGray};
their definition is a variant of that introduced by
Verity~\cite{VerityComplI} in the context of complicial sets. We can
view the definition as a special case of \cref{constr:laxitcat}:
\begin{defn}\label{defn:graytens}
  Suppose $\tc{A}$ and $\tB$ are \itcats{}. Let $S\subseteq \Map([2],\tc{A} \times \tB)$ be the sub-\igpd{}
  consisting of those path components spanned by triangles of the form
  \[\left(\begin{tikzcd}
        & {a_1} &&& {b_1} \\
        {a_0} && {a_2,} &  {b_0} && {b_2}
        \arrow["{f_{01}}", from=2-1, to=1-2]
        \arrow["{f_{12}}", from=1-2, to=2-3]
        \arrow["{f_{02}}"', from=2-1, to=2-3]
        \arrow["{g_{01}}", from=2-4, to=1-5]
        \arrow["{g_{02}}"', from=2-4, to=2-6]
        \arrow["{g_{12}}", from=1-5, to=2-6]
      \end{tikzcd}\right)\]
  where either $f_{12}$ or $g_{01}$ is an equivalence. The \emph{Gray
    tensor product} $\tc{A} \otimes \tB$ is then the \itcat{}
  $(\tc{A} \times \tB)_{S}$ associated to the scaled \itcat{}
  $(\tc{A} \times \tB, S)$.
\end{defn}

\begin{defn}\label{defn:ghlgray}
  There exists a left Quillen bifunctor
  $\otimes \colon \sSetsc \times \sSetsc \to \sSetsc$ which models the
  Gray tensor product (see \cite{GagnaHarpazLanariGray} for
  details). Given scaled simplicial sets $(X,T_X)$ and $(Y,T_Y)$ then
  the Gray tensor product $(X,T_X)\otimes (Y,T_Y)$ is given by the
  scaled simplicial set $(X\times Y, T^{\operatorname{gr}}_{X,Y})$,
  where the underlying simplicial set is the cartesian product $X
  \times Y$, and a triangle $\sigma=(\sigma_X,\sigma_Y) \colon
  \Delta^2 \to X \times Y$ belongs to $T^{\operatorname{gr}}_{X,Y}$ if
  $\sigma_X \in T_{X}$ and $\sigma_Y \in T_{Y}$ and at least one of the following conditions hold:
  \begin{itemize}
  \item The triangle $\sigma_X$ is of the form $s_{1}(e)$ for some
    edge $e \in X_{1}$.
  \item The triangle $\sigma_Y$ is of the form $s_{0}(e)$ for some
    edge $e \in X_{1}$.
  \end{itemize}
\end{defn}

\begin{warning}
  The Gray tensor product is \emph{not} symmetric, which means there
  are two possible conventions for its ordering. Our definition agrees with that
  used in \cite{GagnaHarpazLanariLaxLim}.
\end{warning}

Since the Gray tensor product of scaled simplicial sets is a left
Quillen bifunctor, we get the following on the \icatl{} level:
\begin{thm}[Gagna--Harpaz--Lanari~\cite{GagnaHarpazLanariGray}]\label{thm:graycolim}
  The Gray tensor product defines a functor of \icats{}
  \[ \blank \otimes \blank \colon \CatIT \times \CatIT \to \CatIT\]
  that preserves colimits in each variable.
\end{thm}

\begin{warning}
  The Gray tensor product is \emph{not} compatible with the cartesian
  product: for \itcats{} $\tA, \tB, \tC$ we have a natural map
  \[ \tA \otimes (\tB \times \tC) \to (\tA \otimes \tB) \times \tC,\]
  but this is not usually an equivalence; instead, we will see in
  \S\ref{sec:laxyoneda} that it is a localization at certain
  2-morphisms. It follows that the Gray tensor product is \emph{not} a
  functor of \itcats{} (though it can be made into one if instead of
  natural transformations we consider so-called ``icons'' as
  2-morphisms).
\end{warning}

Since $\CatIT$ is a presentable \icat{}, \cref{thm:graycolim} implies that the Gray tensor product has adjoints in each variable:
\begin{defn}
  Given \itcats{} $\tc{A}, \tB, \tc{C}$, there are natural \itcats{} $\FUN(\tB, \tc{C})^{\lax}$ and $\FUN(\tB, \tc{C})^{\oplax}$ characterized by natural equivalences
  \[ \Map(\tc{A}, \FUN(\tB, \tc{C})^{\oplax}) \simeq \Map(\tB \otimes \tc{A}, \tc{C}),\]
  \[ \Map(\tc{A}, \FUN(\tB, \tc{C})^{\lax}) \simeq \Map(\tc{A} \otimes \tB, \tc{C}).\]  
\end{defn}

\begin{remark}
  In the notation of \cite{GagnaHarpazLanariLaxLim}, our $\FUN(\tA,\tB)^{\oplax}$
  corresponds to $\Fun^{\txt{opgr}}(\tA,\tB)$ and our
  $\FUN(\tA,\tB)^{\lax}$ to $\Fun^{\txt{gr}}(\tA,\tB)$. Let us point out that in \cite[p.~9]{GagnaHarpazLanariGray} the opposite convention appears.
\end{remark}

It is immediate from the definition that
$[0] \otimes \tA \simeq \tA \simeq \tA \otimes [0]$, so the objects of
$\FUN(\tA, \tB)^{\plax}$ are the functors $\tA \to \tB$. The morphisms
are (op)lax natural transformations in the following sense:
\begin{defn}
  An \emph{oplax natural transformation} of functors of \itcats{}
  $\tc{A} \to \tB$ is a functor $\tc{A} \otimes [1] \to \tB$, while a
  \emph{lax natural transformation} is a functor
  $[1] \otimes \tc{A} \to \tB$. We write $\Nat^{\plax}_{\tA,\tB}(F,G)$ for the mapping \icat{} in $\FUN(\tA,\tB)^{\plax}$ for functors $F,G \colon \tA \to \tB$. We also write $\Fun(\tA,\tB)^{\plax}$ for the underlying \icat{} of $\FUN(\tA,\tB)^{\plax}$.
\end{defn}

\begin{defn}
  We will say that an (op)lax natural transformation between functors
  $\tA \to \tB$ is \emph{strong} if all of its (op)lax naturality
  squares commute, \ie{} if for every morphism in $\tA$ the
  corresponding diagram $[1] \otimes [1] \to \tB$ takes the
  non-trivial 2-morphism in the source to an invertible 2-morphism in
  $\tB$. We will see later in \cref{cor:stronglax} that strong (op)lax
  transformations are equivalent to ordinary natural transformations.
\end{defn}

The Gray tensor product is associative; this is easy to see in scaled simplicial sets, see \cite{GagnaHarpazLanariGray}*{Proposition 2.2 and Corollary 2.15}. This gives the following relations between the Gray tensor and $\FUN(\blank,\blank)^{\plax}$:
\begin{lemma}\label{lem:funandgray}
  For \itcats{} $\tA, \tB, \tC$ there are natural equivalences
  \[
    \begin{split}
      \FUN(\tA, \FUN(\tB, \tC)^{\lax})^{\lax} & \simeq \FUN(\tA \otimes \tB, \tC)^{\lax}, \\
      \FUN(\tA, \FUN(\tB, \tC)^{\oplax})^{\oplax}& \simeq \FUN(\tB \otimes \tA, \tC)^{\oplax},\\
      \FUN(\tA, \FUN(\tB, \tC)^{\oplax})^{\lax} & \simeq \FUN(\tB, \FUN(\tA, \tC)^{\lax})^{\oplax},\\
      \FUN(\tA, \FUN(\tB, \tC)^{\lax})^{\oplax} & \simeq \FUN(\tB, \FUN(\tA, \tC)^{\oplax})^{\lax}.
    \end{split}
  \]
\end{lemma}
\begin{proof}
  We prove the last equivalence; the others are proved in the same
  way. For an \itcat{} $\tX$ we have natural equivalences 
  \[
    \begin{split}
      \Map(\tX, \FUN(\tA, \FUN(\tB, \tC)^{\lax})^{\oplax})
      & \simeq \Map(\tA \otimes \tX, \FUN(\tB, \tC)^{\lax}) \\
      & \simeq \Map((\tA \otimes \tX) \otimes \tB, \tC) \\
      & \simeq \Map(\tA \otimes (\tX \otimes \tB), \tC) \\
      & \simeq \Map(\tX \otimes \tB, \FUN(\tA, \tC)^{\oplax}) \\
      & \simeq \Map(\tX, \FUN(\tB, \FUN(\tA, \tC)^{\oplax})^{\lax}). \\     \end{split}
  \]
  By the Yoneda lemma, this implies the required natural equivalence of \itcats{}.
\end{proof}

\begin{observation}\label{obs:grayop}
  While the Gray tensor product is not
  symmetric, we do have a natural equivalence
  \[ (\tA \otimes \tB)^{\op} \simeq \tB^{\op} \otimes \tA^{\op};\]
  this is evident in scaled simplicial sets \cite{GagnaHarpazLanariGray}*{Remark 2.4}. As a consequence, we have
 a natural equivalence
 \[ \FUN(\tA^{\op},\tB^{\op})^{\oplax} \simeq (\FUN(\tA,\tB)^{\lax})^{\op}.\]
 There is also a corresponding equivalence when reversing 2-morphisms, that is
 \begin{equation}
   \label{eq:grayco}
   \begin{split}
   (\tA \otimes \tB)^{\co} & \simeq \tB^{\co} \otimes \tA^{\co}, \\
   \FUN(\tA^{\co},\tB^{\co})^{\oplax} & \simeq
   (\FUN(\tA,\tB)^{\lax})^{\co}.
   \end{split}
 \end{equation}
 This is easy to see for some definitions of the Gray tensor, but not
 for our preferred definition in scaled simplicial sets (since the
 operation $(\blank)^{\co}$ cannot be defined easily in this
 model). A comparison between the Gray tensor product of scaled
 simplicial sets and the Gray tensor product constructed in
 \cite{CampionMaehara} (which clearly satisfies \cref{eq:grayco}) will
 appear in forthcoming work of the first author. In this document we
 will therefore freely use \cref{eq:grayco} when needed.
\end{observation}



\subsection{Cells and orientals}\label{sec:pushouts}
In this section we compute some pushouts that relate the 2-cell
$C_{2}$ to another type of basic building blocks of \itcats{}, namely
Street's ``orientals'' (or oriented simplices) and their boundaries in
low dimensions. The latter fit much better with the description of
\itcats{} as scaled simplicial sets, so these pushouts will be useful
for understanding various Gray tensor products later on.

\begin{notation}
  We introduce notation for the following \itcats{}, which will appear
  frequently in this and the next few sections:
  \[
    \begin{split}
      \partial [1] & :=  \{0,1\} \simeq [0] \amalg [0] \simeq [1](\emptyset) \\
      \partial C_{2} & :=
      \begin{tikzcd}[ampersand replacement=\&]
        0 \ar[r, bend left] \ar[r, bend right] \& 1
      \end{tikzcd}
      \simeq [1] \amalg_{\partial [1]} [1] \simeq [1](\partial[1]), \\
\partial C_{3} & :=
    \begin{tikzcd}[column sep=large,ampersand replacement=\&]
      0 \ar[r, bend left=50, ""{below,name=A,inner sep=2pt}]
      \ar[r, bend right=50, ""{above,name=B,inner sep=2pt}] \& 1
      \ar[from=A, to=B, bend right=50, shift right, Rightarrow, ""{below,name=C,
        inner sep=2pt}]      
      \ar[from=A, to=B, bend left=50, shift left, Rightarrow, ""{above,name=D,
        inner sep=2pt}]
    \end{tikzcd}
                 \!\! \simeq  C_{2} \amalg_{\partial C_{2}} C_{2}
                 \simeq [1](\partial C_{2}).
    \end{split}
  \]
  Whenever the notation makes sense, we also write $\partial(\tA
  \otimes \tB) \to \tA \otimes \tB$ for the functor
  \[ \partial \tA \otimes \tB \amalg_{\partial \tA \otimes \partial
      \tB} \tA \otimes \partial \tB \to \tA \otimes \tB,\]
  \ie{} the pushout-product of $\partial \tA \to \tA$ and $\partial
  \tB \to \tB$ under $\otimes$.
\end{notation}

\begin{remark}\label{rem:modelsforC}
 We have models for $C_2$, $\partial C_2$ and $\partial C_3$ in marked
 simplicial categories given by
 $\mathfrak{C}_2=[1](\Delta^{1,\flat})$, $\partial
 \mathfrak{C}_2=[1](\partial \Delta^{1,\flat})$ $\partial
 \mathfrak{C}_3=[1]((\Delta^1\amalg_{\partial
   \Delta^1}\Delta^1)^{\flat})$, respectively, where the left Quillen functor $[1](-)$ was discussed in \cref{obs:free2cat}.
\end{remark}

\begin{defn}\label{defn:partialoriental}
  The $n$th (2-dimensional) \emph{oriental} is the \itcat{} $\tO^{n}
  := [n]_{\lax}$, corresponding to the scaled simplicial set $\Delta^{n,\flat}$. We define its boundary $\partial \tO^{n}$ as the \itcat{} modeled by the scaled simplicial set $(\partial \Delta^{n})^{\flat}$, or equivalently as the colimit
  of the functor
  \[
   \partial\simp_{/n}\to \CatIT, \enspace (f \colon [k] \to [n]) \mapsto \tO^k,
 \]
 where $\partial\simp_{/n}$ denotes the full subcategory of $\simp_{/[n]}$ that consists of injective maps $f \colon [k] \to [n]$ such that $k\neq n$. (Note that by \cref{defn:rigidification} we have $\Csc[\Delta^{n,\flat}] = \mathfrak{O}^{n}$, so that $\mathfrak{O}^{n}$ is a model for $\tO^{n}$ in marked simplicial categories.)
\end{defn}

\begin{observation}\label{rem:partialo3}
  By \cref{defn:rigidification}, the marked simplicial category
  $\Csc[\partial \Delta^{3,\flat}]$ is a model for $\partial
  \tO^{3}$. Direct inspection reveals that  $\Csc[\partial
  \Delta^{3,\flat}](i,j) =\mathfrak{O}^3(i,j)$ if $i \neq 0$ or $j \neq
  3$ and that $\Csc[\partial \Delta^{3,\flat}](0,3)=(\Lambda^2_1
  \amalg_{\partial \Delta^1}\Lambda^2_1)^{\flat}$. We can therefore
  define a marked simplicial category $\partial\mathfrak{O}^{3}$ by
  the pushout square
  \[
     \begin{tikzcd}
       {[1]}((\Lambda^2_1 \amalg_{\partial \Delta^1}\Lambda^2_1)^{\flat}) \arrow[r] \arrow[d] &  {[1]}((\Delta^2 \amalg_{\partial \Delta^1} \Delta^2)^{\flat}) \arrow[d] \\
       \Csc[(\partial \Delta^3)^{\flat}] \arrow[r] & \partial \mathfrak{O}^3.
     \end{tikzcd}
   \] 
   Here the top horizontal map is a trivial cofibration by
   \cref{obs:free2cat}, hence so is the bottom map. We conclude that
   $\partial\mathfrak{O}^3$ is a cofibrant and fibrant marked
   simplicial category that models $\partial \tO^{3}$.
\end{observation}

\begin{lemma}\label{lem:cofibrationc3}
  The map  $\partial \mathfrak{C}_3 \to  \partial \mathfrak{O}^3$ that selects the two parallel 1-simplices in $\partial \mathfrak{O}^3(0,3)$ is a cofibration of marked simplicial categories.
\end{lemma}
\begin{proof}
We factor the map as $\partial \mathfrak{C}_3 \to A \to \partial \mathfrak{O}^3$ where the first map is a cofibration and the second map is a trivial fibration. It will therefore suffice to show that we can solve the lifting problem
  \[
    \begin{tikzcd}
     \partial \mathfrak{C}_3 \arrow[r,"\phi"] \arrow[d] & A \arrow[d] \\
     \partial \mathfrak{O}^3 \arrow[r] & \partial \mathfrak{O}^3,
    \end{tikzcd}
  \]
  since in that case, the map $\partial \mathfrak{C}_3 \to \partial
  \mathfrak{O}^3$ will be a retract of the cofibration $\partial
  \mathfrak{C}_3 \to A$. Let $x,y \in A$ be the pair of objects selected by $\phi$ in the diagram above.  We start by producing a solution to the lifting problem
   \[
    \begin{tikzcd}
     {[}0{]} \amalg {[}0{]} \arrow[d] \arrow[r] & A \arrow[d] \\
     \Csc[(\partial \Delta^3)^{\flat}] \arrow[ur,dotted,"\psi"] \arrow[r] & \partial \mathfrak{O}^3.
    \end{tikzcd}
  \]
  This exists since the left-most vertical map is a cofibration; to
  see this note that this map is $\Csc$ applied to the monomorphism
  of scaled simplicial sets that selects the initial and terminal objects in $(\partial \Delta^3)^{\flat}$. Next, let us consider another lifting problem,
  \[
    \begin{tikzcd}
       (\partial \Delta^2 \amalg_{\partial \Delta^1}\partial \Delta^2)^{\flat} \arrow[r] \arrow[d] & A(x,y) \arrow[d] \\
    (\Delta^2 \amalg_{\partial \Delta^1} \Delta^2)^{\flat} \arrow[r] \arrow[ur,"\zeta",dotted] & (\Delta^2 \amalg_{\partial \Delta^1} \Delta^2)^{\flat}
    \end{tikzcd}
  \]
  where the top horizontal morphism is obtained by gluing $\phi$ with
  the map
  \[ \Csc[(\partial \Delta^3)^{\flat}](0,3) \to A(x,y)\] from
  $\psi$. The left-most vertical map is clearly a cofibration and the
  right-most is a trivial fibration since the map
  $A \to \partial \mathfrak{O}^{3}$ is one, so the dotted arrow
  exists. We can now use the pushout in \cref{rem:partialo3} to
  amalgamate $\zeta$ and $\Psi$ and produce the desired solution to
  the lifting problem.
\end{proof}

\begin{propn}\label{propn:orientalpo}
  There are pushouts of \itcats{}
  \[
    \begin{tikzcd}
      \partial C_2 \arrow[r] \arrow[d] & \partial \tO^2 \arrow[d]  \\
      C_2 \arrow[r] & \tO^2,
    \end{tikzcd}
    \qquad
    \begin{tikzcd}
      \partial C_3 \arrow[r] \arrow[d] & \partial \tO^3 \arrow[d]  \\
      C_2 \arrow[r] & \tO^3.
    \end{tikzcd}
  \]
\end{propn}
\begin{remark}
  Informally, this says:
  \begin{itemize}
  \item The \itcat{} $\tO^{2}$ can be built by inserting a 2-morphism in $\partial \tO^{2}$.
  \item The oriental $\tO^{3}$ can be built by identifying two (composite) 2-morphisms in $\partial \tO^{3}$.
  \end{itemize}
\end{remark}

\begin{proof}
  We will exhibit these as strict homotopy pushouts of marked
  simplicial categories. First, we observe that for
  $\mathfrak{X} \in \sCatm$, it is immediate from the definition of
  $\mathfrak{O}^{2}$ that $\Hom(\mathfrak{O}^{2}, \mathfrak{X})$ can
  be described as the set whose elements are triples of morphisms
  $(f_{01} \colon x_{0} \to x_{1}, f_{12} \colon x_{1} \to x_{2},
  f_{02} \colon x_{0} \to x_{2})$ in $\mathfrak{X}$ together with a
  1-simplex from $f_{12} \circ f_{01}$ to $f_{02}$ in
  $\mathfrak{X}(x_{0},x_{2})$. We therefore have a pullback square of
  sets
  \[
    \begin{tikzcd}
      \Hom(\mathfrak{O}^{2}, \mathfrak{X}) \ar[r] \ar[d] & \Hom(\Csc[(\partial \Delta^{2})^{\flat}], \mathfrak{X}) \ar[d]  \\
      \Hom(\mathfrak{C}_{2}, \mathfrak{X}) \ar[r] & \Hom(\partial \mathfrak{C}_{2}, \mathfrak{X}),
    \end{tikzcd}
  \]
  and so a pushout square of marked simplicial categories
  \[
    \begin{tikzcd}
      \partial \mathfrak{C}_{2} \ar[r] \ar[d] & \mathfrak{C}_{2} \ar[d]  \\
      \Csc[(\partial \Delta^{2})^{\flat}] \ar[r] & \mathfrak{O}^{2},
    \end{tikzcd}
  \]
  which is a homotopy pushout square since the top horizontal maps is a cofibration and all objects are cofibrant.

   Now, we look at the case of $\tO^{3}$. Given $\mathfrak{X} \in
   \sCatm$, the set $\Hom(\mathfrak{O}^{3}, \mathfrak{X})$ can be
   described as the set of tuples of 1-morphisms $(f_{01}\colon x_0
   \to x_1, f_{12}\colon x_1 \to x_3, f_{23}\colon x_2 \to x_3,
   f_{02}\colon x_0 \to x_2, f_{03}\colon x_0 \to x_3, f_{13}: x_1 \to
   x_3)$ in $\mathfrak{X}$ together with 1-simplices of the form
   $f_{ij} \to f_{k\ell} \circ f_{sk}$ whenever $i,j \in
   \{s,k,\ell\}$, and a commutative diagram $\Delta^{1,\flat} \times \Delta^{1,\flat} \to \mathfrak{X}(x_0,x_3)$ relating the obvious composites
 \[
    \begin{tikzcd}
    f_{03} \arrow[r] \arrow[d] & f_{13} \circ f_{01 } \arrow[d]  \\
     f_{23} \circ f_{02 } \arrow[r] &  f_{23} \circ f_{12 } \circ f_{01}.
  \end{tikzcd}
 \]
 We can similarly describe $\Hom(\partial \mathfrak{O}^{3}, \mathfrak{X})$ by replacing the data of the final commutative square in $\mathfrak{X}(x_0,x_3)$ with a map $(\Delta^2 \amalg_{\partial \Delta^1} \Delta^2)^{\flat} \to \mathfrak{X}(x_0,x_3)$.

 There is an obvious map $\partial \mathfrak{C}_{3} \to \partial \mathfrak{O}^3$ selecting the pair of parallel morphisms in $\mathfrak{O}^3(0,3)$ which induces a pullback square
  \[
    \begin{tikzcd}
      \Hom(\mathfrak{O}^{3}, \mathfrak{X}) \ar[r] \ar[d] & \Hom(\partial \mathfrak{O}^{3}, \mathfrak{X}) \ar[d]  \\
      \Hom(\mathfrak{C}_{2}, \mathfrak{X}) \ar[r] & \Hom(\partial \mathfrak{C}_{3}, \mathfrak{X});
    \end{tikzcd}
  \]
  this in turn yields a pushout square of marked simplicial categories,
   \[
    \begin{tikzcd}
      \partial\mathfrak{C}_{3} \ar[r] \ar[d] & \mathfrak{C}_{2} \ar[d]  \\
      \partial \mathfrak{O}^{3} \ar[r] & \mathfrak{O}^{3}.
    \end{tikzcd}
  \]
  The left-most vertical morphism is a cofibration by
  \cref{lem:cofibrationc3} and since all objects are cofibrant it
  follows that our diagram is a homotopy pushout, as desired.
\end{proof}

\begin{defn}\label{defn:partialvarphigray}
  We have the following maps of marked simplicial categories:
  \begin{itemize}
    \item The map $\phi_2^{i}: \mathfrak{O}^{2} \to \mathfrak{C}_2$ for $i=0,1$, which identifies the objects $i,i+1$ and induces the identity map $\mathfrak{O}^2(0,2) \to \mathfrak{C}_2(0,1)$. The restriction of $\phi_2^{i}$ to $\partial \mathfrak{O}^2$ factors through $\partial \mathfrak{C}_2$ and will, consequently, be denoted by $\partial \phi_2^{i}$.

    \item The map $\partial \phi_3^{i}: \partial \mathfrak{O}^{3} \to \partial \mathfrak{C}_3$ for $i=0,1$, which identifies the objects $i,i+1,i+2$ and whose action on mapping marked simplicial sets
    \[
      \partial \mathfrak{O}^3(0,3)= (\Delta^2 \amalg_{\partial \Delta^1}\Delta^2)^{\flat} \to (\Delta^1 \amalg_{\partial \Delta^1}\Delta^1)^{\flat}=\partial \mathfrak{C}_3(0,1)
    \]
    is given by the map $s_i \amalg_{\partial \Delta^1} s_{1-i}$ where $s_j:\Delta^2 \to \Delta^1$ is the usual degeneracy map.
    \item The map $\phi_3^{i}:  \mathfrak{O}^{3} \to  \mathfrak{C}_2$ for $i=0,1$,  which identifies the objects $i,i+1,i+2$, and whose action on mapping marked simplicial sets 
     \[
      \mathfrak{O}^3(0,3)= (\Delta^1 \times \Delta^1)^{\flat} \to \Delta^{1,\flat}= \mathfrak{C}_{2}(0,1)
    \]
     is given by the map $s_{i} \amalg_{\Delta^1} s_{1-i}$ where $s_j:\Delta^2 \to \Delta^1$ denotes the usual degeneracy map. 
  \end{itemize}
\end{defn}

\begin{cor}\label{cor:C2poO}
  There are pushout squares of \itcats{}
  \[
    \begin{tikzcd}
       \partial \tO^2 \arrow[d] \arrow[r,"\partial \phi_2^{i}"] & \partial C_2 \arrow[d] \\
       \tO^2 \arrow[r,"\phi_2^{i}"] & C_2,
     \end{tikzcd}
     \qquad
    \begin{tikzcd}
       \partial \tO^3 \arrow[d] \arrow[r,"\partial \phi_3^{i}"] & \partial C_3 \arrow[d] \\
       \tO^3 \arrow[r,"\phi_3^{i}"] & C_2,
    \end{tikzcd}
  \]
  in $\CatIT$, for $i=0,1$.
\end{cor}
\begin{proof}
  It is clear from \cref{defn:partialvarphigray} that we have  commutative diagrams
  \[
    \begin{tikzcd}
      \partial C_j \arrow[r] \ar[d] &   \partial \tO^j \arrow[d] \arrow[r,"\partial \phi_j^i"] & \partial C_j \arrow[d] \\
      C_{2} \ar[r] & \tO^j \arrow[r,"\phi_j^{i}"] & C_2
    \end{tikzcd}
  \]
  where the left-hand square is a pushout by \cref{propn:orientalpo}. Since the horizontal composites are identities, it follows that the right-hand squares are also pushouts.
\end{proof}

\subsection{Oplax arrow \itcats{}}\label{sec:aropl}
In this section we take a look at the oplax arrow \itcat{} of an
\itcat{}, in the following sense:

\begin{notation}\label{not:slices}
  We abbreviate $\ARplax(\tA) := \FUN([1],\tA)^{\plax}$ and call this
  the \emph{(op)lax arrow \itcat{}} of $\tA$. For $a \in \tA$ we
  define the (op)lax slices $\tA^{\plax}_{a/}$ and $\tA^{\plax}_{/a}$ as
  the fibres at $a$ of evaluation at $0$ and $1$ on $\ARplax(\tA)$,
  respectively. We will denote by $\tA_{a/}$ and $\tA_{/a}$ fibres of of evaluation at $0$ and $1$ on $\FUN([1],\tA)$ and call them the \emph{strong} slices.
\end{notation}

In particular, we will identify the fibres of the projection
\[ \ARopl(\tA) \to \tA \times \tA.\]
We start by computing some more pushouts of \itcats{}:

\begin{lemma}\label{lem:pushoutquarec2}
  We have a pushout square   
   \[
    \begin{tikzcd}
      C_{2} \ar[d]\ar[r] & {[1]} \ar[d] \\
      {[1] \otimes [1]} \ar[r] & {[1] \times [1]}.
    \end{tikzcd}
  \]
\end{lemma}
\begin{proof}
  Consider the pushout square of scaled simplicial sets
   \[
    \begin{tikzcd}
      \Delta^{2,\flat} \ar[d]\ar[r] & {\Delta^{2,\sharp}} \ar[d] \\
      {\Delta^{1}\otimes \Delta^{1}} \ar[r] & {\Delta^{1} \times
        \Delta^{1}},
    \end{tikzcd}
  \]
  where we obtain the cartesian product by marking the 2-simplex that
  is not marked in the Gray tensor product. Since the top horizontal
  map is a cofibration (\ie{} a
  monomorphism), this gives a pushout of \itcats{}
  \[
    \begin{tikzcd}
      \tO^{2} \ar[r] \ar[d] & {[2]} \ar[d]  \\
      {[1]\otimes [1]} \ar[r] & {[1]\times [1]}.
    \end{tikzcd}
  \]
  Now consider the diagram
  \[
    \begin{tikzcd}
      \partial C_{2} \ar[r] \ar[d] & \partial \tO^{2} \ar[d] \\
      C_{2} \ar[r] \ar[d] & \tO^{2} \ar[d] \\
      {[1]} \ar[r] & {[2]}.
    \end{tikzcd}
  \]
  Here the top square is a pushout by \cref{propn:orientalpo} and it
  follows from the definitions of $\partial C_{2}$ and
  $\partial\tO^{2}$ as colimits that the outer square is a
  pushout. The bottom square is therefore also a pushout, and then the
  square we want is a composition of this with the one above.
\end{proof}

\begin{lemma}\label{lem:cons on 2mor}
  The following are equivalent for a functor of 
  \itcats{} $F \colon \tC \to \tD$:
  \begin{enumerate}[(1)]
  \item $F$ is conservative on $2$-morphisms, \ie{} right orthogonal to
    $C_{2} \to [1]$ (see \cref{def:rightorth}).
  \item The commutative square
    \[
      \begin{tikzcd}
       \tC^{\leq 1} \ar[r] \ar[d, "F^{\leq 1}"'] & \tC  \ar[d, "F"] \\
       \tD^{\leq 1} \ar[r] & \tD
      \end{tikzcd}
    \]
    is a pullback.
  \item For all $x,y \in \tC$, the functor of \icats{}
    $\tC(x,y) \to \tD(Fx,Fy)$ is conservative, \ie{} right orthogonal
    to $[1] \to [0]$.
  \end{enumerate}
\end{lemma}
\begin{proof}
  By definition, $F$ is conservative on $2$-morphisms if the
  commutative square of \igpds{}
  \[
    \begin{tikzcd}
     \Map([1], \tC) \ar[r] \ar[d] & \Map(C_{2}, \tC)  \ar[d] \\
     \Map([1], \tD) \ar[r] & \Map(C_{2}, \tD)
    \end{tikzcd}
  \]
  is a pullback. Since the localization of $C_{2}$ to an \icat{} is
  $[1]$, we have a natural equivalence $\Map([1], \blank) \simeq
  \Map(C_{2}, (\blank)^{\leq 1})$, so this square is naturally
  equivalent to
  \[
    \begin{tikzcd}
     \Map(C_{2}, \tC^{\leq 1}) \ar[r] \ar[d] & \Map(C_{2}, \tC)  \ar[d] \\
     \Map(C_{2}, \tD^{\leq 1}) \ar[r] & \Map(C_{2}, \tD),
    \end{tikzcd}
  \]
  which is a pullback \IFF{} the square in (2) is a pullback, since
  $C_{2}$ detects equivalences of \itcats{} (as the other cells $[0]$
  and $[1]$ are retracts of $C_{2}$). On the other hand, the
  equivalence of (1) and (3) follows from \cref{obs:free2cat ortho}.
\end{proof}

\begin{propn}\label{propn:ARoplfibis1cat}\ 
  \begin{enumerate}[(i)]
  \item There is a pushout square of \itcats{}
    \[
      \begin{tikzcd}
        C_{2}\times\partial [1] \ar[r] \ar[d] & C_{2}\otimes [1] \ar[d]  \\
        {[1]\times \partial [1]} \ar[r] & {[1]\otimes [1]}
      \end{tikzcd}
    \]
  \item For any \itcat{} $\tA$, the functor $\FUN([n], \tA)^{\oplax}
    \to \tA^{\times (n+1)}$ is conservative on $2$-morphisms.
  \item For any \itcat{} $\tA$, the commutative square
    \[
      \begin{tikzcd}
        (\FUN([n], \tA)^{\oplax})^{\leq 1} \ar[r] \ar[d] & \FUN([n], \tA)^{\oplax} \ar[d, ]  \\
        (\tA^{\leq 1})^{\times (n+1)} \ar[r] & \tA^{\times (n+1)}
      \end{tikzcd}
    \]
    is a pullback.
  \item The fibres of $\ARopl(\tA) \to \tA \times \tA$ are \icats{}.
  \end{enumerate}
\end{propn}
\begin{proof}
  In scaled simplicial sets we have a pushout square
  \[
    \begin{tikzcd}
      \Delta^{2,\flat}\times \partial \Delta^{1} \ar[r] \ar[d] & \Delta^{2,\flat}\otimes \Delta^{1} \ar[d]  \\
      \Delta^{2,\sharp}\times \partial \Delta^{1} \ar[r] & \Delta^{2,\sharp}\otimes \Delta^{1},
    \end{tikzcd}
  \]
  since the vertical maps are identities on underlying simplicial sets
  and the scaling of $\Delta^{2,\sharp}\otimes \Delta^{1}$ consists of
  the scaled 2-simplices in $\Delta^{2,\flat}\otimes \Delta^{1}$
  together with the two 2-simplices in the image of
  $\Delta^{2} \times \partial \Delta^{1}$. Since the vertical
  morphisms are cofibrations (\ie{} monomorphisms), this is a homotopy
  pushout, and so gives a pushout of \itcats{} of the form
  \[
    \begin{tikzcd}
      \tO^{2}\times \partial [1] \ar[r] \ar[d] & \tO^{2} \otimes [1] \ar[d]  \\
      {[2]\times \partial [1]} \ar[r] & {[2]\otimes [1]}.
    \end{tikzcd}
  \]
  From the proof of \cref{lem:pushoutquarec2} we have a commutative diagram
  \[
    \begin{tikzcd}
        C_{2} \ar[r] \ar[d] & \tO^{2} \ar[d] \ar[r] & C_{2} \ar[d] \\
      {[1]} \ar[r] & {[2]} \ar[r] & {[1]}
    \end{tikzcd}
  \]
  where the left square is a pushout and the horizontal composites are
  identities. The right square is therefore a pushout, and combining
  this with our first square we get (since both $\times$ and $\otimes$
  preserve colimits in each variable) a commutative cube from which we
  can extract our desired pushout. This proves (i), and the case
  $n = 1$ of (ii) follows by rewriting the lifting property we ask for
  in terms of the Gray tensor product; the general case follows from
  this since right orthogonality is closed under limits. Now (iii)
  follows from \cref{lem:cons on 2mor}, and this implies (iv) as the
  fibres in question are also fibres of a functor of \icats{}.
\end{proof}

\begin{propn}\label{propn:2foldsegalfromoplax}
  Let $\tA$ be an \itcat{}. Then the corresponding 2-fold Segal space $\tA^{\Seg}_{\bullet}$, viewed as a simplicial \icat{}, is given by the pullback
  \[
    \begin{tikzcd}
      \tA^{\Seg}_{\bullet} \arrow{r} \arrow{d} & \FUN([\bullet], \tA)^{\oplax} \arrow{d} \\
      i_{*}\tA^{\simeq} \arrow{r} & i_{*} \tA,
    \end{tikzcd}
  \]
  where $i_{*}$ denotes right Kan extension along the inclusion $\{[0]\} \to \Dop$ (so $(i_{*}\tA)_{n-1} \simeq \tA^{\times n}$) and the right vertical map is the unit for the adjunction $i^{*} \dashv i_{*}$.
\end{propn}
\begin{proof}
  By definition (see \cref{obs:2foldseg}), we have \[\Map([m], \tA^{\Seg}_{n}) \simeq \Map_{\CatIT}(\tau([n],[m]), \tA),\] and from \cite{HHLN1}*{Lemma 5.1.11}, we know that there is a natural pushout square
  \[
    \begin{tikzcd}
      \coprod_{i=0}^{n} \{i\} \otimes [m] \arrow{r} \arrow{d} & {[n] \otimes [m]} \arrow{d} \\
      \coprod_{i=0}^{n} * \arrow{r} & \tau([n],[m]),
    \end{tikzcd}
  \]
  so that we have pullback squares
  \[
    \begin{tikzcd}
      \Map([m], \tA^{\Seg}_{n}) \arrow{r} \arrow{d} & \Map([n] \otimes [m], \tA) \ar[d] \\
      \prod_{i=0}^{n} \tA^{\simeq} \arrow{r} & \prod_{i=0}^{n} \Map([m], \tA)
    \end{tikzcd}
  \]
  or
  \[
    \begin{tikzcd}
      \Map([m], \tA^{\Seg}_{n}) \arrow{r} \arrow{d} & \Map([m], \FUN([n], \tA)^{\oplax}) \ar[d] \\
      \Map([m], \prod_{i=0}^{n} \tA^{\simeq}) \arrow{r} & \Map([m], \prod_{i=0}^{n} \tA).
    \end{tikzcd}
  \]
  This shows that in the commutative diagram
  \[
    \begin{tikzcd}
      \tA^{\Seg}_{n} \ar[r] \ar[d] & (\FUN([n],\tA)^{\oplax})^{\leq 1} \ar[d] \ar[r] & \FUN([n], \tA)^{\oplax} \ar[d] \\
      \prod_{i=0}^{n} \tA^{\simeq} \ar[r] & \prod_{i=0}^{n} \tA^{\leq 1} \ar[r] & \prod_{i=0}^{n} \tA,
    \end{tikzcd}
  \]
  the left-hand square is a pullback. Since the right-hand square is a pullback by \cref{propn:ARoplfibis1cat}, so is the composite square, as required.
\end{proof}

\begin{cor}\label{cor:aroplfibres}
  For objects $x, y$ in an \itcat{} $\tA$, the mapping \icat{} $\tA(x,y)$ is the fibre $\ARopl(\tA)_{x,y}$. Moreover, for a functor $F \colon \tA \to \tB$, the map on fibres over $(x,y)$ in the square
  \[
    \begin{tikzcd}
      \ARopl(\tA) \ar[r, "\ARopl(F)"] \ar[d] & \ARopl(\tB) \ar[d]  \\
      \tA\times\tA \ar[r, "F\times F"'] & \tB\times\tB
    \end{tikzcd}
  \]
  is the induced functor $\tA(x,y) \to \tB(Fx,Fy)$. 
\end{cor}
\begin{proof}
  If we think of \itcats{} as $2$-fold Segal spaces, then we can
  define the mapping \icat{} $\tA(x,y)$ as the fibre of
  $\tA^{\Seg}_{1} \to (\tA^{\simeq})^{\times 2}$ at $(x,y)$, and the
  map $\tA(x,y) \to \tB(Fx,Fy)$ as the map on fibres in the square
  \[
    \begin{tikzcd}
      \tA^{\Seg}_{1} \ar[r, "F^{\Seg}_{1}"] \ar[d] & \tB^{\Seg}_{1} \ar[d]  \\
      \tA^{\simeq}\times\tA^{\simeq} \ar[r, "F^{\simeq}\times F^{\simeq}"'] & \tB^{\simeq}\times\tB^{\simeq};
    \end{tikzcd}
  \]
  this is known to agree with mapping \icats{} as defined in other
  models.
  By \cref{propn:2foldsegalfromoplax} taking fibres in this square gives the same thing as taking fibres in oplax arrows.  
\end{proof}

\subsection{Sub-$(\infty,2)$-categories}\label{sec:subcat}
In this section we will look at various types of sub-objects of
\itcats{}. We will characterize these by right lifting properties and
describe the relations between them. For the \itcatl{} definitions, we
first need some terminology for sub-objects of \icats{}:
\begin{defn}
    A functor $F \colon \oC \to \oD$ of \icats{} is:
  \begin{itemize}
  \item \emph{faithful} if $\oC(c,c') \to \oD(Fc,Fc')$ is a
    monomorphism of \igpds{} for all $c,c' \in \oC$;
  \item a \emph{(subcategory) inclusion} if it is faithful 
    and  $\oC^{\simeq} \to \oD^{\simeq}$ is a monomorphism --- we also say that $\oC$ is a \emph{subcategory}\footnote{Note that for ordinary categories this corresponds to the notion of \emph{replete} subcategory in much of the category theory literature, rather than the more traditional definition of subcategory, which is not invariant under equivalences of categories.} of $\oD$.
  \end{itemize}
\end{defn}

\begin{observation}\label{obs:monopb}
 A morphism $f \colon X \to Y$ in $\Spc$ is a
  \emph{monomorphism} \IFF{} the commutative square
  \[
    \begin{tikzcd}
      X \arrow{r} \arrow{d} & Y \arrow{d} \\
      X \times X \arrow{r} & Y \times Y
    \end{tikzcd}
  \]
  is cartesian --- this follows from \cite{HTT}*{Lemma 5.5.6.1} since
  the map from $X$ to the pullback is precisely the diagonal $X \to X
  \times_{Y} X$ for $f$; this condition is equivalent to $X$ being a
  union of path-components in $Y$. In general, we then say that a
  \emph{monomorphism} in an \icat{} $\oC$ is a morphism
  $\phi \colon x \to y$ such that $\oC(c, x) \to \oC(c,y)$ is a monomorphism of
  \igpds{} for all $c \in \oC$. If $\oC$ is an \icat{} with binary products,
  it follows that $\phi$ is a monomorphism \IFF{}
    \[
    \begin{tikzcd}
      x \arrow{r} \arrow{d} & y \arrow{d} \\
      x \times x \arrow{r} & y \times y
    \end{tikzcd}
  \]
  is a pullback square in $\oC$.
\end{observation}

\begin{defn}
  A functor $F \colon \tC \to \tD$ of \itcats{} is called
  \begin{itemize}
  \item \emph{fully faithful} (or a \emph{full (sub-\itcat{})
      inclusion}) if
  $\tC(c,c') \to \tD(Fc,Fc')$ is an equivalence of \icats{} for all
  objects $c,c' \in \tC$ --- we then also say that $\tC$ is a \emph{full
    sub-\itcat{}} of $\tD$;
  \item \emph{locally fully faithful} if $\tC(c,c') \to \tD(Fc,Fc')$
    is fully faithful for all $c,c' \in \tC$;
  \item a \emph{locally full (sub-\itcat{}) inclusion} if $F$ is locally fully faithful and
    $\tC^{\simeq} \to \tD^{\simeq}$ is a monomorphism --- we then also say
    that $\tC$ is a \emph{locally full sub-\itcat{}} of $\tD$;
  \item \emph{locally faithful} if $\tC(c,c') \to \tD(Fc,Fc')$ is
    faithful for all $c,c' \in \tC$;
  \item \emph{locally an inclusion} if $\tC(c,c') \to \tD(Fc,Fc')$ is
    a subcategory inclusion for all $c,c' \in \tC$;
  \item a \emph{(sub-\itcat{}) inclusion} if it is locally an
    inclusion and $\tC^{\simeq} \to \tD^{\simeq}$ is a monomorphism
    --- we then also say that $\tC$ is a \emph{sub-\itcat{}} of $\tD$.
  \end{itemize}
  We also say that a sub-\itcat{} $\tC$ of $\tD$ is \emph{wide} if $\tC^{\simeq} \to \tD^{\simeq}$ is an equivalence.
\end{defn}

\begin{observation}
  A fully faithful functor of \icats{} is also a subcategory inclusion
  (that is, it is also a monomorphism on underlying \igpds{}); see for
  instance \cite{enrtens}*{Corollary 10.5}. The same is therefore true for a fully
  faithful functor of \itcats{}, since its underlying functor of
  \icats{} is also fully faithful. Similarly, we see that a locally
  fully faithful functor is locally an inclusion, and a locally full
  inclusion is an inclusion.
\end{observation}

We will characterize these classes of functors by lifting properties,
in the following \icatl{} sense:
\begin{defn}\label{def:rightorth}
  Given morphisms $f \colon x \to y$ and $g \colon z \to w$ in an \icat{} $\oC$, we say
  that $g$ is \emph{right orthogonal} to $f$ (or that $f$ is
  \emph{left orthogonal} to $g$) if for every commutative
  square
  \[
    \begin{tikzcd}
      x \arrow{r} \arrow{d}[swap]{f} & z \arrow{d}{g} \\
      y \arrow{r} \arrow[dotted]{ur} & w
    \end{tikzcd}
  \]
  there is a unique lift $y \to z$, \ie{} the space of such lifts is contractible. This condition says that the fibres of
  \[ \oC(y,z) \to \oC(x,z) \times_{\oC(x,w)} \oC(y,w) \]
  are contractible, which is equivalent to this map being an equivalence, or the commutative square
    \[
    \begin{tikzcd}
      \oC(y,z) \arrow{r}{g_{*}} \arrow{d}[swap]{f^{*}} & \oC(y,w) \arrow{d}{f^{*}} \\
      \oC(x,z) \arrow{r}[swap]{g_{*}}  &  \oC(x,w)
    \end{tikzcd}
  \]
  being a pullback.
\end{defn}

\begin{observation}\label{obs:free2cat ortho}
  From the adjunction in \cref{obs:free2cat} it follows that a functor
  of \itcats{} $F \colon \tC \to
  \tD$ is right orthogonal to a map of the form \[[1](f) \colon
  [1](\oA) \to [1](\oB)\] for a functor $f \colon \oA \to \oB$ of
  \icats{} \IFF{} for all $x,y \in \tC$ the functor $\tC(x,y) \to
  \tD(Fx,Fy)$ is right orthogonal to $f$. 
\end{observation}

\begin{observation}\label{obs:ortho for 1cat}
  From \cref{obs:monopb} we see that a functor $F$ of \icats{} is a
  monomorphism on underlying \igpds{}
  \IFF{} it is right orthogonal to \[
    \partial [1] \to [0] \quad = \quad 
    [0] \amalg [0]  \to
  [0].\] Using \cref{obs:free2cat ortho} (or rather its analogue for \icats{}
  and \igpds{}) we also see that
  \begin{itemize}
  \item $F$ is faithful \IFF{} it is right
  orthogonal to \[\partial C_{2}  \to [1] \quad = \quad [1]([0] \amalg
    [0]) \to [1]([0]);\]
\item $F$ is an inclusion \IFF{} it is right orthogonal to $\partial
  C_{2} \to [1]$ and $\partial [1] \to [0]$;
\item $F$ is fully faithful \IFF{} it is right
  orthogonal to \[
    \partial [1] \to [1] \quad = \quad [1](\emptyset) \to [1](*),\] since a
  morphism of \igpds{} is  an equivalence
  \IFF{} it is right orthogonal to $\emptyset \to *$;
\item $F$ is an equivalence \IFF{} it is right orthogonal to
  $\partial [1] \to [1]$ (fully faithful) and $\emptyset \to [0]$
  (equivalence on underlying \igpds{}).
  \end{itemize}
\end{observation}

\begin{observation}\label{obs:subcats ortho by adjn}  
  Using \cref{obs:free2cat} and \cref{obs:ortho for 1cat}, we get the
  following for a functor $F$ of \itcats{}:
  \begin{itemize}
  \item $F$ is fully faithful \IFF{} it is right orthogonal to
    \[\partial C_{2} \to C_{2} \quad=\quad [1](\partial [1]) \to [1]([1])\] and
    \[\partial [1] \to [1] \quad=\quad [1](\emptyset) \to [1]([0]).\]
  \item $F$ is locally fully faithful \IFF{} it is right orthogonal to
     $\partial C_{2} \to C_{2}$.
  \item $F$ is a locally full inclusion \IFF{} it is right orthogonal
    to $\partial C_{2} \to C_{2}$ and $\partial [1] \to [0]$.
  \item $F$ is locally faithful \IFF{} it is right orthogonal to
    \[\partial C_{3} \to C_{2}   \quad=\quad  [1](\partial C_{2}) \to [1]([1]).\]
  \item $F$ is locally an inclusion \IFF{} it is right orthogonal to
    $\partial C_{3} \to C_{2}$ and $\partial C_{2} \to [1]$.
  \item $F$ is an inclusion \IFF{} it is right orthogonal to $\partial
    [1] \to [0]$, $\partial C_{2} \to [1]$ and $\partial C_{3} \to
    C_{2}$.
  \item $F$ is a wide inclusion \IFF{} it is right orthogonal to
    $\emptyset \to [0]$, $\partial C_{2} \to [1]$ and
    $\partial C_{3} \to C_{2}$.
  \end{itemize}
\end{observation}

\begin{remark}
  The classes of fully faithful, locally fully faithful, and locally
  faithful functors of \itcats{} are also studied in
  \cite{Soergel}*{\S 5.3} where they appear as the \emph{$j$-faithful}
  functors of \itcats{} for $j = -1,0,1$, respectively, as part of a
  general study of $j$-faithful functors of $(\infty,n)$-categories
  for $j \geq -2$. In particular, the characterizations of these
  functors given in \cref{obs:subcats ortho by adjn} are part of
  \cite{Soergel}*{Theorem 5.3.7}, which also identifies the morphisms
  that are left orthogonal to the $j$-faithful functors as precisely
  the \emph{$j$-surjective} functors. These are the morphisms that are
  (essentially) surjective on $i$-morphisms for $0 \leq i \leq j+1$
  (where $0$-morphisms mean objects), so this identification also
  immediately implies \cref{cor:ffonspcs} and \cref{cor:lffon1cats}.
\end{remark}

We now want to consider some alternative characterizations of fully
faithful functors and inclusions.

\begin{propn}\label{propn:fffcond}
  For a functor of \itcats{} $F \colon \tC \to \tD$, the following are
  equivalent:
  \begin{enumerate}
  \item $F$ is fully faithful.
  \item\label{it:ffbdry} $F$ is right orthogonal to $\partial [1] \to [1]$ and
    $\partial C_{2} \to C_{2}$.
  \item\label{it:ffcoprod} $F$ is right orthogonal to $\partial [1] \to [1]$ and
    $[1] \amalg [1] \to [1] \otimes [1]$.
  \item\label{it:ffro0+0} $F$ is right orthogonal to $\partial [1] \to [1]$ and
    $\partial [1] \to C_{2}$.
  \end{enumerate}
\end{propn}

\begin{remark}
  The characterization of fully faithful functors in condition
  \cref{it:ffro0+0} of \cref{propn:fffcond} is a special case of
  \cite{Loubaton}*{Proposition 4.2.1.64}, which characterizes fully
  faithful functors of $(\infty,\infty)$-categories.
\end{remark}

\begin{observation}
  Recall that a class of maps $S$ in a presentable \icat{} $\oC$ is called \emph{saturated} if
  \begin{itemize}
  \item $S$ is closed under colimits in $\Ar(\oC)$,
  \item $S$ contains all equivalences and is closed under composition,
  \item $S$ is closed under cobase change.
  \end{itemize}
  The \emph{saturation} of $S$ is the smallest saturated class that
  contains $S$. The class of maps that are left orthogonal to a given
  set of maps is always saturated, by \cite{HTT}*{Proposition
    5.2.8.6}. This means that if a map $g$ is right orthogonal to
  a set of maps $T$ then $g$ is also right orthogonal to the
  saturation of $T$ (since the set of maps that are left orthogonal to
  $g$ is saturated and contains $T$).
\end{observation}

For the proof of \cref{propn:fffcond} we need the following:
\begin{lemma}\label{lem:bdry1otimes1}\ 
\begin{enumerate}[(i)]
\item The three maps $\partial([1]
  \otimes [1]) \to [1] \otimes [1]$, $\partial \tO^{2} \to \tO^{2}$, and $\partial C_{2} \to C_{2}$
  all generate the same saturated class.
\item $\partial C_{3} \to C_{2}$ is in the
  saturation of $\partial C_{2} \to C_{2}$.
\item The four maps
  $\partial ([1] \otimes C_{2}) \to [1] \otimes C_{2}$,
  $\partial ([1] \otimes \tO^{2}) \to [1] \otimes \tO^{2}$,
  $\partial \tO^{3} \to \tO^{3}$, and $\partial C_{3} \to C_{2}$ all
  generate the same saturated class.
\end{enumerate}
\end{lemma}
\begin{proof}
  In this proof we will freely use the description of the Gray tensor
  product in terms of scaled simplicial sets (\cref{defn:ghlgray}).
  
  For both (i) and (iii) it is immediate from \cref{propn:orientalpo}
  and \cref{cor:C2poO} that $\partial \tO^{i} \to \tO^{i}$ and
  $\partial C_{i} \to C_{2}$ generate the same saturated class, for
  $i = 2,3$. To show (i) it therefore suffices to relate the morphisms of scaled simplicial sets
  $(\partial \Delta^2)^{\flat} \to \Delta^{2,\flat}$ and
    \[ \partial (\Delta^1 \otimes \Delta^1)=\Lambda^2_1
  \amalg_{\{0\}\amalg \{2\}} \Lambda^2_1 \to \Delta^1 \otimes
  \Delta^1. \] The latter factors as a trivial cofibration and a pushout of 
  $(\partial \Delta^2)^{\flat} \to \Delta^{2,\flat}$. Conversely, the
  map $(\partial \Delta^2)^{\flat} \to \Delta^{2,\flat}$ can be obtained
  from
  $\partial (\Delta^1 \otimes \Delta^1) \to \Delta^1 \otimes \Delta^1$
  by collapsing the unique thin 2-simplex to the edge
  $(0,0) \to (1,1)$.

  To show (ii), we note that the map $\partial C_3 \to C_2$ admits a
  section $C_2 \to \partial C_3$ which is in the saturation of
  $\partial C_{2} \to C_{2}$. Thus in the composition
  \[ C_{2} \to \partial C_{3} \to C_{2}\]
  the composite and the first map lie in this saturated class, which
  means that so does the second map by \cite{HTT}*{Corollary 5.5.5.8}.

  Let us finally show (iii). We first observe that
  $\partial ([1] \otimes C_{2}) \to [1] \otimes C_{2}$ and
  $\partial ([1] \otimes \tO^{2}) \to [1] \otimes \tO^{2}$ generate
  the same saturated class. To see this, we use
  \cref{propn:ARoplrortho} to see that a map $F \colon \tC \to \tD$ is
  right orthogonal to the latter \IFF{} the induced map
  $\ARopl(\tC) \to \ARopl(\tD) \times_{\tD^{\times 2}} \tC^{\times 2}$
  is right orthogonal to $\partial \tO^{2} \to \tO^{2}$. By (i), this
  is equivalent to right orthogonality with respect to
  $\partial C_{2} \to C_{2}$, and applying \cref{propn:ARoplrortho}
  again this is the same as the functor $F$ being right orthogonal
  with respect to
  $\partial ([1] \otimes C_{2}) \to [1] \otimes C_{2}$.
  It therefore only remains to compare the two maps $\partial ([1] \otimes \tO^{2}) \to [1] \otimes \tO^{2}$ and
  $\partial \tO^{3} \to \tO^{3}$. We can model these in scaled simplicial sets as 
 \[ \partial ([1] \otimes \Delta^{2,\flat}) = (\{0,1\} \times \Delta^{2,\flat}) \amalg_{\{0,1\} \times \partial \Delta^{2,\flat}} \Delta^{1} \otimes \partial \Delta^{2,\flat} \to [1] \otimes \Delta^{2,\flat} \] and
 \[ \partial \Delta^{3,\flat} \to \Delta^{3,\flat}.\]
 It is easy to show that we have a factorization $\partial ([1] \otimes \Delta^{2,\flat}) \to Q \to  [1] \otimes \Delta^{2,\flat}$, where $Q$ contains every simplex in $ [1] \otimes \Delta^{2,\flat}$ except for the 3-simplex $ \sigma\colon(0,0) \to (0,1) \to (0,2) \to (1,2)$. The first map is an anodyne morphism of scaled simplicial sets, since it can be obtained by first adding the simplex $(0,0) \to (1,0) \to (1,1) \to (1,2)$ by taking a pushout along the anodyne morphism $(\Lambda^3_1,\Delta^{\{0,1,2\}})  \to (\Delta^3,\Delta^{\{0,1,2\}})$ and then adding the simplex $(0,0) \to (0,1) \to (1,1) \to (1,2)$ by taking a pushout along the anodyne morphism $(\Lambda^3_2,\Delta^{\{1,2,3\}})  \to (\Delta^3,\Delta^{\{1,2,3\}})$. Note that the final map sits in a pushout square
 \[
   \begin{tikzcd}
     \partial \Delta^{3,\flat} \arrow[r] \arrow[d] &  Q  \arrow[d] \\
     \Delta^{3,\flat} \arrow[r] &  {[}1{]} \otimes \Delta^{2,\flat}.
   \end{tikzcd}
 \]
 There is a map $[1] \otimes \Delta^{2,\flat} \to \Delta^{3,\flat}$
 that collapses everything onto the simplex $\sigma$. It follows that
 we can extend the previous pushout square to a commutative diagram
 \[
   \begin{tikzcd}
     \partial \Delta^{3,\flat} \arrow[r] \arrow[d] &  Q  \arrow[d]  \arrow[r] & \partial \Delta^{3,\flat} \arrow[d] \\
     \Delta^{3,\flat} \arrow[r] &  {[}1{]}\otimes \Delta^{2,\flat} \arrow[r] & \Delta^{3,\flat}
   \end{tikzcd}
 \]
 where the horizontal composites are identities. We conclude that the right-most square is also a pushout, and the result follows.
\end{proof}

\begin{proof}[Proof of \cref{propn:fffcond}]
  We already saw in \cref{obs:subcats ortho by adjn} that $F$ is fully
  faithful \IFF{} \cref{it:ffbdry} holds.  Let us name the maps we're considering as
  \[ p \colon \partial [1] \to [1], \quad q \colon [1] \amalg [1]
    \to [1] \otimes [1], \quad r \colon \partial C_{2} \to C_{2},
    \quad s \colon \partial [1] \to C_{2}.\]
  To see that \cref{it:ffbdry} is equivalent to \cref{it:ffcoprod}, we
  first observe that the map $r \colon \partial C_{2} \to C_{2}$ is a
  pushout $\id_{[0]} \amalg_{p} q \amalg_{p} \id_{[0]}$ and so lies in
  the saturation of $p$ and $q$. Conversely, we can factor $q$ as
  \[ [1] \amalg [1] \to \partial([1] \otimes [1]) \to [1] \otimes
    [1] \]
  where the first map is in the saturation of $p$ and the second is in
  the saturation of $r$ by 
  \cref{lem:bdry1otimes1}(i).

  To see that \cref{it:ffbdry} is equivalent to \cref{it:ffro0+0}, we
  consider the factorization of $s$ as
  \[ \partial [1] \to \partial C_{2} \to C_{2}.\] Here the first map
  is a pushout of two copies of $p$, so a functor that is right
  orthogonal to $p$ is right orthogonal to $r$ \IFF{} it is right
  orthogonal to $s$.
\end{proof}

\begin{cor}\label{cor:ffonspcs}
  A functor of \itcats{} $F \colon \tA \to \tB$ is fully faithful
  \IFF{} for all \itcats{} $\tX$, the commutative square
  \[
    \begin{tikzcd}
      \Map(\tX,\tA) \ar[r, "F_{*}"] \ar[d] & \Map(\tX,\tB) \ar[d] \\
\Map(\tX^{\simeq},\tA) \ar[r, "F_{*}"] & \Map(\tX^{\simeq},\tB)
    \end{tikzcd}
  \]
  is a pullback.
\end{cor}
\begin{proof}
  The given condition in the cases $\tX = [1], C_{2}$ is equivalent to
  \cref{propn:fffcond}(4), so if it holds then $F$ must be fully
  faithful. For the converse, we first observe that for $J \in
  \Theta_{2}$, the colimit decomposition of $J$ that corresponds to
  the Segal condition is preserved by $(\blank)^{\simeq}$ (which in
  this case amounts to getting colimits on the \emph{sets} of
  objects). Since limits commute, the statement for $\tX = J$
  therefore follows from that for $C_{2}, [1]$ and $[0]$. For a
  general \itcat{} $\tX$, we then observe that the canonical (Yoneda) colimit
  cocone $(\Theta_{2/\tX})^{\triangleright} \to \PSh(\Theta_{2})$ is
  preserved by the localization to $\CatIT$ (since the colimit is
  already an \itcat{}). Applying $(\blank)^{\simeq}$ to this colimit
  diagram then amounts to evaluating at $[0] \in \Theta_{2}$, so this
  produces a colimit of spaces. Using again that limits commute, it
  follows that the statement for $\tX$ follows from that for objects
  of $\Theta_{2}$.
\end{proof}

By the same argument as for \cref{cor:ffonspcs}, we get:
\begin{cor}\label{cor:lffon1cats}
  A functor of \itcats{} $F \colon \tA \to \tB$ is locally fully
  faithful \IFF{} for all \itcats{} $\tX$, the commutative square
  \[
    \begin{tikzcd}
      \Map(\tX,\tA) \ar[r, "F_{*}"] \ar[d] & \Map(\tX,\tB) \ar[d] \\
\Map(\tX^{\leq 1},\tA) \ar[r, "F_{*}"] & \Map(\tX^{\leq 1},\tB)
    \end{tikzcd}
  \]
  is a pullback.\qed
\end{cor}

We now give some alternative characterizations of inclusions:
\begin{propn}\label{propn:inclcond}
  The following are equivalent for a functor $F \colon \tC \to \tD$ of \itcats{}:
  \begin{enumerate}[(1)]
  \item $F$ is an inclusion.
  \item $F$ is right orthogonal to $\partial [1] \to [0]$, $\partial
    C_{2} \to [1]$ and $\partial C_{3} \to C_{2}$.
  \item $F$ is right orthogonal to 
    $C_{2} \amalg C_{2} \to C_{2}$.
  \item $F$ is a monomorphism in $\CatIT$.
  \item The commutative square
    \begin{equation}
      \label{eq:tcatmonosq}
      \begin{tikzcd}
        \tC \arrow{r} \arrow{d} & \tD \arrow{d} \\
        \tC \times \tC \arrow{r} & \tD \times \tD
      \end{tikzcd}
    \end{equation}
    is a pullback.
  \end{enumerate}
\end{propn}
\begin{proof}
  We already saw that the first two conditions are equivalent in \cref{obs:subcats ortho by adjn}.
  Moreover, the last two
  conditions are equivalent by \cref{obs:monopb}; moreover, since
  $C_{2}$ detects equivalences in $\CatIT$ (as $[0]$ and $[1]$ are
  retracts of $C_{2}$ and we may view $\CatIT$ as a full subcategory
  of presheaves on $\Theta_{2}$ where the Segal conditions decompose
  all objects into colimits of these three) they are equivalent to the
  third condition.

  It therefore only remains to compare (2) and (3). On the one hand,
  from the colimits defining $\partial C_{2}$ and $\partial C_{3}$ we
  see that the maps $\partial C_{2} \to [1]$ and $\partial C_{3} \to
  C_{2}$ are in the saturation of the map from (3) together with its
  retracts $[i] \amalg [i] \to [i]$ for $i =0,1$. On the other
  hand, we have a factorization
  \[ C_{2} \amalg C_{2} \xto{j} \partial C_{3} \to C_{2},\] where $j$
  is in the saturation of $\partial [1] \to [0]$ and
  $\partial C_{2} \to [1]$.
\end{proof}

\subsection{Sub-$(\infty,2)$-categories and (op)lax transformations}\label{sec:sublax}
In this section, we show that all of our classes of sub-\itcats{} are
preserved by functors of the form $\FUN(\tA,\blank)^{\plax}$. The key
to the proof is the following observation, from which we will derive
new descriptions of our subobject inclusions:

\begin{propn}\label{propn:ARoplrortho}
  For a functor $F \colon \tC \to \tD$, let
  $\Phi \colon \ARopl(\tC) \to \ARopl(\tD) \times_{\tD^{\times 2}}
  \tC^{\times 2}$ be the map to the pullback in the square
  \begin{equation}
    \label{eq:aroplsq}
    \begin{tikzcd}
      \ARopl(\tC) \arrow{d} \arrow{r} & \ARopl(\tD) \arrow{d} \\
      \tC \times \tC \arrow{r} & \tD \times \tD.
    \end{tikzcd}
  \end{equation}
  Then $\Phi$ is right orthogonal to a map $f \colon \tK \to
  \tL$ \IFF{} $F$ is right orthogonal to
  \[ [1] \otimes \tK \amalg_{\partial [1] \otimes \tK} \partial [1]
    \otimes \tL \to [1] \otimes \tL. \]
\end{propn}
\begin{proof}
  The given condition on $F$ is that it is right orthogonal to the
  pushout-product of $\partial [1] \to [1]$ and $f$ with respect to
  the Gray tensor product. This is the Day convolution monoidal
  structure on $\Ar(\CatIT)$ coming from the ``minimum'' monoidal
  structure on $[1]$ together with the Gray tensor product. The result
  therefore follows from the formula for the internal Hom of Day
  convolution, as in \cite{NikolausYoneda}*{Proposition 3.11}.
\end{proof}

\begin{cor}\label{cor:Phidesc}
  Consider a functor $F \colon \tC \to \tD$ and let
  $\Phi$ be defined as in \cref{propn:ARoplrortho}.
  Then:
  \begin{enumerate}
  \item $F$ is fully faithful \IFF{} $\Phi$ is an equivalence.    
  \item $F$ is locally fully faithful \IFF{} $\Phi$ is fully faithful.
  \item $F$ is locally faithful \IFF{} $\Phi$ is locally fully faithful.
  \item $F$ is locally an inclusion \IFF{} $\Phi$ is a locally full
    inclusion.
  \end{enumerate}
\end{cor}
\begin{proof}
  $\Phi$ is an equivalence \IFF{} it is right orthogonal to
  $\emptyset \to [0]$, $\emptyset \to [1]$ and $\emptyset \to
  C_{2}$. Via \cref{propn:ARoplrortho} this corresponds to $F$ being
  right orthogonal to $[0] \amalg [0] \to [1]$, $[1] \amalg [1] \to
  [1] \otimes [1]$ and $C_{2} \amalg C_{2} \to C_{2} \otimes [1]$. The
  first two correspond to $F$ being fully faithful by
  \cref{propn:fffcond}, so it remains to show that a fully faithful
  functor is also right orthogonal to the third map. To see this we decompose
  this map as the composite
  \[ C_{2} \amalg C_{2} \to \partial([1] \otimes C_{2}) \to [1]
    \otimes C_{2}.\]
  Here the first map is a cobase change of $\partial C_{2} \amalg
  \partial C_{2} \to [1] \otimes \partial C_{2}$, which lies in the
  saturation of $[1] \amalg [1] \to [1] \otimes [1]$, and the second
  is in the saturation of $\partial C_{2} \to C_{2}$ by
  \cref{lem:bdry1otimes1}(ii) and (iii).

  Using \cref{propn:ARoplrortho} and \cref{obs:subcats ortho by adjn},
  we see that $\Phi$ being fully faithful is equivalent to $F$ being
  right orthogonal with respect to the maps
  \[ \partial ([1] \otimes [1]) \to
    [1] \otimes [1], \qquad \partial ([1] \otimes C_{2}) \to
    [1] \otimes C_{2}.\]
  We know from \cref{lem:bdry1otimes1} that the first map generates
  the same saturated class as $\partial C_{2} \to C_{2}$ and that the
  second is in the saturation of this. This condition is therefore
  equivalent to $F$ being locally fully faithful by \cref{obs:subcats
    ortho by adjn}.
  
  Similarly, using \cref{obs:subcats
    ortho by adjn} we see that $\Phi$ is locally
  fully faithful \IFF{} $F$ is
  right orthogonal to
  \[ \partial ([1] \otimes C_{2}) \to [1] \otimes C_{2}.\] This is the
  same as $F$ being right orthogonal to $\partial C_{3} \to C_{2}$ by
  \cref{lem:bdry1otimes1}, which corresponds to $F$ being locally
  faithful by \cref{obs:subcats
    ortho by adjn}.
  
  Finally, using \cref{obs:subcats
    ortho by adjn} once more we have that $\Phi$ is a
  locally full inclusion \IFF{} $F$ is right orthogonal to the maps
  \[ \partial ([1] \otimes C_{2}) \to [1] \otimes C_{2}\]
  and
  \[ ([1]\amalg [1]) \amalg_{[0]^{\amalg 4}} ([0] \amalg [0]) \to
    [1].\] Here the second is precisely the map
  $\partial C_{2} \to [1]$, and we just saw that $F$ is right
  orthogonal to the first \IFF{} it is so with respect to
  $\partial C_{3} \to C_{2}$. By \cref{obs:subcats ortho by adjn}
  these conditions hold precisely when $F$ is locally an inclusion.
\end{proof}

\begin{thm}\label{thm:funlaxprops}
  Suppose $F \colon \tC \to \tD$ is either fully faithful, or locally
  fully faithful, or locally faithful, or locally an inclusion, or a locally full inclusion, or an inclusion. Then so is
  \[  \FUN(\tA, \tC)^{\plax} \xto{F_{*}} \FUN(\tA, \tD)^{\plax} \]
  for all \itcats{} $\tA$.
\end{thm}
\begin{proof}
  We first consider the lax case; then have a natural equivalence
  \[ \ARopl(\FUN(\tA, \tB)^{\lax}) \simeq
    \FUN(\tA, \ARopl(\tB))^{\lax}\]
  by \cref{lem:funandgray}. Since $\FUN(\tA, \blank)^{\lax}$
  preserves limits (being a right adjoint), this means that the square
  \[
    \begin{tikzcd}
      \ARopl(\FUN(\tA, \tC)^{\lax}) \arrow{d} \arrow{r} &
      \ARopl(\FUN(\tA, \tD)^{\lax}) \arrow{d} \\
      \FUN(\tA, \tC)^{\lax,\times 2}  \arrow{r} & \FUN(\tA,
      \tD)^{\lax,\times 2}
    \end{tikzcd}
  \]
  can be identified with $\FUN(\tA, \blank)^{\lax}$ applied to the
  square \cref{eq:aroplsq} for $F$. Moreover, if we let
  $\tX := \ARopl(\tD) \times_{\tD^{\times 2}} \tC^{\times 2}$ and
  write $\Phi \colon \ARopl(\tC) \to \tX$ for the map to the pullback
  in the latter square, then the map to the pullback in the square for
  $F_{*}$ is
  \[ \Phi_{*} \colon \FUN(\tA, \ARopl(\tC))^{\lax} \to
    \FUN(\tA, \tX)^{\lax}.\]
  We can then observe in turn that:
  \begin{itemize}
  \item By \cref{cor:Phidesc}(1), $F$ is fully faithful \IFF{} $\Phi$
    is an equivalence; then $\Phi_{*}$ is an equivalence and so
    $F_{*}$ is fully faithful.
  \item By \cref{cor:Phidesc}(2), $F$ is locally fully faithful \IFF{}
    $\Phi$ is fully faithful; by the preceding point this implies that
    $\Phi_{*}$ is fully faithful, and hence $F_{*}$ is locally fully
    faithful.
  \item By \cref{cor:Phidesc}(3), $F$ is locally faithful \IFF{}
    $\Phi$ is locally fully faithful; by the preceding point this implies that
    $\Phi_{*}$ is locally fully faithful, and hence $F_{*}$ is locally
    faithful.
  \end{itemize}
  If we consider instead $\FUN(\tA,\blank)^{\lax}$ applied to the
  square \cref{eq:tcatmonosq} we similarly get from
  \cref{propn:inclcond} that if $F$ is an inclusion then so is
  $F_{*}$. 

  The case where $F$ is a locally full inclusion follows by combining
  the above statements for locally fully faithful functors and
  inclusions. The only remaining case is where $F$ is locally an
  inclusion, where \cref{cor:Phidesc} shows that this is equivalent to
  $\Phi$ being a locally full inclusion. This implies that $\Phi_{*}$
  is a locally full inclusion, and hence $F_{*}$ is locally an
  inclusion too.

  The oplax case follows using the natural equivalences of the form
  \[ \FUN(\tA,
    \tB)^{\oplax} \simeq (\FUN(\tA^{\op}, \tB^{\op})^{\lax})^{\op}, \]
  and noting that each of the classes of maps we consider contains a
  functor $F$ \IFF{} it contains $F^{\op}$.
\end{proof}

\subsection{Partially lax transformations}\label{sec:partiallax}
In this section we consider natural transformations that are only
\emph{partially} (op)lax, meaning that a specified class of (op)lax
naturality squares must in fact commute. We show that these define
locally full sub-\itcats{} of $\FUN(\blank,\blank)^{\plax}$ and extend 
\cref{thm:funlaxprops} to these sub-\itcats{}. We will give the
definition in terms of the following \emph{marked} version of the Gray
tensor product:

\begin{defn}
  A \emph{marked \itcat{}} is a pair $(\tc{A}, E)$ where $\tc{A}$ is
  an \itcat{} and $E$ is a collection of 1-morphisms in $\tc{A}$ that
  contains the equivalences. We write $E^{\natural}$ for the
  collection of morphisms of $\tc{A}$ consisting precisely of the
  equivalences and $E^{\sharp}$ for the collection of \emph{all}
  morphisms in $\tA$.
\end{defn}

\begin{defn}\label{defn:markedgraytens}
  Suppose $(\tc{A}, E)$ and $(\tc{B},F)$ are marked \itcats{}. Let $S$ be the set of
  commutative triangles in $\tc{A} \times \tc{B}$ of the form
  \[\left(\begin{tikzcd}
        & {a_1} &&& {b_1} \\
        {a_0} && {a_2,} &  {b_0} && {b_2}
        \arrow["{f_{01}}", from=2-1, to=1-2]
        \arrow["{f_{12}}", from=1-2, to=2-3]
        \arrow["{f_{02}}"', from=2-1, to=2-3]
        \arrow["{g_{01}}", from=2-4, to=1-5]
        \arrow["{g_{02}}"', from=2-4, to=2-6]
        \arrow["{g_{12}}", from=1-5, to=2-6]
      \end{tikzcd}\right)\]
  where either $f_{12}$ lies in $E$ or $g_{01}$ lies in $F$. The
  \emph{marked Gray
    tensor product} $\tc{A} \otimes_{E,F} \tc{B}$ is then the \itcat{}
  $(\tc{A} \times \tc{B})_{S}$ associated to the scaled \itcat{}
  $(\tc{A} \times \tc{B}, S)$. We abbreviate
  \[ \tc{A} \otimes_{\natural,F} \tc{B} := \tc{A}
    \otimes_{E^{\natural},F} \tc{B}, \qquad \tc{A} \otimes_{\sharp,F}
    \tc{B} := \tc{A} \otimes_{E^{\sharp},F} \tc{B},\]
  etc. Note that here $\tc{A} \otimes_{\natural,\natural} \tB$ is the
  ordinary Gray tensor product, while for \emph{any} markings we
  recover the cartesian product as
  \[  \tc{A} \otimes_{\sharp,F} \tB \simeq \tA \times \tB \simeq \tA
    \otimes_{E,\sharp} \tB. \]
\end{defn}

\begin{defn}
  Let $(\tA, E)$ be a marked \itcat{}. Then an \emph{$E$-oplax natural
  transformation} is a functor $\tA \otimes_{E,\natural} [1] \to \tB$,
while an \emph{$E$-lax natural transformation} is a functor $[1]
\otimes_{\natural,E} \tA \to \tB$.
\end{defn}

\begin{lemma}\label{lem:markedgraypo}
  For marked \itcats{} $(\tA,E)$ and $(\tB,F)$, we have a pushout
  \[
    \begin{tikzcd}
      \coprod_{S} [1] \otimes [1] \ar[r] \ar[d] & \coprod_{S} [1]
      \times [1] \ar[d] \\
      \tA \otimes \tB \ar[r] & \tA \otimes_{E,F} \tB,
    \end{tikzcd}
  \]
  where $S$ is the set of equivalence classes of pairs of maps
  $f \colon [1] \to \tA$, $g \colon [1] \to \tB$ such that either $f$
  lies in $E$ or $g$ lies in $F$, and neither $f$ nor $g$ is
  invertible.
\end{lemma}
\begin{proof}
  We fix fibrant scaled simplicial sets 
  $\sA=(A,T_A)$ and $\sB=(B,T_B)$ modelling $\tA$ and $\tB$
  respectively.

  Let $P$ be the scaled simplicial set obtained from $\sA \otimes \sB$
   by scaling those triangles in the image of the map
  $\coprod_{S}  \Delta^1 \otimes \Delta^1 \to \sA \otimes \sB$. Since $P$ is
  precisely the homotopy pushout of the diagram in the statement it
  will suffice to show that the canonical map
  $P \to \sA \otimes_{E,F} \sB$ is a weak equivalence.

  Suppose $\sigma \colon \Delta^2 \to \sA$ and
  $\rho \colon \Delta^2 \to \sB$ define a thin simplex $\omega$ in
  $\sA \otimes_{E,F} \sB$, so that either $d_{0}\sigma$ or $d_{2}\rho$
  is marked. We want to show that we can add this 2-simplex to the
  scaling in $P$ by taking a pushout of a trivial cofibration. We will
  explain this in the case where $d_{0}\sigma$ is marked;
  the proof of the other case is essentially the same.

  Consider the 3-simplex $ \Xi=s_2(\sigma) \otimes s_1(\rho)$ in
  $\sA \otimes_{E,F} \sB$ and observe that the only faces of $\Xi$
  which are not already thin in $P$ are $d_3(\Xi)$ and
  $d_2(\Xi)=\omega$. If we manage to scale $d_3(\Xi)$ using pushouts
  along trivial cofibrations the claim will follow
  from \cref{lem:sattriangles}.

  We define another 3-simplex $\Theta$ in $\sA \otimes_{E,F} \sB$ given by $s_1(\sigma)$ and $s_0s_{1}d_2(\rho)$. We make the following observations:
\begin{itemize}
  \item We have $d_1(\Theta)=d_3(\Xi)$.
  \item The face $d_0(\Theta)$ is scaled since it lies on the image of $\coprod_{S} \Delta^1 \otimes \Delta^1 \to \sA \otimes \sB$.
  \item The faces $d_2(\Theta)$ and $d_3(\Theta)$ are already thin in $\sA \otimes \sB$.
\end{itemize}
We conclude from \cref{lem:sattriangles} that we can scale $d_3(\Xi)$, and thus the result follows.
\end{proof}

\begin{observation}\label{obs:markedgraydecomp}
  Let $(\tA,E)$ be a marked \itcat{}. From \cref{lem:markedgraypo} we
  get for any \itcat{} $\tX$ pushout squares of the form
  \[
    \begin{tikzcd}
      \coprod [1] \otimes \tA \ar[r] \ar[d] &  \coprod [1]
      \otimes_{\natural,E}\tA \ar[d] \\
      \tX \otimes \tA \ar[r] & \tX \otimes_{\natural,E} \tA,
    \end{tikzcd}
    \quad
    \begin{tikzcd}
      \coprod \tA \otimes [1] \ar[r] \ar[d] & \coprod \tA
      \otimes_{E,\natural} [1] \ar[d] \\
      \tA \otimes \tX \ar[r] & \tA \otimes_{E,\natural} \tX,
    \end{tikzcd}
  \]
  \[
    \begin{tikzcd}
      \coprod_{E} [1] \otimes \tX \ar[r] \ar[d] &  \coprod_{E} [1]
      \times\tX \ar[d] \\
      \tA \otimes \tX \ar[r] & \tA \otimes_{E,\natural} \tX,
    \end{tikzcd}
    \quad
    \begin{tikzcd}
      \coprod_{E} \tX \otimes [1] \ar[r] \ar[d] & \coprod_{E} \tX
      \times [1] \ar[d] \\
      \tX \otimes \tA \ar[r] & \tX \otimes_{\natural,E} \tA,
    \end{tikzcd}
  \]
  using the 3-for-2 property of pushouts, where the first coproducts
  are over all equivalence classes of non-invertible morphisms in $\tX$.
\end{observation}

\begin{observation}\label{obs:graysqloc}
  The pushout of \cref{lem:pushoutquarec2} means
  that a functor $[1] \otimes [1] \to \tA$ factors through $[1] \times
  [1]$ precisely when the non-invertible 2-morphism in $[1] \otimes
  [1]$ is mapped to an invertible 2-morphism in $\tA$. Note that
  combined with \cref{lem:markedgraypo}, this means that the morphism
  \[ \tA \otimes \tB \to \tA \otimes_{E,F} \tB\]
  is in the saturated class of $C_{2} \to [1]$ for all marked
  \itcats{} $(\tA,E)$ and $(\tB, F)$.
\end{observation}

\begin{observation}\label{obs:Elax as lax with property}
  Let $(\tA, E)$ be a marked \itcat{}.  Combining
  \cref{lem:markedgraypo} with \cref{obs:graysqloc}, we see that an
  $E$-oplax transformation $\tA \otimes_{E,\natural} [1] \to \tB$ is
  precisely an oplax natural transformation
  $\tc{A} \otimes [1] \to \tc{B}$ such that the oplax square
  \[ [1] \otimes [1] \xto{f \otimes \id} \tc{A} \otimes [1] \to
    \tc{B} \] commutes for every morphism $f$ in $E$. Similarly, an
  $E$-lax transformation is a lax transformation such that the squares
  corresponding to morphisms in $E$ commute.
\end{observation}

\begin{propn}\label{propn:markedgraycolim}
  Let $(\tA, E)$ be a marked \itcat{}. Then the functors
  \[ \blank \otimes_{\natural,E} \tA,\, \tA \otimes_{E,\natural}
    \blank \colon \CatIT \to \CatIT \]
  preserve colimits.
\end{propn}
\begin{proof}
  This follows from Corollary 4.1.10 in \cite{GagnaHarpazLanariLaxLim}.
\end{proof}

\begin{construction}
  Since $\CatIT$ is a presentable \icat{}, it follows from
  \cref{propn:markedgraycolim} that for a marked \itcat{} $(\tA,E)$,
  the functors $\blank \otimes_{\natural,E} \tA$ and
  $\tA \otimes_{E,\natural} \blank$ have right adjoints. We denote
  these by $\FUN(\tA,\blank)^{\elax}$ and
  $\FUN(\tA,\blank)^{\eoplax}$, respectively, so that we have natural
  equivalences
  \[ \Map(\tX \otimes_{\natural,E} \tA, \tB) \simeq \Map(\tX,
    \FUN(\tA,\tB)^{\elax}),\]
   \[ \Map(\tA \otimes_{E,\natural} \tX, \tB) \simeq \Map(\tX,
    \FUN(\tA,\tB)^{\eoplax}).
  \]
  We write
  $\Nat^{\eplax}_{\tA,\tB}(F,G)$ for the mapping \icats{} in
  $\FUN(\tA,\tB)^{\eplax}$ between functors $F,G$.
\end{construction}

\begin{observation}\label{obs:factorsthroughElax}
  It follows from \cref{obs:markedgraydecomp} that a functor
  $\tX \to \FUN(\tA, \tB)^{\elax}$ is equivalently a functor
  $\tX \otimes \tA \to \tB$ such that for every morphism in $E$, the
  restriction $\tX \otimes [1] \to \tB$ factors through
  $\tX \times [1]$.
\end{observation}

\begin{propn}\label{propn:elaxlocfull}
  Suppose $(\tA,E)$ is a marked \itcat{}. Then
  $\FUN(\tA,\tB)^{\eplax}$ is a wide locally full sub-\itcat{} of
  $\FUN(\tA,\tB)^{\plax}$ for any \itcat{} $\tB$.
\end{propn}
\begin{proof}
  We prove the lax case; the oplax case is proved similarly. We have a
  natural transformation $(\blank) \otimes \tA \to (\blank)
  \otimes_{\natural,E} \tA$, which is adjoint to a natural functor
  \[ \FUN(\tA,\tB)^{\elax} \to \FUN(\tA,\tB)^{\lax}.\]
  This is an equivalence on the space of objects, since we have
  \[ [0] \otimes \tA \simeq [0] \otimes_{\natural,E} \tA \simeq \tA.\]
  To prove that this functor is also locally fully faithful, we use
  \cref{obs:subcats ortho by adjn} and show that it is right orthogonal to
  $\partial C_{2} \to C_{2}$. Unwinding the adjunctions, this amounts
  to the functor
  \[ (\partial C_{2} \otimes_{\natural,E} \tA) \amalg_{\partial C_{2}
      \otimes \tA} (C_{2} \otimes \tA) \to C_{2} \otimes_{\natural,E}
    \tA \]
  being an equivalence. This is true since it follows from
  \cref{obs:markedgraydecomp} that in the commutative diagram
  \[
    \begin{tikzcd}
     ([1] \amalg [1]) \otimes \tA \ar[r] \ar[d] & \partial C_{2}
     \otimes \tA \ar[r] \ar[d] & C_{2} \otimes \tA \ar[d] \\
     ([1] \amalg [1]) \otimes_{\natural,E} \tA \ar[r]  & \partial C_{2}
     \otimes_{\natural,E} \tA \ar[r] & C_{2} \otimes_{\natural,E} \tA,
    \end{tikzcd}
  \]
  both the left and outer squares are pushouts.
\end{proof}

As a special case, we have:
\begin{cor}\label{cor:stronglax}
  For \itcats{} $\tA,\tB$, the \itcat{} $\FUN(\tA,\tB)$ is equivalent
  to the wide locally
  full sub-\itcat{} of $\FUN(\tA,\tB)^{\plax}$ that contains the
  strong (op)lax transformations. \qed
\end{cor}

We now get the following extension of \cref{thm:funlaxprops}:

\begin{cor}\label{cor:elaxfunprops}
  Suppose $F \colon \tC \to \tD$ is either fully faithful, or locally
  fully faithful, or locally faithful, or locally an inclusion, or a locally full inclusion, or
  an inclusion. Then so is
  \[  \FUN(\tA, \tC)^{\eplax} \xto{F_{*}} \FUN(\tA, \tD)^{\eplax} \]
  for any marked \itcat{} $(\tA,E)$.
\end{cor}
\begin{proof}
  We prove the $E$-lax case; the oplax case is proved similarly.
  
  If $F$ is an inclusion, then it follows by applying
  $\FUN(\tA,\blank)^{\eplax}$ to the square \cref{eq:tcatmonosq} and
  the criterion of \cref{propn:inclcond}(4) that $F_{*}$ is also an
  inclusion.

  We also have a commutative square
  \[
    \begin{tikzcd}
      \FUN(\tA, \tC)^{\elax} \arrow{r}{F_{*}} \arrow{d} & \FUN(\tA, \tD)^{\elax} \arrow{d} \\
      \FUN(\tA,\tC)^{\lax} \arrow{r}{F^{\lax}_{*}} & \FUN(\tA, \tD)^{\lax},
    \end{tikzcd}
  \]
  where the vertical maps are locally fully faithful by
  \cref{propn:elaxlocfull} and we write $F_{*}^{\lax}$ for the
  composition functor on lax transformations for clarity.  Since they
  are characterized by right orthogonality properties, both locally
  faithful and locally fully faithful maps, as well as maps that are
  locally an inclusion, are the right classes in factorization systems
  on $\CatIT$, so it follows that if $F^{\lax}_{*}$ has either
  property then so must $F_{*}$. \cref{thm:funlaxprops} then implies
  that if $F$ is locally (fully) faithful or locally an inclusion then
  so is $F_{*}$. This also gives the case where $F$ is a locally full
  inclusion, since this by definition means that $F$ is locally fully
  faithful and an inclusion.

  We are left with the case where $F$ is fully faithful.
  Given $\phi,\psi \colon \tA \to \tC$ we have a commutative square of
  \icats{}
  \[
    \begin{tikzcd}
      \Nat^{\elax}_{\tA,\tC}(\phi,\psi)\arrow[hookrightarrow]{d}
      \arrow[hookrightarrow]{r} & \Nat^{\elax}_{\tA,\tD}(F\phi, F\psi)
      \arrow[hookrightarrow]{d} \\
      \Natlax_{\tA,\tC}(\phi,\psi) \arrow{r}{\sim} & \Natlax_{\tA,\tD}(F\phi, F\psi).
    \end{tikzcd}
  \]
  To see that the top horizontal map is also an equivalence, we need
  to check that if $\alpha \colon \phi \to \psi$ is a lax
  transformation such that $F\alpha$ is an $E$-lax transformation,
  then $\alpha$ must also be an $E$-lax transformation. Using
  \cref{obs:Elax as lax with property} we see that this follows from
  the fact that the fully faithful functor $F$ detects invertible
  2-morphisms. 
\end{proof}

As a special case, we get the corresponding statement for ordinary
functor \itcats{}:
\begin{cor}\label{cor:funprops}
  Suppose $F \colon \tC \to \tD$ is either fully faithful, or locally
  fully faithful, or locally faithful, or locally an inclusion, or a locally full inclusion, or
  an inclusion. Then so is
  \[  \FUN(\tA, \tC) \xto{F_{*}} \FUN(\tA, \tD) \]
  for any \itcat{} $\tA$. \qed
\end{cor}

\begin{remark}
  The fully faithful case of \cref{cor:funprops} is also a special
  case of \cite{enrtens}*{Corollary 10.8}.
\end{remark}



\subsection{The Yoneda embedding and lax transformations}\label{sec:laxyoneda}

If $\oA, \oB$ are \icats{}, then the Yoneda embedding
$\oB \hookrightarrow \Fun(\oB^{\op}, \Spc)$ induces a fully faithful
embedding
$\Fun(\oA, \oB) \hookrightarrow \Fun(\oA \times \oB^{\op}, \Spc)$. The
Yoneda lemma for \itcats{} is a special case of Hinich's enriched
Yoneda lemma \cite{HinichYoneda}; see also \cite{Ab23} for a
fibrational approach to the proof. Our goal in this section is to use
this to obtain a similar embedding
\[ \FUN(\tA,\tB)^{\lax} \hookrightarrow \FUN(\tA \times \tB^{\op},
  \CATI)^{\lax}\] for \itcats{} $\tA,\tB$, which we will make repeated
use of later on. The key input for this is a description of the relation between two combinations of Gray and cartesian products:
\begin{propn}\label{propn:tripleprodpo}
  For \itcats{} $\tA,\tB,\tC$, there are  natural pushout squares 
  \[
    \begin{tikzcd}
      \tA \otimes (\tB^{\simeq} \times \tC) \arrow{r} \arrow{d} & \tA
      \times \tB^{\simeq} \times \tC \arrow{d} \\
      \tA \otimes (\tB \times \tC) \arrow{r} & (\tA \otimes \tB) \times \tC,
    \end{tikzcd}\quad \quad
    \begin{tikzcd}
     (\tA \times \tB^{\simeq}) \otimes \tC \arrow{r} \arrow{d} & \tA
      \times \tB^{\simeq} \times \tC \arrow{d} \\
      (\tA \times \tB ) \otimes \tC \arrow{r} & \tA \times (\tB \otimes \tC).
    \end{tikzcd}
  \]
\end{propn}
\begin{proof}
  We will prove the case of the left-hand diagram; the other case
  follows by taking opposites and using \cref{obs:grayop}. For this we
  will use the description of the Gray tensor product in scaled
  simplicial sets. We thus consider scaled simplicial sets
  $\sA=(A,T_A),\sB=(B,T_B)$ and $\sC=(C,T_C)$ that represent our
  \itcats{} $\tA, \tB$, and $\tC$, and define a fourth scaled
  simplicial set $\sX=(A \times B \times C,T_X)$, where a triangle
  $\sigma$ belongs to $T_X$ if $\sigma$ is a thin triangle in
  $\sA \otimes (\sB \times \sC)$, or if its image in $\sB$ factors
  through the underlying \igpd{} of $\sB$. We then have a pushout square of scaled simplicial sets
  \[
    \begin{tikzcd}
      \sA \otimes (\sB^{\simeq} \times \sC) \arrow{r} \arrow{d} & \sA
      \times \sB^{\simeq} \times \sC \arrow{d} \\
      \sA \otimes (\sB \times \sC) \arrow{r} & \sX,
    \end{tikzcd}
  \]
  where every object is cofibrant and every morphism is a cofibration,
  so this is also a homotopy pushout. A fibrant replacement of $\sX$
  therefore represents the pushout in our square,
  and it suffices to show that we have a weak equivalence of scaled
  simplicial sets $\sX \xto{\sim} (\sA \otimes \sB) \times \sC$.

  Note that both sides have $A \times B \times C$ as their underlying
  simplicial set, so it suffices to compare the thin 2-simplices in
  these scaled simplicial sets.  For
  $\phi \colon \Delta^{2} \to A \times B \times C$ we denote by $\phi_{A}$,
  $\phi_{B}$, $\phi_{C}$ the images of $\phi$ in
  $A$, $B$, and $C$, respectively. Then $\phi$ is thin in $\sX$
  if:
  \begin{itemize}
     \item The simplices $\phi_{A}$, $\phi_{B}$ and $\phi_{C}$ are thin in $\sA,\sB$ and $\sC$.
     \item At least one the following conditions is satisfied:
     \begin{itemize}
       \item The edge $\phi_{A}(1 \to 2)$ is an equivalence in $\sA$.
       \item The edges $\phi_{B}(0 \to 1)$ and $\phi_{C}(0 \to 1)$
         are equivalences in $\sB$ and $\sC$, respectively.
       \item $\phi_{B}$ factors through the underlying \igpd{} of $\sB$.
     \end{itemize}
   \end{itemize} 
   Similarly, we see that $\phi$ is thin in $(\sA \otimes \sB) \times \sC$ \IFF{}:
   \begin{itemize}
     \item The simplices $\phi_{A}$, $\phi_{B}$ and $\phi_{C}$ are thin in $\sA,\sB$ and $\sC$.
     \item At least one the following conditions is satisfied:
     \begin{itemize}
       \item The edge $\phi_{A}(1 \to 2)$ is an equivalence in $\sA$.
       \item The edge $\phi_{B}(0\to 1)$ is an equivalence in $\sB$.
     \end{itemize}
   \end{itemize}
  Let
  $\phi \colon \Delta^{2} \to (\sA \otimes \sB) \times \sC$ be a thin 2-simplex
  which is not contained in $T_X$; then 
  $\phi_{B}(0 \to 1)$ must be an equivalence. We define a 3-simplex
  $\Phi \colon \Delta^{3} \to \sX$ by specifying its projections $\Phi_{A},\Phi_{B},\Phi_{C}$ to $A,B,C$:
 \begin{itemize}
   \item $\Phi_{A} = s_2(\phi_{\sA})$,
   \item $\Phi_{B} = s_1(\phi_{\sB})$,
   \item $\Phi_{C} = s_1(\phi_{\sC})$. 
 \end{itemize}
 We note that $d_3(\Phi_{B})$ belongs to $\sB^{\simeq}$ and that
 $d_{i}(\Phi)$ belongs to $T_X$ for $i=0,1$ since the edge
 $\Phi_{A}(2 \to 3)$ is an equivalence. Since $d_2(\Phi)=\phi$ it
 follows from \cref{lem:sattriangles} that we can scale this simplex via a pushout along a trivial
 cofibration. 
\end{proof}

\begin{observation}\label{obs:triple prod map is locn}
  Combining the pushouts in \cref{propn:tripleprodpo} with those from
  \cref{lem:markedgraypo} and \cref{lem:pushoutquarec2}, we see that
  the maps
  \[ \tA \otimes (\tB \times \tC) \to (\tA \otimes \tB) \times
    \tC, \quad
    (\tA \times \tB ) \otimes \tC \to \times (\tB \otimes \tC)
  \]
  are localizations that invert certain 2-morphisms.
\end{observation}

\begin{cor}\label{cor:thattripleproductlocn}
  A functor $F \colon \tX \otimes (\tY \times \tZ) \to \tA$ factors
  through the localization to
  $(\tX \otimes \tY) \times \tZ$ \IFF{} for every $y \in \tY$
  the restriction $F(\blank, y, \blank)$ factors through $\tX \times
  \tZ$, \ie{} if for every $x \to x'$ in $\tX$ and $z \to z' \in \tZ$
  the 2-morphism in the square
  \[
    \begin{tikzcd}
      F(x,y,z) \arrow{r} \arrow{d} & F(x,y,z') \arrow{d} \\
      F(x',y,z) \arrow{r} \arrow[Rightarrow]{ur} & F(x',y,z')
    \end{tikzcd}
  \]
  is invertible; equivalently, for every morphism $x \to x'$ in $\tX$
  the lax transformation $F(x,y,\blank) \to F(x',y,\blank)$ is
  strong. \qed
\end{cor}

\begin{cor}
  The \itcat{} $\FUN(\tA, \FUN(\tB, \tC))^{\lax}$ is a wide locally full
  sub-\itcat{} of $\FUN(\tA \times \tB, \tC)^{\lax}$ whose morphisms
  are the lax transformations \[\Phi \colon [1] \otimes (\tA \times \tB) \to \tC\]
  such that for every $a \in \tA$, the restriction $\Phi(\blank, a,
  \blank)$ is strong.
\end{cor}
\begin{proof}
  We have natural equivalences
  \[ \Map(\tX, \FUN(\tA, \FUN(\tB, \tC))^{\lax}) \simeq \Map((\tX
    \otimes \tA) \times \tB, \tC),\]
  \[ \Map(\tX, \FUN(\tA \times \tB, \tC)^{\lax}) \simeq \Map(\tX
    \otimes (\tA \times \tB), \tC).\]
  The claim is then a rephrasing of \cref{cor:thattripleproductlocn}.
\end{proof}

Combining this with \cref{thm:funlaxprops} and the Yoneda lemma, we get our desired inclusion:
\begin{cor}\label{funlaxemb}
  For \itcats{} $\tA, \tB$, the \itcat{} $\FUN(\tA, \tB)^{\lax}$ is a
  locally full sub-\itcat{} of $\FUN(\tA \times \tB^{\op},
  \CATI)^{\lax}$, with
  \begin{itemize}
  \item objects those functors $F \colon \tA \times \tB^{\op} \to
    \CATI$ such that $F(a,\blank)$ is representable for every $a \in
    \tA$,
  \item morphisms those lax transformations $\Phi \colon [1] \otimes (\tA \times
    \tB^{\op}) \to \CATI$ such that for every $a \in \tA$, the lax transformation
    $\Phi(\blank,a,\blank) \colon [1] \otimes \tB^{\op} \to \CATI$ is
    strong. \qed
  \end{itemize}
\end{cor}

\begin{remark}\label{rem:funoplaxemb}
  The previous corollary can be dualized to show that $\FUN(\tA,\tB)^{\oplax}$ embeds into $\FUN(\tA \times \tB^{\op},\CATI)^{\oplax}$ as a locally full sub-\itcat{}. This follows formally from the right-most pushout in \cref{propn:tripleprodpo} following the same strategy as in the lax case.
\end{remark}

We note a few more useful consequences of \cref{propn:tripleprodpo}:
\begin{cor}\label{cor:takingslicesout}
  Let $\tA$, $\tB$ be \itcats{}. Then for every object $b \in \tB$ we have a fully faithful functor (cf. \cref{not:slices})
  \[
     \iota \colon \FUN(\tA, \tB_{/b})^{\lax} \to \FUN(\tA,\tB)^{\lax}_{/\underline{b}},
 \] 
 where $\underline{b}$ denotes the constant functor on $b$; the image
 of $\iota$ is given by the strong natural transformations
 $F \to \underline{b}$.
\end{cor}
\begin{proof}
  For every $(\infty,2)$-category $\tX$ we have natural pullback squares
  \[
    \begin{tikzcd}
      \Map(\tX,\FUN(\tA, \tB_{/b})^\lax) \ar[r] \ar[d] & \Map([1] \times (\tX \otimes \tA) , \tB) \ar[d] \\
      \{\underline{b}\} \ar[r] & \Map(\{1\}\times (\tX \otimes \tA)  , \tB),
    \end{tikzcd}
  \]
  \[
    \begin{tikzcd}
    \Map(\tX,\FUN(\tA,\tB)^{\lax}_{/\underline{b}})   \ar[r] \ar[d] & \Map(([1] \times \tX) \otimes \tA, \tB) \ar[d] \\
      \{\underline{b}\} \ar[r] & \Map((\{1\} \times \tX) \otimes \tA, \tB).
    \end{tikzcd}
  \]
  From \cref{propn:tripleprodpo} we also have a pullback square
  \[
    \begin{tikzcd}
      \Map(([1] \times \tX) \otimes \tA, \tB) \ar[r] \ar[d] & \Map([1] \times (\tX \otimes \tA) , \tB) \ar[d] \\
\Map(([1] \times \tX^{\simeq}) \otimes \tA, \tB) \ar[r] & \Map([1] \times \tX^{\simeq} \times \tA , \tB);
    \end{tikzcd}
  \]
  combining this with the first squares we see that the natural
  transformation
  \[([1] \times \tX) \otimes \tA \to [1] \times (\tX \otimes \tA)\]
  induces a functor
  \[ \iota \colon \FUN(\tA, \tB_{/b})^\lax \to \FUN(\tA,\tB)^{\lax}_{/\underline{b}}\]
  by the Yoneda lemma, and composition with this gives a natural
  pullback square
  \[
    \begin{tikzcd}
      \Map(\tX,\FUN(\tA, \tB_{/b})^\lax) \ar[r, "\iota_{*}"] \ar[d] & \Map(\tX,\FUN(\tA,\tB)^{\lax}_{/\underline{b}}) \ar[d] \\
\Map(\tX^{\simeq},\FUN(\tA, \tB_{/b})^\lax) \ar[r, "\iota_{*}"] & \Map(\tX^{\simeq},\FUN(\tA,\tB)^{\lax}_{/\underline{b}}).
    \end{tikzcd}
  \]
  This says precisely that $\iota$ is fully faithful by \cref{cor:ffonspcs}.
\end{proof}

\begin{propn}\label{obs:Slaxoplaxrev}
  Let $(\tX, E)$ be a marked \itcat{}. For any \itcats{} $\tY, \tZ$, the natural equivalence
  \[ \FUN(\tX, \FUN(\tY, \tZ)^{\oplax})^{\lax} \simeq
    \FUN(\tY, \FUN(\tX, \tZ)^{\lax})^{\oplax}\]
  of \cref{lem:funandgray} restricts to a fully faithful inclusion
\[ \FUN(\tY, \FUN(\tX, \tZ)^{\elax})^{\oplax} \hookrightarrow \FUN(\tX, \FUN(\tY,\tZ)^{\oplax})^{\elax},\]
with image those functors $F \colon \tY \otimes \tX \to \tZ$ such that
  for every morphism $x \to x'$ in $E$ and every morphism $y \to y'$ in $\tY$, the 2-morphism in the
  square
  \[
    \begin{tikzcd}
      F(y,x) \arrow{r} \arrow{d} & F(y,x') \arrow{d} \\
      F(y',x) \arrow{r} \arrow[Rightarrow]{ur} & F(y',x')
    \end{tikzcd}
  \]
  is invertible.
\end{propn}
\begin{proof}
  Our goal is to compare the images of the two \itcats{} \[\FUN(\tY,
    \FUN(\tX, \tZ)^{\elax})^{\oplax}, \qquad \FUN(\tX, \FUN(\tY,\tZ)^{\oplax})^{\elax}\] in $\FUN(\tX,\FUN(\tY,\tZ)^{\oplax})^{\lax}$. 
  Let us fix a functor $f \colon \tA \to \FUN(\tX, \FUN(\tY, \tZ)^{\oplax})^{\lax}$ that corresponds to a functor $F \colon \tY \otimes \tA \otimes \tX \to \tZ$. We consider each case separately.
  \begin{itemize}
    \item  The functor $F$ factors through $\FUN(\tX, \FUN(\tY, \tZ)^{\oplax})^{\elax}$ precisely if (cf. \cref{obs:factorsthroughElax}) for every 
  morphism $x \to x'$ in $E$, the restriction to $\tY \otimes \tA \otimes [1]$
  factors through $\tY \otimes (\tA \times [1])$. 
  This means that for
  every $y \in \tY$ and every morphism $a \to a'$ in $\tA$, the 2-morphism in the square
  \[
    \begin{tikzcd}
      F(y,a,x) \arrow{r} \arrow{d} & F(y,a,x') \arrow{d} \\
      F(y,a',x) \arrow{r} \arrow[Rightarrow]{ur} & F(y,a',x')
    \end{tikzcd}
  \]
  is invertible; this is only a condition on the value of $F$ at
  \emph{morphisms} in $\tA$, while the value on objects is
  unrestricted.
  \item  On the other hand, to land in
  $\FUN(\tY, \FUN(\tX, \tZ)^{\elax})^{\oplax}$, the functor $F$ should
  factor through $(\tY \otimes \tA) \times [1]$ at every morphism in
  $E$. By \cref{cor:thattripleproductlocn}, this amounts to the
  previous condition together with the condition that for every object
  $a \in \tA$ and morphism $y \to y'$ in $\tY$, the 2-morphism in the
  square
  \[
    \begin{tikzcd}
      F(y,a,x) \arrow{r} \arrow{d} & F(y,a,x') \arrow{d} \\
      F(y',a,x) \arrow{r} \arrow[Rightarrow]{ur} & F(y',a,x')
    \end{tikzcd}
  \]
  is invertible. This corresponds precisely to the given condition on the value of $F$ on
  \emph{objects}.
  \end{itemize}
  The result now follows. 
\end{proof}


\section{Fibrations and lax transformations}\label{sec:fibs}
We start this section by reviewing fibrations of \itcats{}. We first
define (co)cartesian-enriched functors in \S\ref{sec:cartenr}, and
then review the \itcatl{} analogues of (co)cartesian fibrations and
their straightening in \S\ref{sec:cartesian}. In \S\ref{sec:locfib} we
then consider \emph{local} fibrations and their straightening, which
we then specialize to \emph{Gray fibrations}, which straighten to
functors from Gray tensor products, in \S\ref{sec:grayfib}. Using
these we can then prove \cref{thm:strlaxtr}, \ie{} straightening for
lax transformations, in \S\ref{sec:strlaxtr}. In \S\ref{sec:graybifib}
we study the bivariant version of Gray fibrations, which will be
important when we look at the mate correspondence in \S\ref{sec:mates}
below.

\subsection{Cartesian-enriched functors}\label{sec:cartenr}
As a preliminary to defining fibrations of \itcats{}, it is convenient to first briefly introduce the notion of ``(co)cartesian-enriched functors'' separately.

\begin{defn}\label{defn:cartenriched}
  A functor of \itcats{} $\pi \colon \tc{E} \to \tc{B}$ is
  \emph{cartesian-enriched} if
  \begin{enumerate}[(1)]
  \item For all $x,y \in \tc{E}$, the functor $\tc{E}(x,y) \to \tc{B}(\pi x, \pi y)$ is a cartesian fibration of \icats{}.
  \item For all $x,y,z \in \tc{E}$, the commutative square
    \[
      \begin{tikzcd}
        \tc{E}(x,y) \times \tc{E}(y,z) \arrow{r} \arrow{d} & \tc{E}(x,z) \arrow{d} \\
        \tc{B}(\pi x, \pi y) \times \tc{B}(\pi y, \pi z) \arrow{r} & \tc{B}(\pi x, \pi z),
      \end{tikzcd}
    \]
    where the horizontal morphisms are given by composition, is a
    morphism of cartesian fibrations (\ie{} the top horizontal functor
    preserves cartesian morphisms).
  \end{enumerate}
  If $\pi \colon \tc{E} \to \tc{B}$ and $\pi' \colon \tc{E}' \to \tc{B}'$ are cartesian-enriched, we say that a commutative square of \itcats{}
  \[
    \begin{tikzcd}
      \tc{E} \arrow{r}{\psi} \arrow{d}[swap]{\pi} & \tc{E}' \arrow{d}{\pi'} \\
      \tc{B} \arrow{r}{\phi} & \tc{B}'
    \end{tikzcd}
  \]
  is a \emph{morphism of cartesian-enriched functors} if for all $x, y \in \tc{E}$, the commutative square
  \[
    \begin{tikzcd}
      \tc{E}(x,y) \arrow{r} \arrow{d} & \tc{E}'(\psi x, \psi y) \arrow{d} \\
      \tc{B}(\pi x, \pi y) \arrow{r} & \tc{B}'(\pi' \psi x, \pi' \psi y)
    \end{tikzcd}
  \]
  is a morphism of cartesian fibrations.
  Dually, if $\pi^{\co}$ is cartesian-enriched, we say that $\pi$ is \emph{cocartesian-enriched}, and define morphisms of cocartesian-enriched functors similarly. We write $\Cartenr{\tB}$ for the locally full sub-\itcat{} of $\CATITsl{\tB}$ spanned by the cartesian-enriched functors and morphisms of cartesian-enriched functors over $\tB$; similarly, $\Cocartenr{\tB}$ is the cocartesian-enriched analogue.
\end{defn}

\begin{defn}
  We say that $\pi \colon \tE \to \tB$ is \emph{right-enriched} (or \emph{left-enriched}) if \[\tE(x,y) \to \tB(\pi x, \pi y)\] is a right (or left) fibration of \icats{} for all $x, y \in \tE$. We write $\Lenr{\tB}$ and $\Renr{\tB}$ for the full sub-\itcats{} of $\CATITsl{\tB}$ spanned by the left- and right-enriched functors, respectively.
\end{defn}

\begin{observation}\label{obs:lenrsubcat}
  If $\pi$ is right- (or left-)enriched then it is also
  (co)cartesian-enriched, since the second condition in the definition
  is vacuous when all morphisms are (co)cartesian. For the same
  reason, every functor between right- (or left-)enriched functors is
  (co)cartesian-enriched. Thus $\Lenr{\tB}$ and $\Renr{\tB}$ are also
  full sub-\itcats{} of $\Cocartenr{\tB}$ and $\Cartenr{\tB}$,
  respectively.
\end{observation}

\begin{observation}\label{ob:precomposition}
 Let $\phi \colon  f_0 \to f_1$ and $\psi \colon g_0 \to g_1$ be cartesian edges in $\tc{E}(x,y)$ and $\tc{E}(y,z)$, respectively, and consider the associated commutative diagram in $\tc{E}(x,z)$:
 \[
   \begin{tikzcd}
     g_0 \circ f_0 \arrow[r, "\psi \circ f_{0}"] \ar[dr, "\psi \circ \phi"{description}] \arrow[d, "g_{0} \circ \phi"{swap}] & g_1 \circ f_0 \arrow[d, "g_{1} \circ \phi"] \\
     g_0 \circ f_1 \arrow[r, "\psi \circ f_{1}"{swap}] & g_1 \circ f_1.
   \end{tikzcd}
 \]
 Condition (2) in \cref{defn:cartenriched} requires the composite
 $\psi \circ \phi$ to be cartesian. This holds
 provided the edges in the square are cartesian morphisms, since
 these are closed under composition (because
 $\tc{E}(x,z) \to \tc{B}(a,c)$ is by assumption a cartesian
 fibration). This condition is therefore equivalent to the seemingly
 weaker condition that pre- and postcomposition with 1-morphisms
 preserve cartesian 2-morphisms.
\end{observation}

\begin{observation}\label{obs:enrto1cat}
  If $\oB$ is an \icat{} and $\tE$ an \itcat{}, then any functor
  $\tE \to \oB$ is both cartesian- and cocartesian-enriched, since any
  functor to an \igpd{} is a (co)cartesian fibration whose (co)cartesian lifts are the equivalences.
\end{observation}

\begin{lemma}\label{lem:cartenrcomp}
  Suppose $p \colon \tE \to \tF$ and $q \colon \tF \to \tB$ are cartesian-enriched functors. Then the composite $qp \colon \tE \to \tB$ is also cartesian-enriched, and the $qp$-cartesian morphisms in $\tE(x,y)$ are those that are $p$-cartesian over a $q$-cartesian morphism in $\tF(px,py)$.
\end{lemma}
\begin{proof}
  This follows from the description of cartesian morphisms for a
  composite of cartesian fibrations in \cite{HTT}*{Proposition
    2.4.1.3}.
\end{proof}

\subsection{Cartesian morphisms and fibrations of \itcats{}}\label{sec:cartesian}

In this section we review the definitions of (co)cartesian morphisms
and fibrations of \itcats{} and their straightening. Such fibrations
were first defined by Gagna, Harpaz, and
Lanari~\cite{GagnaHarpazLanariFib} and were subsequently studied by
the first author and Stern~\cite{AGS1,AGS2}, who prove a straightening
theorem; this is also a special case of Nuiten's general straightening
for fibrations of $(\infty,n)$-categories \cite{Nuiten}.

\begin{defn}
  Let $\pi \colon \tc{E} \to \tc{B}$ be a functor of \itcats{}. If
  $\overline{f} \colon x \to y$ is a morphism in $\tc{E}$ lying over
  $f \colon a \to b$ in $\tc{B}$, then we say that $\overline{f}$ is a
  \emph{$p$-cocartesian} morphism if for every
  object $z \in \tc{E}$ over $c \in \tc{B}$, the commutative square of
  \icats{}
  \[
    \begin{tikzcd}
      \tc{E}(y,z) \arrow{r}{\overline{f}^{*}} \arrow{d} & \tc{E}(x,z) \arrow{d} \\
      \tc{B}(b,c) \arrow{r}{f^{*}} & \tc{B}(a,c)
    \end{tikzcd}
  \]
  is a pullback. Dually, we say that $\overline{f}$ is a \emph{$p$-cartesian} morphism if all the squares
  \[
    \begin{tikzcd}
      \tc{E}(z,x) \arrow{r}{\overline{f}_{*}} \arrow{d} & \tc{E}(z,y) \arrow{d} \\
      \tc{B}(c,a) \arrow{r}{f_{*}} & \tc{B}(c,b)
    \end{tikzcd}
  \]
  are pullbacks. We say that $\tc{E}$ has all \emph{$p$-cocartesian lifts}
  of some class $T$ of morphisms in $\tc{B}$ if given a morphism
  $f \colon a \to b$ in $T$ and an object $x$ in $\tc{E}$ over $a$, there
  exists a $p$-cocartesian morphism $\overline{f} \colon x \to y$
  lying over $f$; similarly, we say that $\tc{E}$ has all
  \emph{$p$-cartesian lifts} of $T$ if the dual condition holds.
\end{defn}

\begin{observation}\label{obs:fibre comp cocart}
  Consider a functor of \itcats{} $p \colon \tE \to \tB$ and suppose
  $\bar{f} \colon x \to f_{!}x$ is a $p$-cocartesian morphism over
  $f \colon a \to b$. Then for every object $y \in \tE$ over $b$,
  composition with $\bar{f}$ gives an equivalence
  \[ \tE_{b}(f_{!}x,y) \isoto \tE(x,y)_{f}\]
  by taking fibres over $\id_{b}$ in the square above. 
  If $\bar{g} \colon y \to g_{!}y$ and $\bar{g}' \colon f_{!}x \to g_{!}f_{!}x$ are cocartesian over $g \colon b \to
  c$, then the composition on fibres
  \[ \tE(x,y)_{f} \times \tE(y,z)_{g} \to \tE(x,z)_{gf}\]
  can furthermore be identified with the composite
  \[ \tE_{b}(f_{!}x,y) \times \tE_{c}(g_{!}y,z) \to
    \tE_{c}(g_{!}f_{!}x, g_{!}y) \times \tE_{c}(g_{!}f_{!}x,z)
    \to \tE_{c}(g_{!}f_{!}x, z)
  \] where
  the first map is given by cocartesian transport along $g$ and the
  second by composition in $\tE_{c}$. This can be seen from the
  commutative diagram
  \[
    \begin{tikzcd}
      {} & \tE_{c}(g_{!}f_{!}x,g_{!}y) \times \tE_{c}(g_{!}y,z) \ar[r]
      \ar[d, "\sim"] & \tE_{c}(g_{!}f_{!}x,z) \ar[d, "\sim"] \\
      \tE_{b}(f_{!}x,y) \times \tE_{c}(g_{!}y,z) \ar[r] \ar[d, "\sim"]
      & \tE(f_{!}x,g_{!}y)_{g} \times \tE_{c}(g_{!}y,z) \ar[r]  &
      \tE(f_{!}x, z)_{g} \ar[d, "\sim"] \\
      \tE(x,y)_{f} \times \tE(y,z)_{g} \ar[rr] & & \tE(x,z)_{gf}.
    \end{tikzcd}
  \]
\end{observation}

\begin{observation}\label{obs:cocartmormapfib}
  In the context of \cref{obs:fibre comp cocart}, suppose $p$ is
  moreover cartesian-enriched. If $\bar{g} \colon x \to g_{!}x$ is
  cocartesian over $g \colon a \to b$ and
  $\overline{\eta}\colon h \to \bar{g}$ is a cartesian morphism in
  $\tE(x, g_{!}x)$ over a morphism $\eta \colon f \to g$ in
  $\tB(a,b)$, then $h$ factors uniquely as $u \circ \bar{f}$ where $u$
  lies over $\id_{b}$. We claim that the cartesian transport functor
  $\eta^{*}$ can then be identified with composition with $u$, in the
  sense that we have a commutative square
  \[
    \begin{tikzcd}
      \tE_{b}(g_{!}x,y) \ar[r, "\bar{g}^{*}"{above}, "\sim"{below}] \ar[d, "u^{*}"] & \tE(x,y)_{g} \ar[d, "\eta^{*}"] \\
      \tE_{b}(f_{!}x,y) \ar[r, "\bar{f}^{*}"{above}, "\sim"{below}] & \tE(x,y)_{f}.
    \end{tikzcd}
  \]
  In order to see this, we observe that precomposition with $\overline{\eta} \colon u \circ \overline{f} \to \overline{g}$ yields a commutative diagram
  \[
    \begin{tikzcd}
      \tE_{b}(g_{!}x,y) \times \{1\} \arrow[r,"\sim"] \arrow[d] & \tE(x,y)_{g} \arrow[d] \\
       \tE_{b}(g_{!}x,y) \times [1] \arrow[r,"H"] & \tE(x,y)
    \end{tikzcd}
  \]
  where $H$ takes $q \colon g_{!}x \to y$ to the morphism $q \circ
  \overline{\eta} \colon q \circ u
  \circ \overline{f} \to q \circ \overline{g}$.
  Since $p$ is by assumption  cartesian-enriched, the components of the natural transformation
  $H$ must be given by cartesian 1-morphisms. In order words, we can identify the cartesian
  transport functor with the restriction of $H$ to $\tE(x,y)_{g} \simeq \tE_{b}(g_{!}x,y) \times \{0\}$.
  However, by construction this is precisely given by $\bar{f}^{*} \circ u^{*} \circ (\bar{g}^{*})^{-1}$.
\end{observation}

\begin{notation}
  It will sometimes be notationally convenient to use the terms
  $0$- and $1$-cartesian fibrations for cocartesian and
  cartesian fibrations of \icats{}, respectively. Similarly, we will sometimes refer
  to $0$- and $1$-cartesian morphisms in both \icats{} and
  \itcats{}.
\end{notation}

\begin{defn}
  A functor $\pi \colon \tc{E} \to \tc{B}$ of \itcats{} is a \emph{$(0,1)$-fibration} if:
  \begin{enumerate}[(1)]
  \item $\tc{E}$ has $p$-cocartesian lifts of all morphisms in $\tc{B}$.
  \item $p$ is cartesian-enriched.
  \end{enumerate}
  If $\pi \colon \tc{E} \to \tc{B}$ and $\pi' \colon \tc{E}' \to \tc{B}'$ are $(0,1)$-fibrations, we say that a commutative square of \itcats{}
  \[
    \begin{tikzcd}
      \tc{E} \arrow{r}{\psi} \arrow{d}[swap]{\pi} & \tc{E}' \arrow{d}{\pi'} \\
      \tc{B} \arrow{r}{\phi} & \tc{B}'
    \end{tikzcd}
  \]
  is a \emph{morphism of $(0,1)$-fibrations} if
  \begin{enumerate}[(1)]
  \item the square is a morphism of
    cartesian-enriched functors,
  \item $\psi$ takes $\pi$-cocartesian morphisms to $\pi'$-cocartesian
    ones.
  \end{enumerate}
  Similarly, we have the notions
  of $(i,j)$-fibrations and their morphisms for all choices of
  $i,j \in \{0,1\}$, where $p \colon \tE \to \tB$ is an $(i,j)$-fibration if it is
  $j$-cartesian-enriched and $\tc{E}$ has all $p$-$i$-cartesian lifts
  over $\tc{B}$. We write $\FIB_{(i,j)}(\tB)$ for the locally full sub-\itcat{} of
  $\CATITsl{\tB}$ containing the $(i,j)$-fibrations, the morphisms of
  $(i,j)$-fibrations, and all 2-morphisms among these. 
\end{defn}

\begin{observation}
  Suppose $\oc{B}$ is an \icat{}.  Then a functor of \itcats{}
  $p \colon \tc{E} \to \oc{B}$ is automatically cartesian-enriched by
  \cref{obs:enrto1cat}, so $p$ is a $(0,1)$-fibration \IFF{} $\tc{E}$
  has $p$-cocartesian lifts of morphisms in $\oc{B}$, in which case
  $p$ is also a $(0,0)$-fibration. To emphasize this, we will just
  call such functors $0$-fibrations (and, in the dual case,
  $1$-fibrations).
\end{observation}

\begin{observation}\label{obs:fibmormapfib}
  Suppose $\phi \colon \tE \to \tF$ is a morphism of
  $(0,1)$-fibrations over $\tB$. Then for objects $x,y$ in $\tE$ over
  $a,b \in \tB$ we have a morphism of cartesian fibrations
  \[
    \begin{tikzcd}[column sep=small]
      \tE(x,y) \arrow{rr}{\phi_{x,y}} \arrow{dr} & & \tF(\phi x,\phi y) \arrow{dl}
      \\
       & \tB(a,b).
    \end{tikzcd}
  \]
  Given a cocartesian morphism $\overline{f} \colon x \to f_{!}x$ in
  $\tE$ over $f \colon a \to b$, the map on fibres over $f$ fits in a
  commuting square
  \[
    \begin{tikzcd}
      \tE(x,y)_{f} \ar[r, "\phi_{x{,}y}"] \ar[d, "\sim"] & \tF(\phi x, \phi y)_{f} \ar[d, "\sim"] \\
      \tE_{b}(f_{!}x,y) \ar[r, "\phi_{f_{!}x{,}y}"] & \tF_{b}(\phi(f_{!}x), \phi y),
    \end{tikzcd}
  \]
  where both vertical maps are equivalences since $\phi(\overline{f})$ is also cocartesian.
\end{observation}

\begin{thm}[Nuiten~\cite{Nuiten}, Abell\'an--Stern~\cite{AGS2}]\label{thm:str01}
  There is an equivalence of \itcats{}
  \[ \FIB_{(0,1)}(\tB) \simeq \FUN(\tB, \CATIT),\]
  which is contravariantly natural in $\tB$ with respect to pullback along 1- and 2-morphisms
  on the left and composition with 1- and 2-morphisms on the right. Moreover, the fibres 
  of a $(0,1)$-fibration corresponding to a functor $F \colon \tB \to \CATIT$ 
  are identified under this equivalence with the values of $F$.
\end{thm}

\begin{observation}
  The 3 other types of fibrations are related to $(0,1)$-fibrations by equivalences of the form
  \begin{equation}
    \label{eq:fibop}
    \begin{split}
      \FIB_{(1,1)}(\tB) & \xrightarrow[\op]{\sim} \FIB_{(0,1)}(\tB^{\op})^{\co},\\
      \FIB_{(0,0)}(\tB)  & \xrightarrow[\co]{\sim} \FIB_{(0,1)}(\tB^{\co}),\\
      \FIB_{(1,0)}(\tB)  & \xrightarrow[\coop]{\sim} \FIB_{(0,1)}(\tB^{\coop})^{\co},
    \end{split}
  \end{equation}
  where we have used that reversing $1$-morphisms gives an equivalence
  $\op \colon \CATIT \isoto \CATIT^{\co}$.  We can then derive
  straightening equivalences for these types of fibrations as the
  composites
  \begin{equation}
    \label{eq:strvariants}
    \begin{split}
      \FIB_{(1,1)}(\tB) & \simeq \FIB_{(0,1)}(\tB^{\op})^{\co} \simeq \FUN(\tB^{\op}, \CATIT)^{\co} \\
                        & \xrightarrow[\op]{\sim} \FUN(\tB^{\op}, \CATIT^{\co})^{\co} \simeq \FUN(\tB^{\coop}, \CATIT)\\
      \FIB_{(0,0)}(\tB) & \simeq \FIB_{(0,1)}(\tB^{\co}) \simeq \FUN(\tB^{\co}, \CATIT) \\
                        & \xrightarrow[\co]{\sim} \FUN(\tB^{\co}, \CATIT)\\
      \FIB_{(1,0)}(\tB) & \simeq \FIB_{(1,0)}(\tB^{\coop})^{\co} \simeq \FUN(\tB^{\coop}, \CATIT)^{\co} \\
                        & \xrightarrow[\coop]{\sim} \FUN(\tB^{\coop}, \CATIT^{\co})^{\co} \simeq \FUN(\tB^{\op}, \CATIT),\\
    \end{split}
  \end{equation}
  where the equivalences in each second line are required to get the
  identity when $\tB$ is a point (or in other words to get the fibres of the
  fibration as the values of the functor).
\end{observation}

\begin{remark}
  The following table summarizes our terminology for the 4 types of fibrations and its relation to other terminology used in the literature, as well as the variance of their straightenings:
  \begin{center}
    \begin{tabular}[c]{ccccc}
      Name  & Other name & 1-morphisms  & 2-morphisms & Over $\tB$ straightens to \\
      \hline
      (0,1) & inner cocartesian & cocartesian & cartesian & $\tB \to \CATIT$ \\    
      (1,1) & inner cartesian & cartesian & cartesian & $\tB^{\coop} \to \CATIT$ \\
      (1,0) & outer cartesian & cartesian & cocartesian & $\tB^{\op} \to \CATIT$ \\
      (0,0) & outer cocartesian & cocartesian & cocartesian & $\tB^{\co} \to \CATIT$ 
    \end{tabular}
  \end{center}
\end{remark}

\begin{remark}
  At first glance it is surprising that a fibration $p \colon \tE \to
  \tB$ that encodes a fully covariant functor $\tB \to \CATIT$ should
  be given by \emph{cartesian} fibrations on mapping \icats{}. To explain this, consider a 2-morphism in $\tB$ and the associated diagram of \itcats{}:
  \[
    \begin{tikzcd}[column sep=large]
      b \ar[r, bend left=50, "f"{above}, ""{below,name=A,inner sep=2pt}]
      \ar[r, bend right=50, "g"{below}, ""{above,name=B,inner sep=2pt}] & b'
      \ar[from=A, to=B, Rightarrow, "\eta"]
    \end{tikzcd} \quad \mapsto \quad
    \begin{tikzcd}[column sep=large]
      \tE_{b} \ar[r, bend left=50, "f_{!}"{above}, ""{below,name=A,inner sep=2pt}]
      \ar[r, bend right=50, "g_{!}"{below}, ""{above,name=B,inner sep=2pt}] & \tE_{b'}
      \ar[from=A, to=B, Rightarrow, "\eta_{!}"]
    \end{tikzcd}
  \]
  Given $x \in \tE_{b}$, we have cocartesian 1-morphisms
  $x \to f_{!}x$ and $x \to g_{!}x$ over $f$ and $g$, and we want a
  morphism $f_{!}x \to g_{!}x$ in $\tE_{b'}$ that will be a component
  of the natural transformation $\eta_{!}$. To find this, we consider for $y \in \tE_{b'}$ the cartesian transport functor
  \[ \eta^{*} \colon \tE(x,y)_{g} \to \tE(x,y)_{f};\]
  the cocartesian morphisms over $f$ and $g$ identify this with a functor
  \[ \tE_{b'}(g_{!}x,y) \simeq \tE(x,y)_{g} \xto{\eta^{*}}
    \tE(x,y)_{f} \simeq \tE_{b'}(f_{!}x,y), \] and using naturality in
  $y$ and the Yoneda lemma (or just putting $y = g_{!}x$) we get the
  required morphism $f_{!}x \to g_{!}x$. If $p$ was
  \emph{cocartesian} on 2-morphisms, we would instead construct a map
  $g_{!}x \to f_{!}x$, which is consistent with the straightening
  reversing 2-morphisms.
\end{remark}

\begin{propn}
  Suppose
    \[
    \begin{tikzcd}
      \tc{E} \arrow{r}{\psi} \arrow{d}[swap]{\pi} & \tc{E}' \arrow{d}{\pi'} \\
      \tc{B} \arrow{r}{\phi} & \tc{B}'
    \end{tikzcd}
  \]
is a morphism of $(i,j)$-fibrations. Then this square is a pullback \IFF{} for every $x \in \tB$, the morphism on fibres $\tE_{x} \to \tE'_{\phi(x)}$  is an equivalence of \itcats{}.
\end{propn}
\begin{proof}
  If the square is cartesian, the pasting property of pullbacks applied to the commutative diagram
    \[
    \begin{tikzcd}
 \tc{E}_{x} \ar[r] \ar[d] &     \tc{E} \arrow{r}{\psi} \arrow{d}{\pi} & \tc{E}' \arrow{d}{\pi'} \\
     \{x\} \ar[r] & \tc{B} \arrow{r}{\phi} & \tc{B}'
    \end{tikzcd}
  \]
  shows that $\tE_{x} \to \tE'_{\phi(x)}$ is an equivalence. 

  To prove the converse, we need to show that the induced map
  $f \colon \tE \to \tE' \times_{\tB} \tB'$ is an equivalence of \itcats{}. Replacing $\psi$ by this map, we may assume without loss of generality that $\phi$ is $\id_{\tB}$, so that we have a commutative triangle
  \[
    \begin{tikzcd}
      \tE \ar[rr, "f"] \ar[dr] & & \tE' \ar[dl] \\
       & \tB.
    \end{tikzcd}
  \]
  We know that $f$ is essentially surjective on each fibre over $\tB$, so it must be essentially surjective. It remains to show that $f$ is fully faithful, that is to say that for all $x,y \in \tE$  over $a,b \in \tB$, the horizontal map in the commutative triangle
  \[
    \begin{tikzcd}
      \tE(x,y) \ar[rr, "f"] \ar[dr] & & \tE'(fx,fy) \ar[dl] \\
       & \tB(a,b)
    \end{tikzcd}
  \]
is an equivalence. This is a map of cartesian fibrations of \icats{} that preserve cartesian morphisms, so we know it is an equivalence provided that the map on fibres
  \[
    \tE(x,y)_{\beta} \to \tE'(f(x),f(b))_{f(\beta)}
  \]
  is an equivalence for all $\beta \in \tB(a,b)$. This follows directly from our hypothesis together with \cref{obs:cocartmormapfib}.
\end{proof}

It is convenient to also introduce some terminology for fibrations whose fibres are \icats{} and \igpds{}:

\begin{defn}
  We say an $(i,j)$-fibration $\tc{E} \to \tc{B}$ is \emph{$1$-fibred}
  if the fibres $\tE_{b}$ are \icats{} for all $b \in \tB$, and
  \emph{$0$-fibred} if they are \igpds{}.  We write
  $\FIB_{(i,j)}^{(t)}(\tB)$ for the full sub-\itcat{} of
  $\FIB_{(i,j)}(\tB)$ spanned by the $t$-fibred $(i,j)$-fibrations for
  $t = 0,1$.
\end{defn}

\begin{propn}\label{propn:1fibredcond}
  Let $p \colon \tE \to \tB$ be a functor of \itcats{}.
  \begin{enumerate}[(1)]
  \item $p$ is a $1$-fibred $(0,1)$-fibration \IFF{} $p$ is
    right-enriched and $\tE$ has $p$-cocartesian lifts of all
    morphisms in $\tB$.
  \item $p$ is a $0$-fibred $(0,1)$-fibration \IFF{}
    $p$ is right-enriched and every 1-morphism in $\tE$ is
    $p$-cocartesian.
  \end{enumerate}
\end{propn}
\begin{proof}
  By \cref{obs:cocartmormapfib} the fibres of $\tE(x,y) \to
  \tB(px,py)$ can be identified with mapping \icats{} in $\tE_{py}$.
  If $p$ is $1$-fibred then the latter are \igpds{}, so $\tE(x,y) \to
  \tB(px,py)$ is a right fibration. Conversely, the mapping \icat{}
  $\tE_{b}(x,y)$ is the fibre at $\id_{b}$ of $\tE(x,y) \to \tB(b,b)$,
  so if the latter is a right fibration it follows that $\tE_{b}$ is
  an \icat{}.

  To prove (2), we note that by (1) both conditions imply that $p$ is
  a $1$-fibred $(0,1)$-fibration, so every 1-morphism in $\tE$ factors
  as a cocartesian morphism followed by a fibrewise one. Moreover, a fibrewise
  morphism is cocartesian \IFF{} it is invertible, so that the fibres
  of $p$ are \igpds{} \IFF{} every morphism is $p$-cocartesian.
\end{proof}

\begin{propn}
  The straightening equivalence restricts to natural equivalences
  \[ \FIB_{(0,1)}^{(1)}(\tB) \simeq \FUN(\tB, \CATI)\]
  \[ \FIB_{(0,1)}^{(0)}(\tB) \simeq \FUN(\tB, \Spc)\]  
  for $1$- and $0$-fibred fibrations.
\end{propn}
\begin{proof}
  Both $\Spc$ and $\CATI$ are full sub-\itcats{} of $\CATIT$.
  By \cref{cor:funprops}, this means that $\FUN(\tB, \Spc)$ and $\FUN(\tB, \CATI)$ are the full sub-\itcats{} of 
  $\FUN(\tB, \CATIT)$ spanned by those functors whose value at each
  object of $\tB$ is an \igpd{} or an \icat{}, respectively. On the fibrational side, this
  precisely corresponds to the $(0,1)$-fibrations whose fibre at each
  object is an \igpd{} or an \icat{}.
\end{proof}

\subsection{Locally cartesian morphisms and local fibrations}\label{sec:locfib}
In this section, we review the definitions of locally (co)cartesian
morphisms and local fibrations of \itcats{}, and their straightening,
as studied in \cite{Ab23}.

\begin{defn}
  Suppose $p \colon \tc{E} \to \tc{B}$ is a functor of \itcats{}. A morphism $\overline{f}$ in $\tc{E}$ is \emph{locally $p$-(co)cartesian} over $f$ in $\tc{B}$ \IFF{} $\overline{f}$ is a (co)cartesian morphism for the pullback $f^{*}\tc{E} \to [1]$.
\end{defn}

\begin{defn}
  We say that a functor $p \colon \tc{E} \to \tc{B}$ is a \emph{local
    $(0,1)$-fibration} if $p$ is cartesian-enriched and the pullback
  $f^{*}\tc{E} \to [1]$ is a $(0,1)$-fibration for all morphisms $f$
  in $\tc{B}$. We write $\LFIB_{(0,1)}(\tB)$ for the locally full
  sub-\itcat{} of $\CATITsl{\tB}$ that contains the local
  $(0,1)$-fibrations and the cartesian-enriched functors that preserve
  locally cocartesian morphisms. We also define local
  $(i,j)$-fibrations for all $i,j \in \{0,1\}$ similarly.
\end{defn}

More generally, we can consider fibrations that are local with respect
to a \emph{scaling} of an \itcat{}, in the following sense:
\begin{defn}
  Let $(\tc{B},S)$ be a scaled \itcat{}. We say that a functor
  $p \colon \tc{E} \to \tc{B}$ is an \emph{$S$-local
    $(0,1)$-fibration} if it is a local $(0,1)$-fibration, and for
  every commutative triangle
    \[
    \begin{tikzcd}
      a \arrow{rr}{f} \arrow{dr}[swap]{gf} & & b \arrow{dl}{g} \\
       & c
    \end{tikzcd}
  \]
  in $S$, if $\overline{f} \colon x \to y$ and
  $\overline{g} \colon y \to z$ are locally $p$-cocartesian morphisms
  over $f$ and $g$, then the composite $\overline{g}\overline{f}$ is a
  locally $p$-cocartesian morphism (over $gf$). We write
  $\LFIB^{S}_{(0,1)}(\tB)$ for the full sub-\itcat{} of
  $\LFIB_{(0,1)}(\tB)$ spanned by the $S$-local $(0,1)$-fibrations.
\end{defn}

\begin{observation}\label{obs:localfibpb}
  For any morphism $f \colon (\tA, S) \to (\tB,T)$ of scaled
  \itcats{}, pullback along $f$ gives a functor
  \[ f^{*} \colon \LFIB^{T}_{(0,1)}(\tB) \to \LFIB^{S}_{(0,1)}(\tA).\]
\end{observation}

\begin{lemma}
  A functor $p \colon \tc{E} \to \tc{B}$ is an
  $S^{\sharp}_{\tc{B}}$-local $(0,1)$-fibration, where $S^{\sharp}_{\tB}$ consists of all commuting triangles in $\tB$, \IFF{} $p$ is a
  $(0,1)$-fibration.
\end{lemma}
\begin{proof}
  This is a special case of \cite{Ab23}*{Theorem 4.27}.
\end{proof}

\begin{thm}[{\cite{Ab23}}]\label{thm:localstr}
  There is a natural equivalence of \itcats{}
  \[ \LFIB_{(0,1)}^{S}(\tc{B}) \simeq \FUN(\tc{B}_{S}, \CATIT),\]
  where $\tc{B}_{S}$ is as in \cref{constr:laxitcat}. In particular,
  \[ \LFIB_{(0,1)}(\tB) \simeq \FUN(\tB_{\oplax}, \CATIT)\]
  for any \itcat{} $\tB$.\qed
\end{thm}

\begin{remark}
  Another version of the straightening theorem for local fibrations
  has been proved by Ayala, Mazel-Gee, and Rozenblyum
  \cite{AMGR}*{Theorem B.4.3}. This uses an a priori different notion
  of lax functors, but this has recently been compared to the
  definition via scaled simplicial sets that we use in work of the
  first author \cite{AbLax}.
\end{remark}

\subsection{Gray fibrations}\label{sec:grayfib}
In this section we introduce \emph{Gray fibrations} over products of
\itcats{}, which can be straightened to functors out of Gray tensor
products; this generalizes the Gray fibrations of \icats{} studied in
\cite{HHLN1}*{\S 2.4}. We start by defining these as $S$-local
fibrations for a particular scaling $S$; the rest of the section will
then be devoted to providing a more concrete description, generalizing
\cite{HHLN1}*{2.4.1}.

\begin{defn}
  If $\tc{A}$ and $\tc{B}$ are \itcats{}, we say that a functor
  $p \colon \tc{E} \to \tc{A} \times \tc{B}$ is a \emph{Gray
    $(0,1)$-fibration} if it is an $S$-local $(0,1)$-fibration where
  $S$ is as in \cref{defn:graytens}. We write
  $\GFIB_{(0,1)}(\tc{A}, \tc{B})$ for the corresponding \itcat{}
  $\LFIB_{(0,1)}^{S}(\tc{A} \times \tc{B})$.
\end{defn}

Given our definition of the Gray tensor product in terms of scaled
\itcats{} in \S\ref{subsec:gray} above, \cref{thm:localstr}
specializes to the following straightening result for Gray fibrations:
\begin{cor}\label{cor:gfibtogtens}
  There is a natural equivalence of \itcats{}
  \[ \GFIB_{(0,1)}(\tc{A}, \tc{B}) \simeq \FUN(\tc{A} \otimes \tc{B}, \CATIT)\]
  for \itcats{} $\tc{A}, \tc{B}$. \qed
\end{cor}

Our main goal in this section is to give a simpler characterization of
Gray $(0,1)$-fibrations, in \cref{propn:grayfib}. Before we come to this, it is convenient to first introduce some terminology for more general two-variable fibrations:

Given a functor $\tc{A} \to \CATITsl{\tc{B}}$, we can unstraighten
this to a commutative triangle
\[
  \begin{tikzcd}
    \tc{E} \arrow{dr} \arrow{rr} & & \tc{A} \times \tc{B} \arrow{dl}{\pr_{\tc{A}}} \\
     & \tc{A},
  \end{tikzcd}
\]
which is a morphism of $(0,1)$-fibrations over $\tc{A}$. Such triangles are characterized by the following definition:
\begin{defn}
  We say that a functor $\pi \colon \tE \to \tA \times \tB$ is a
  \emph{$(0,1)$-fibration over $\tA$} if:
  \begin{itemize}
  \item $\pi_{\tA} \colon \tE \to \tA$ is a $(0,1)$-fibration,
  \item $\pi$ is a morphism of $(0,1)$-fibrations over $\tA$, \ie{}
    cocartesian $1$-morphisms and cartesian $2$-morphisms in $\tE$ map
    to equivalences in $\tB$.
  \end{itemize}
  In other words, $\pi$ is an object of $\FIB_{(0,1)}(\tA)_{/\tA
    \times \tB}$.
\end{defn}

\begin{observation}
  The straightening equivalence induces an equivalence on slices
  \begin{equation}
    \label{eq:slicestr}
    \FIB_{(0,1)}(\tA)_{/\tA
      \times \tB} \simeq \FUN(\tA, \CATIT)_{/\tB} \simeq \FUN(\tA,
    \CATITsl{\tB}),
  \end{equation}
  so that $(0,1)$-fibrations over $\tA$ with target $\tA \times \tB$
  correspond to functors $\tA \to \CATITsl{\tB}$.
\end{observation}

\begin{thm}\label{propn:grayfib}
  A functor $\pi \colon \tE \to \tA \times \tB$ is a Gray $(0,1)$-fibration
  \IFF{} the following conditions hold:
  \begin{enumerate}[(1)]
  \item $\pi$ is a $(0,1)$-fibration over $\tA$,
  \item for every $a \in \tA$, the functor on fibres $\tE_{a} \to \tB$
    is a $(0,1)$-fibration,
  \item for every morphism $a \to a'$ in $\tA$, the cocartesian
    transport functor $\tE_{a} \to \tE_{a'}$ is a morphism of
    cartesian-enriched functors to $\tB$.
  \end{enumerate}
\end{thm}

\begin{remark}
  \cref{propn:grayfib} is a special case of \cite{AMGR}*{Lemma B.4.8},
  which only requires local $(0,1)$-fibrations where we ask for
  $(0,1)$-fibrations.
\end{remark}

We first show that part of the conditions in \cref{propn:grayfib} imply that the functor in question is cartesian-enriched:
\begin{lemma}\label{lem:graycartenr}
  Suppose the functor $\pi \colon \tE \to \tA \times \tB$
  is a $(0,1)$-fibration over $\tA$. Then $\pi$ is a
  cartesian-enriched functor \IFF{} it satisfies
  the following conditions:
  \begin{enumerate}[(1)]
  \item for every $a \in \tA$, the functor on fibres $\tE_{a} \to \tB$
    is cartesian-enriched,
  \item for every morphism $a \to a'$ in $\tA$, the cocartesian
    transport functor $\tE_{a} \to \tE_{a'}$ is a morphism of
    cartesian-enriched functors to $\tB$.
  \end{enumerate}
\end{lemma}
\begin{proof}
  We first assume the two numbered conditions
  hold. Given a pair of objects $x,y \in \tE$ over $a,a' \in \tA$ and
  $b,b'$ in $\tB$, we have by assumption a morphism of
  cartesian-enriched functors
    \[
    \begin{tikzcd}
      \tE(x,y) \arrow[rr,"\pi_{x,y}"] \arrow[dr,swap] & &  \tc{A}(a,a') \times \tc{B}(b,b') \arrow[dl] \\
      &  \tc{A}(a,a'),
    \end{tikzcd}
  \]
  and we want to show that the horizontal map is a cartesian fibration. For this we use 
  \cite{cois}*{Lemma A.1.8}, from which we see that it suffices to show that
  \begin{itemize}
  \item $\tE(x,y)_{f} \to \tB(b,b')$ is a cartesian fibration for
    every $f \colon a \to a'$ in $\tA$,
  \item for every morphism $\varphi \colon f \Rightarrow g$ in $\tA(a,a')$, the cartesian transport functor
    \[ \varphi^{*} \colon \tE(x,y)_{g} \to \tE(x,y)_{f}\]
    preserves cartesian morphisms over $\tB(b,b')$.
  \end{itemize}
  For the first condition, we know from  \cref{obs:fibmormapfib} that the fibre over $f$ can be identified with
  the composite
  \[
     \tE_{a'}(f_{!}x,y) \xrightarrow{\bar{f}^{*}} \tE(x,y) \to  \tc{B}(b,b'),
   \]
   where $\bar{f} \colon x \to f_{!}x$ is a cocartesian
   morphism over $f$; this functor is a cartesian fibration by
   assumption. Next, to see that $\varphi^{*}$ preserves cartesian morphisms, we pick a cocartesian morphism $\bar{g} \colon x \to g_{!}x$ over $g$ and a cartesian 2-morphism $\overline{\varphi} \colon h \Rightarrow \bar{g}$ over $\varphi$. Then $h$ factors through the cocartesian morphism $\bar{f}$, so we obtain a morphism
   $u \colon f_{!}x \to g_{!}x$ lying over the identity of
   $a'$ such that $h \simeq u \circ \bar{f}$.
   By \cref{obs:cocartmormapfib} we then have a commutative square
   \[
     \begin{tikzcd}
       \tE_{a'}(g_{!}x,y) \ar[r, "\sim"] \ar[d, "u^{*}"] & \tE(x,y)_{g} \ar[d, "\varphi^{*}"] \\
       \tE_{a'}(f_{!}x,y) \ar[r, "\sim"] & \tE(x,y)_{f}
     \end{tikzcd}
   \]
   over $\tB(b,b')$.
   
   Now since $\tE_{a'} \to \tB$ is a cartesian-enriched functor,
   composition with $u$ preserves cartesian morphism over
   $\tB(b,b')$ and hence so does the cartesian transport
   functor $\varphi^{*}$.
   It then follows from \cite{cois}*{Lemma A.1.8} that $\pi_{x,y}$ is
   a cartesian fibration whose cartesian edges are precisely those
   that can be expressed as composites $u\circ v$ where $u$ is a
   cartesian 2-morphism for $\pi_{\tc{A}}$ and $v$ is a cartesian
   2-morphism in a fibre $\tE_{a'}\to \tc{B}$.

   Next, we must prove that composition preserves cartesian
   2-morphisms, \ie{} that for a third object $z$ in $\tE$ over $a''
   \in \tA$ and $b''
   \in \tB$, the commutative square
\[
   \begin{tikzcd}
       \tE(x,y) \times \tE(y,z) \ar[r] \ar[d] & \tE(x,z)  \ar[d] \\
       \tA(a,a') \times \tB(b,b') \times \tA(a',a'') \times \tB(b',b'')  \ar[r] & \tA(a,a'') \times \tB(b,b'')
     \end{tikzcd}
\]
   is a morphism of cartesian fibrations. By assumption the top
   horizontal functor preserves cartesian morphisms over composition
   in $\tA$, so it suffices to check that the square on fibres over $f
   \colon a \to a'$ and $g \colon a' \to a''$ is a
   morphism of cartesian fibrations. By \cref{obs:fibre comp cocart}
   this square is equivalent to
      \[
     \begin{tikzcd}
      \tE_{a'}(f_{!}x,y) \times \tE_{a''}(g_{!}y,z) \ar[r]
      \ar[d] &
      \tE_{a''}(g_{!}f_{!}x,g_{!}y) \times \tE_{a''}(g_{!}y,z)
      \ar[r] \ar[d] &
      \tE_{a''}(x,z) \ar[d] \\
       \tB(b,b') \times \tB(b',b'')  \ar[r, "="] & \tB(b,b') \times \tB(b',b'')
        \ar[r] & \tB(b,b'').
     \end{tikzcd}
   \]
   Here the left square is a morphism of cartesian fibrations by the
   assumption that $g_{!} \colon \tE_{a'} \to \tE_{a''}$ is a
   morphism of cartesian-enriched functors over $\tB$, while the right
   square is one since $\tE_{a''}$ gives a cartesian-enriched functor
   to $\tB$.

   We now prove the converse. If $\pi$ is cartesian-enriched, then so
   are its fibres $\pi_{a} \colon \tE_{a} \to \tB$, since
   cartesian-enriched functors are closed under base change; a
   cartesian 2-morphism in $\tE_{a}$ is a $\pi$-cartesian 2-morphism
   in $\tE$ that lies over $\id_{a}$ in $\tA$.
   Since $\pi$ is cartesian-enriched, for $x,y,z \in \tE$ 
   over $(a,b), (a,b'), (a',b'')$, composition gives a morphism
   of cartesian fibrations
\[
   \begin{tikzcd}
       \tE(x,y) \times \tE(y,z) \ar[r] \ar[d] & \tE(x,z)  \ar[d] \\
       \tA(a,a) \times \tB(b,b') \times \tA(a,a') \times \tB(b',b'')
       \ar[r] & \tA(a,a') \times \tB(b,b'').
     \end{tikzcd}
   \]
   Taking fibres over $\id_{a}$ and $f \colon a \to a'$ we get a
   morphism of cartesian fibrations
   \[
   \begin{tikzcd}
       \tE_{a}(x,y) \times \tE(y,z)_{f} \ar[r] \ar[d] & \tE(x,z)_{f}  \ar[d] \\
        \tB(b,b') \times \tB(b',b'')
       \ar[r] & \tB(b,b'');
     \end{tikzcd}
   \]
   in particular, composition with a cocartesian morphism $\bar{f}
   \colon y \to f_{!}y$ over $f$ preserves cartesian morphism, so that
   the left triangle in the commutative diagram
   \[
   \begin{tikzcd}
       \tE_{a}(x,y)  \ar[r, "\bar{f}_{*}"] \ar[dr] & \tE(x,f_{!}y)_{f}
       \ar[d]  & \tE_{a'}(f_{!}x,f_{!}y) \ar[l, "\sim"] \ar[dl] \\ 
       & \tB(b,b')
     \end{tikzcd}
   \]
   is a morphism of cartesian fibrations. The right triangle, given by
   composition with a cocartesian morphism $x \to f_{!}x$, is an
   equivalence, and the horizontal composite is the cocartesian
   transport functor, so this shows that the latter preserves
   cartesian morphisms, as required.
\end{proof}

We also need the following trivial observation:
\begin{lemma}
  Given functors $\tE \xto{p} \tc{B} \xto{q} \tc{A}$ of \itcats{}, suppose $f \colon x \to y$  is a morphism in $\tE$ such that $p(f)$ is $q$-cocartesian. Then $f$ is $p$-cocartesian \IFF{} it is $qp$-cocartesian.
\end{lemma}
\begin{proof}
  Given $z \in \tE$, we apply the pasting lemma for pullbacks to the commutative diagram
  \[
    \begin{tikzcd}
      \tE(y,z) \arrow{r} \arrow{d} & \tE(x,z) \arrow{d} \\
      \tc{B}(py,pz) \arrow{r} \arrow{d} & \tc{B}(px,pz) \arrow{d} \\
      \tc{A}(qpy,qpz) \arrow{r} & \tc{A}(qpx,qpz)
    \end{tikzcd}
  \]
  where the bottom square is cartesian since $p(f)$ is $q$-cocartesian.
\end{proof}

As a useful special case we have:
\begin{lemma}\label{lem:projeqcart}
  Given $\pi \colon \tE \to \tc{A} \times \tc{B}$, a morphism in
  $\tE$ whose image in $\tc{B}$ is an equivalence is
  $\pi$-cocartesian \IFF{} it is $\pi_{\tc{A}}$-cocartesian. \qed
\end{lemma}

 We are now ready to prove \cref{propn:grayfib}:

\begin{proof}[Proof of \cref{propn:grayfib}]
  We first assume $\pi$ satisfies the given conditions; we want to
  prove that $\pi$ is then a Gray $(0,1)$-fibration. By
  \cref{lem:graycartenr} we know that $\pi$ is cartesian-enriched, and
  we next prove that $\pi$ is a local $(0,1)$-fibration.

  Let $f \colon x \to z$ in $\tE$ be a morphism over $(a,b) \to
  (a',b')$ that can be written
  as a composite $x \xrightarrow{u} y \xrightarrow{v} z$, where $u$ is
  $\pi_{\tA}$-cocartesian over $a \to a'$ (and so can be taken to lie
  over $\id_{b}$ in $\tB$) and $v$ is cocartesian over
  $b \to b'$
  in the fibre $\tE_{a'} \to \tB$. Then we claim that $f$ is
  locally $\pi$-cocartesian. To check this, we need to verify
  that for every $w \in \tE_{(a',b')}$, we have a
  pullback square
\[
  \begin{tikzcd}
    \tE_{(a',b')}(z,w) \arrow[r,"f^*"] \arrow[d] & \tE(x,w) \arrow[d] \\
    {[0]} \arrow[r,"\pi(f)"] & \tA(a,a')\times \tB(b,b').
  \end{tikzcd}
\]
We can extend this square into the diagram
\[
  \begin{tikzcd}
    (\tE_{a'})_{b}(z,w) \arrow[r] \arrow[d] & \tE_{a'}(y,w) \arrow[d] \arrow[r] &  \tE(x,w) \arrow[d] \\
    {[0]}  \arrow[r] & \tc{B}(b, b') \arrow[r] & \tc{A}(a,a')\times \tc{B}(b,b')
  \end{tikzcd}
\]
where the two inner diagrams are pullbacks. We conclude that the outer square is also a pullback. 

This shows that we have the required locally $\pi$-cocartesian
morphisms, so the final thing to check is that these compose along the
triangles in the Gray scaling $S$. Let us consider a pair of locally
cocartesian morphisms $f,g$ whose composite $g\circ f$ lies over a
triangle in $S$. We set $g=g_1 \circ g_0$ and $f=f_1 \circ f_0$
according to the previous description of the locally
$\pi$-cocartesian morphisms given above. We consider two cases:
\begin{itemize}
  \item Assume that the image of $g$ is invertible in $\tA$. Then $g_0$ is invertible in $\tE$ and thus $g$ can be identified with a cocartesian morphism in a fibre. We conclude that $g \circ f$ is locally $\pi$-cocartesian.
  \item Assume that the image of $f$ is invertible in $\tB$. Then $f_1$ is invertible in $\tE$, so that $f$ is $\pi_{\tA}$-cocartesian. Again, we conclude that $g \circ f$ is locally $\pi$-cocartesian.
  \end{itemize}
  
  Now we prove the converse: Suppose that
  $\pi \colon \tE \to \tA \times \tB$ is a Gray
  $(0,1)$-fibration. Then $\pi_{\tA} \colon \tE \to \tA$ is a
  composite of cartesian-enriched functors, and so it is itself
  cartesian-enriched by \cref{lem:cartenrcomp}. Moreover, the
  $\pi_{\tA}$-cartesian 2-morphisms are those $\pi$-cartesian 2-morphisms
  whose image in $\tB$ is invertible, since the latter condition
  describes the cartesian 2-morphisms for the projection
  $\tA \times \tB \to \tA$.

 Let $\bar{f} \colon x \to y$ in $\tE$ be a locally $\pi$-cocartesian
 that lies over $f \colon a \to a'$ in $\tA$ and an equivalence $f'
 \colon b
 \isoto b'$ in $\tB$. We
 claim that $\bar{f}$ is $\pi$-cocartesian; this will then imply, by
 \cref{lem:projeqcart}, that $\bar{f}$ is $\pi_{\tA}$-cocartesian. To see
 this we need to show that for any $z \in \tE$ over $(a'',b'')$ the diagram 
\[
  \begin{tikzcd}
    \tE(y,z) \arrow[r,"\bar{f}^*"] \arrow[d] & \tE(x,z) \arrow[d] \\
    \tA(a', a'')\times \tB(b', b'') \arrow[r,"{(f,f')^{*}}"] & \tA(a,
    a'')\times \tB(b, b'')
  \end{tikzcd}
\]
is a pullback square. For $g \colon a' \to a''$ and $h \colon b' \to
b''$, let $u \colon y \to y'$ be a locally $\pi$-cocartesian morphism over
$(g,h)$. Then the commutative triangle
\[
  \begin{tikzcd}
   x \ar[rr, "\bar{f}"] \ar[dr, "u\bar{f}"'] & & y \ar[dl, "u"] \\
    & y'
  \end{tikzcd}
\]
lies over a triangle in $S$, and so $u\bar{f}$ is also locally
$\pi$-cocartesian. It follows that in the commutative triangle
\[
  \begin{tikzcd}
    \tE_{(a'',b'')}(y',z) \ar[d, "u^{*}", "\sim"'] \ar[dr,
    "u\bar{f}^{*}", "\sim"'], \\
    \tE(y,z)_{(g,h)} \ar[r, "\bar{f}^{*}"'] & \tE(x,z)_{(gf,h)}
  \end{tikzcd}
\]
two of the three maps are equivalences, as indicated; the last map,
which is the map on fibres over $(g,h)$ in our square above, is
therefore also an equivalence. This shows that the square is a
pullback,
and so $\bar{f}$ is a $\pi_{\tA}$-cocartesian morphism. It
follows that $\tE$ has all $\pi_{\tA}$-cocartesian lifts, so that condition (1) is satisfied.

Note that since $\pi$ is $S$-local, condition (2) follows
immediately. It remains to show that condition (3) is satisfied, and
this follows from
\cref{lem:graycartenr}.
\end{proof}

\begin{variant}\label{var:markedgray}
  Suppose $\tA$ is an \itcat{} and $(\tB, E)$ is a marked
  \itcat{}. Then a functor $\tA \otimes_{\natural,E} \tB \to \CATIT$
  corresponds to an $S$-local $(0,1)$-fibration over $\tA \times \tB$
  with $S$ as in \cref{defn:graytens}; we call these \emph{marked Gray
    $(0,1)$-fibrations}. The same argument as for \cref{propn:grayfib}
  shows that these can be characterized as functors $\pi \colon \tE \to \tA \times \tB$ such that
  \begin{enumerate}[(1)]
  \item $\pi$ is a $(0,1)$-fibration over $\tA$,
  \item for every $a \in \tA$, the functor on fibres $\tE_{a} \to \tB$
    is a $(0,1)$-fibration,
  \item for every morphism $a \to a'$ in $\tA$, the cocartesian
    transport functor $\tE_{a} \to \tE_{a'}$ is a morphism of
    cartesian-enriched functors to $\tB$,
  \item for every morphism $a \to a'$ in $\tA$, the cocartesian
    transport functor $\tE_{a} \to \tE_{a'}$ preserves cocartesian morphisms over $E$.
  \end{enumerate}
\end{variant}

\subsection{Straightening for lax transformations}\label{sec:strlaxtr}
We will now apply the description of Gray fibrations in the previous
subsection to prove a straightening equivalence for lax transformations.

\begin{defn}
  Let $\FIB_{(i,1)}^{\lax}(\tB)$ denote the full sub-\itcat{} of
  $\Cartenr{\tB}$ spanned by the $(i,1)$-fibrations for
  $i=0,1$. Similarly, we define $\FIB_{(i,0)}^{\lax}(\tB)$ as a full
  sub-\itcat{} of $\Cocartenr{\tB}$.
\end{defn}

\begin{propn}\label{propn:fiblaxtogray}
  The equivalence of \cref{eq:slicestr} restricts to a natural
  equivalence
  \[ \FUN(\tA, \FIB_{(0,1)}^{\lax}(\tB)) \simeq
    \GFIB'_{(0,1)}(\tA,\tB),\]
  where $\GFIB'_{(0,1)}(\tA,\tB)$ denotes the locally full
  sub-\itcat{} of $\FIB_{(0,1)}(\tA)_{/\tA \times \tB}$ whose objects
  are Gray $(0,1)$-fibrations and whose maps are the morphisms of
  $(0,1)$-fibrations over $\tA$ that are fibrewise morphisms of
  cartesian-enriched functors to $\tB$.
\end{propn}
\begin{proof}
  By definition, $\FIB_{(0,1)}^{\lax}(\tB)$ is a locally full
  subcategory of $\CATITsl{\tB}$, hence by \cref{cor:funprops} we know
  that $\FUN(\tA, \FIB_{(0,1)}^{\lax}(\tB))$ is a locally full
  subcategory of the \itcat{} $\FUN(\tA, \CATITsl{\tB})$. It thus corresponds under
  the equivalence \cref{eq:slicestr} to a locally full
  subcategory of $\FIB_{(0,1)}(\tA)_{/\tA \times \tB}$. Unwinding the conditions on objects and morphisms, we see that the objects are precisely those we characterized as Gray fibrations in \cref{propn:grayfib}, and the morphisms are the ones in $\GFIB'_{(0,1)}(\tA,\tB)$.
\end{proof}

\begin{observation}\label{obs:maptolaxgpd}
  The \itcats{} $\GFIB'_{(0,1)}(\tA,\tB)$ and $\GFIB_{(0,1)}(\tA,\tB)$
  have the same underlying \igpds{}, so we get natural equivalences
  \[ \Map(\tA, \FIB_{(0,1)}^{\lax}(\tB)) \simeq
    \GFIB'_{(0,1)}(\tA,\tB)^{\simeq} \simeq
    \GFIB_{(0,1)}(\tA,\tB)^{\simeq}.
  \]
\end{observation}

\begin{thm}\label{thm:fib01laxstr}
  For an \itcat{} $\tB$, there are equivalences
  \begin{align*}
    \FIB_{(0,1)}^{\lax}(\tB) & \simeq \FUN(\tB, \CATIT)^{\lax},\\
   \FIB_{(0,0)}^{\lax}(\tB) &  \simeq \FUN(\tB^{\co},\CATIT)^{\lax},\\
   \FIB_{(1,1)}^{\lax}(\tB) & \simeq \FUN(\tB^{\coop},\CATIT)^{\oplax}, \\
   \FIB_{(1,0)}^{\lax}(\tB) &  \simeq \FUN(\tB^{\op},\CATIT)^{\oplax},
                              & 
  \end{align*}
  given on objects by straightening. These equivalences are all contravariantly
  natural in $\tB$ with respect to pullback on the left and
  composition on the right.
\end{thm}
\begin{proof}
  For the case of $(0,1)$-fibrations, \cref{obs:maptolaxgpd} and
  \cref{cor:gfibtogtens} imply that
  we have natural equivalences
  \[
    \begin{split}
      \Map(\tA, \FIB_{(0,1)}^{\lax}(\tB)) & \simeq
                                      \GFIB_{(0,1)}(\tA,\tB)^{\simeq} \\
                                    & \simeq  \Map(\tA \otimes \tB, \CATIT) \\
                                    & \simeq \Map(\tA, \FUN(\tB,\CATIT)^{\lax})
    \end{split}
  \]
  for any \itcat{} $\tA$. By the Yoneda lemma, this implies the
  required equivalence of \itcats{}. For the other cases, we observe
  that we have analogues of the equivalences \cref{eq:fibop} for
  $\FIB_{(i,j)}^{\lax}(\tB)$, and then the equivalence we just proved for
  $(0,1)$-fibrations gives analogues of \cref{eq:strvariants}. 
\end{proof}

\begin{remark}
  If cartesian-enriched morphisms between $(0,1)$-fibrations over
  $\tB$ correspond to lax transformations, it is natural to ask what
  \emph{arbitrary} functors between $(0,1)$-fibrations over $\tB$
  correspond to. The \itcat{} $\CATIT$ is the underlying \itcat{} of
  the \emph{$(\infty,3)$-category} $\boldcatname{CAT}_{(\infty,2)}$ of
  \itcats{}. Moreover, if $\tB$ is an \itcat{}, then the Gray tensor
  product $[1] \otimes \tB$ we have considered is a truncation of an
  $(\infty,3)$-category $[1] \otimes^{+} \tB$.\footnote{In general,
    there is a Gray tensor product that takes an $(\infty,n)$-category
    and an $(\infty,m)$-category to an $(\infty,n+m)$-category.} This
  means that we can consider ``even more lax'' transformations between
  functors $\tB \to \CATIT$ in the form of functors of
  $(\infty,3)$-categories
  $[1] \otimes^{+} \tB \to \boldcatname{CAT}_{(\infty,2)}$. We expect
  that there is a straightening equivalence under which these
  correspond to arbitrary functors among $(0,1)$-fibrations over
  $\tB$.
\end{remark}

\begin{defn}\label{def:elaxfibcat}
  Let $(\tB,E)$ be a marked \itcat{} and let $\FIB_{(i,j)}^{\elax}(\tB)$ denote the wide locally full subcategory of
  $\FIB_{(i,j)}^{\lax}(\tB)$ where the morphisms are required to preserve $i$-cartesian morphisms over $E$. We denote by $\Fun^{E\dcart}_{/\tB}(-,-)$ the mapping $\infty$-category functor for $\FIB_{(1,j)}^{\elax}(\tB)$ and similarly $\Fun^{E\dcoc}_{/\tB}(-,-)$ for $\FIB_{(0,j)}^{\elax}(\tB)$.
\end{defn}

We can also consider a marked variant of \cref{thm:fib01laxstr}:
\begin{propn}\label{propn:markedlaxstr}
  Let $(\tB, E)$ be a marked \itcat{}. Then there are equivalences
  \begin{align*}
    \FIB_{(0,1)}^{\elax}(\tB) & \simeq \FUN(\tB, \CATIT)^{\elax},\\
   \FIB_{(0,0)}^{\elax}(\tB) &  \simeq \FUN(\tB^{\co},\CATIT)^{\elax},\\
   \FIB_{(1,1)}^{\elax}(\tB) & \simeq \FUN(\tB^{\coop},\CATIT)^{\eoplax},\\
   \FIB_{(1,0)}^{\elax}(\tB) &  \simeq \FUN(\tB^{\op},\CATIT)^{\eoplax},                        
  \end{align*}
  given on objects by straightening. These equivalences are all
  contravariantly natural in $\tB$ with respect to pullback on the
  left and composition on the right.
\end{propn}
\begin{proof}
  The equivalence of \cref{propn:fiblaxtogray} restricts to a natural equivalence between $\Map(\tA, \FIB_{(0,1)}^{\elax}(\tB))$ and the \igpd{} of marked Gray fibrations.   
  Using \cref{var:markedgray}, we then
 have natural equivalences
  \[
    \begin{split}
      \Map(\tA, \FIB_{(0,1)}^{\elax}(\tB)) & \simeq
                                        \Map(\tA \otimes_{\natural,E} \tB, \CATIT) \\
                                    & \simeq \Map(\tA, \FUN(\tB,\CATIT)^{\elax}),
    \end{split}
  \]
  as required. The other variances follows as in the fully lax
  case.
\end{proof}

\begin{remark}
  Taking the maximal marking in \cref{propn:markedlaxstr}, we obtain a natural equivalence
  \[ \FIB_{(0,1)}(\tB) \simeq \FUN(\tB, \CATIT).\] Nuiten's uniqueness
  theorem for straightening equivalences \cite{Nuiten}*{Theorem 6.20}
  implies that this agrees with the usual straightening equivalence of
  \cref{thm:str01}.
\end{remark}

We can also restrict these equivalences to functors to $\CATI$:
\begin{defn}
  Let $\FIB_{(i,j)}^{(1),\lax}(\tB)$ be the full sub-\itcat{} of
  $\FIB_{(i,j)}^{\lax}(\tB)$ spanned by the $1$-fibred
  $(i,j)$-fibrations. Note that by \cref{propn:1fibredcond} and
  \cref{obs:lenrsubcat}, this is a full sub-\itcat{} of $\CATITsl{\tB}$. Similarly, we define  $\FIB_{(i,j)}^{(1),\elax}(\tB)$ for a marked \itcat{} $(\tB, E)$.
\end{defn}

\begin{cor}
  Let $(\tB, E)$ be a marked \itcat{}. There are natural equivalences
  \begin{align*}
    \FIB_{(0,1)}^{(1),\elax}(\tB) & \simeq \FUN(\tB, \CATI)^{\elax},\\
   \FIB_{(0,0)}^{(1),\elax}(\tB) &  \simeq \FUN(\tB^{\co},\CATI)^{\elax},\\
   \FIB_{(1,1)}^{(1),\elax}(\tB) & \simeq \FUN(\tB^{\coop},\CATI)^{\eoplax},\\
   \FIB_{(1,0)}^{(1),\elax}(\tB) &  \simeq \FUN(\tB^{\op},\CATI)^{\eoplax},                    
  \end{align*}
  given on objects by straightening.
\end{cor}
\begin{proof}
  We prove the case of $(0,1)$-fibrations. Since $\CATI$ is a full sub-\itcat{} of $\CATIT$, we get from \cref{cor:elaxfunprops} that
  $\FUN(\tB, \CATI)^{\elax}$ is the full subcategory of $\FUN(\tB,
  \CATIT)^{\elax}$ spanned by the functors that take values in
  $\CATI$. On the other side of \ref{propn:markedlaxstr}, this
  condition precisely amounts to asking for $1$-fibred
  $(0,1)$-fibrations.
\end{proof}

\subsection{Gray bifibrations}\label{sec:graybifib}
Our goal in this subsection is to describe the bivariant version of Gray
fibrations, which will play an important role when we discuss mates in
the next section. In particular, we prove a key lemma about certain
functors of \itcats{} $\tE \to \tc{A} \times \tc{B}$ such that the
composite to $\tc{A}$ is an $(i,j)$-fibration and that to $\tc{B}$ is
a $(1-i,1-j)$-fibration. This generalizes results from
\cite{HHLN1}*{2.3.3}.

\begin{defn}\label{defn:grayfibration}
  Let $\tB$ be an $(\infty,2)$-category. We say that a functor $p \colon \tE \to \tA \times \tB$ is a
  \emph{Gray $(i,j)$-$(k,l)$-fibration} if
  \begin{enumerate}
  \item $\tE \to \tA$ is an $(i,j)$-fibration,
  \item $\tE \to \tA \times \tB$ is a morphism of $(i,j)$-fibrations,
  \item $\tE_{a} \to \tB$ is a $(k,l)$-fibration for every $a \in
    \tA$,
  \item the $i$-cartesian transport functor $\tE_{a_i} \to \tE_{a_{1-i}}$ for
    every map $a_0\to a_1$ in $\tA$ is a morphism of
    $l$-cartesian-enriched functors over $\tB$.
  \end{enumerate}
  If $(k,l)=(1-i,1-j)$ we say that $p$ is a \emph{Gray
    $(i,j)$-bifibration}. We say a commutative triangle
  \[
    \begin{tikzcd}
     \tE \ar[rr, "f"] \ar[dr, "\pi"'] & & \tE' \ar[dl, "\pi'"] \\
      & \tA \times \tB
    \end{tikzcd}
  \]
  where $\pi$ and $\pi'$ are Gray $(i,j)$-$(k,l)$-fibrations, is a
  \emph{morphism of Gray $(i,j)$-$(k,l)$-fibrations} if $f$ is a
  morphism of $j$-cartesian-enriched functors from $\pi_{\tA}$ to
  $\pi'_{\tA}$, and for all $a \in \tA$ the functor on fibres
  $f_{a} \colon \tE_{a} \to \tE'_{a}$ is a morphism of
  $l$-cartesian-enriched functors over $\tB$. We write $\GFIB_{(i,j);(k,l)}(\tA,\tB)$ for the locally
  full sub-\itcat{} of $\CATITsl{\tA \times \tB}$ spanned by the Gray
  $(i,j)$-$(k,l)$-fibrations and the morphisms thereof. We abbreviate
  $\GBIFIB_{(i,j)}(\tA,\tB) := \GFIB_{(i,j);(1-i,1-j)}(\tA,\tB)$.
\end{defn}

\begin{defn}\label{defn:tfibredgrayfibration}
  We say that an $(i,j)$-$(k,l)$-fibration $\tE \to \tA \times \tB$
  is $1$-fibred if the fibres $\tE_{(a,b)}$ are $\infty$-categories
  for every $(a,b) \in \tA \times \tB$, and $0$-fibred if they are
  $\infty$-groupoids. We write $\GFIB^{(t)}_{(i,j);(k,l)}(\tA,\tB)$
  for the full sub-2-category of $\GFIB_{(i,j);(k,l)}(\tA,\tB)$
  spanned by the $t$-fibred Gray
  $(i,j)$-$(k,l)$-fibration.
\end{defn}

\begin{defn}\label{defn:EMbifibration}
  Let $(\tA,E)$ and $(\tB,M)$ be marked $(\infty,2)$-categories. We
  define
  \[ \GBIFIB^{(E,M)}_{(i,j)}(\tA,\tB) \subseteq
    \GBIFIB_{(i,j)}(\tA,\tB) \]
  to be the locally full sub-\itcat{} given as follows:
   \begin{itemize}
     \item The objects are Gray $(i,j)$-bifibrations such that the $i$-cartesian transport functor $\tE_{a_i} \to \tE_{a_{1-i}}$ for
    every map $a_0\to a_1$ in $\tA$ preserves $(1-i)$-cartesian morphisms lying over $M$.
    \item The morphisms are those morphisms of Gray $(i,j)$-bifibrations $\tE \to \tE^{\prime}$ that preserve $i$-cartesian 1-morphisms lying over $E$ and such that for every $a \in \tA$ the induced functor $\tE_{a} \to \tE_{a}^{\prime}$ preserves $(1-i)$-cartesian morphisms lying over $M$.
   \end{itemize}
   We note that if $E,M$ are the respective collections of
   equivalences in $\tA$ and $\tB$, we recover the notion of a Gray
   $(i,j)$-bifibration.
\end{defn}

\begin{ex}\label{ex:oplaxarrowbifib}
  The oplax arrow category comes equipped with a functor \[\ev_i: \ARopl(\tA) \to \tA,\] for $i=0,1$, induced by the map that selects the object $i$ in $[1]$. It follows from \cite[Theorem 3.0.7]{GagnaHarpazLanariFib} that $\ev_0$ is a $(1,0)$-fibration and $\ev_1$ is a $(0,1)$-fibration.  We observe that for every object $a \in \tA$ we have a pullback square
  \[
    \begin{tikzcd}
      \tA_{a /}^{\oplax} \arrow[r] \arrow[d] & \ARopl(\tA) \arrow[d,"\ev_0"] \\
      {[}0{]} \arrow[r] & \tA
    \end{tikzcd}
  \]
  and that for every $a \to a'$ the cartesian transport functor
  \[
    \begin{tikzcd}
      \tA_{a' /}^{\oplax} \arrow[dr] \arrow[rr] & & \tA_{a /}^{\oplax} \arrow[dl] \\
      & \tA  & 
    \end{tikzcd}
  \]
  is a morphism of $(0,1)$-fibrations. We conclude that $\ARopl(\tA) \to \tA \times \tA$ is a Gray $(1,0)$-bifibration which further factors through $\GBIFIB^{(\sharp,\sharp)}_{(1,0)}(\tA,\tA)$. Dually it follows that $\ARlax(\tA) \to \tA \times \tA$ is a Gray $(1,1)$-bifibration which factors through $\GBIFIB^{(\sharp,\sharp)}_{(1,1)}(\tA,\tA)$.
\end{ex}

\begin{propn}\label{propn:gbifibEMst}
  Let $(\tA,E)$ and $(\tB,M)$ be marked $(\infty,2)$-categories. Then
  we have natural equivalences of \itcats{}
  \begin{align}
    \label{eq:EMbifib}
    \GBIFIB^{(E,M)}_{(0,1)}(\tA,\tB)  & \simeq
                                        \FUN(\tA,\FUN(\tB^\op,\CATIT)^{M\doplax})^{\elax} \\
    \GBIFIB^{(E,M)}_{(1,0)}(\tA,\tB) & \simeq
                                       \FUN(\tA^\op,\FUN(\tB,\CATIT)^{M\dlax})^{\eoplax}
                                       \nonumber \\
    \GBIFIB^{(E,M)}_{(0,0)}(\tA,\tB) & \simeq
                                       \FUN(\tA^\co,\FUN(\tB^{\coop},\CATIT)^{M\doplax})^{\elax} \nonumber \\
    \GBIFIB^{(E,M)}_{(1,1)}(\tA,\tB) & \simeq  \FUN(\tA^\coop,\FUN(\tB^\co,\CATIT)^{M\dlax})^{\eoplax}\nonumber 
  \end{align}
\end{propn}
\begin{proof}
  We verify only \cref{eq:EMbifib} since the rest is dual.  First, we
  observe that by definition $\GBIFIB^{(E,M)}_{(0,1)}(\tA,\tB)$ is a
  locally full sub-\itcat{} of
  $\FIB^{\elax}_{(0,1)}(\tA)_{/\tA \times \tB}$. Here the projection
  $\tA \times \tB \to \tA$ straightens to the constant functor with
  value $\tB$, so that \cref{propn:markedlaxstr} identifies
  $\GBIFIB^{(E,M)}_{(0,1)}(\tA,\tB)$ with a locally full sub-\itcat{}
  of $\FUN(\tA, \CATIT)^{\elax}_{/\underline{\tB}}$. Moreover, we know
  from \cref{cor:takingslicesout} that we can identify
  $\FUN(\tA, \CATITsl{\tB})^{\elax}$ with the full sub-\itcat{} of
  $\FUN(\tA, \CATIT)^{\elax}_{/\underline{\tB}}$ spanned by the strong
  natural transformations to $\underline{\tB}$; this contains the
  image of $\GBIFIB^{(E,M)}_{(0,1)}(\tA,\tB)$ since a Gray
  $(0,1)$-bifibration has to give a morphism of $(0,1)$-fibrations to
  $\tA \times \tB$, which straightens to a strong natural
  transformation. We claim that the locally full sub-\itcat{} of
  $\FUN(\tA, \CATITsl{\tB})^{\elax}$ that corresponds to
  $\GBIFIB^{(E,M)}_{(0,1)}(\tA,\tB)$ is precisely
  $\FUN(\tA, \FIB_{(1,0)}^{M\dlax}(\tB))^{\elax}$ (which is a locally
  full sub-\itcat{} by \cref{thm:funlaxprops}). Indeed, the objects of
  this locally full sub-\itcat{} are the functors
  $F \colon \tA \to \CATITsl{\tB}$ such that
  \begin{itemize}
  \item for every object $a \in \tA$, $F(a)$ is a $(1,0)$-fibration,
  \item and for every morphism $f \colon a \to a'$ in $\tA$, $F(a) \to
    F(a')$ is cocartesian-enriched and preserves cartesian morphisms
    over $M$,
  \end{itemize}
  while the morphisms are the $E$-lax natural transformations
  $\alpha \colon F \to G$ among such functors such that for every object $a \in
  \tA$, the component $\alpha_{a} \colon F(a) \to G(a)$ is 
  cocartesian-enriched and preserves cartesian morphisms
  over $M$. These are precisely the conditions on the fibres over
  $\tA$ for the objects and morphisms of
  $\GBIFIB^{(E,M)}_{(0,1)}(\tA,\tB)$. Now \cref{propn:markedlaxstr}
  gives a natural equivalence
  \[\FUN(\tA, \FIB_{(1,0)}^{M\dlax}(\tB))^{\elax} \simeq \FUN(\tA,
    \FUN(\tB^{\op}, \CATIT)^{M\doplax})^{\elax},\]
 which completes the proof.
\end{proof}

\begin{cor}\label{cor:ortho}
  Let $\tA^{\sharp}$ and $\tB^{\sharp}$ be maximally marked $(\infty,2)$-categories. Then
  we have natural equivalences of \itcats{}
  \begin{align*}
    \GBIFIB^{(\sharp,\sharp)}_{(0,1)}(\tA,\tB)  & \simeq
                                        \FUN(\tA \times \tB^{\op},\CATIT)  \\
    \GBIFIB^{(\sharp,\sharp)}_{(1,0)}(\tA,\tB)  & \simeq
                                        \FUN(\tA^\op \times \tB,\CATIT)  \\
     \GBIFIB^{(\sharp,\sharp)}_{(0,0)}(\tA,\tB)  & \simeq
                                        \FUN(\tA^{\co} \times \tB^{\coop},\CATIT)  \\
     \GBIFIB^{(\sharp,\sharp)}_{(1,1)}(\tA,\tB)  & \simeq
                                        \FUN(\tA^{\coop} \times \tB^{\co},\CATIT). \qed   \\     
  \end{align*}
\end{cor}

\begin{propn}\label{propn:graybifib}
  A functor $(\pi_\tA,\pi_\tB) \colon \tE \to \tA \times \tB$ is a
  Gray $(i,j)$-bifibration \IFF{} $(\pi_\tB,\pi_\tA)$ is a Gray
  $(1-i,1-j)$-bifibration.
\end{propn}

The following is a key observation for the proof:
\begin{lemma}\label{lem:fibrecocartcrit}
  Suppose we have a commutative triangle
      \[
  \begin{tikzcd}
    \tE \arrow[dr,swap,"\pi_{\tc{A}}"] \arrow{rr}{\pi} & & \tc{A} \times \tc{B} \arrow{dl}{\pr_{\tc{A}}} \\
     & \tc{A}
  \end{tikzcd}
\]
where $\pi_{\tc{A}}$ is an $(i,j)$-fibration and $\pi$ is a morphism of $(i,j)$-fibrations. Suppose $\bar{f} \colon x \to y$ is a morphism in $\tE_{a}$ for some $a \in \tc{A}$ that is $(1-i)$-cartesian for $\pi_a$. Then $\phi$ is $(1-i)$-cartesian for $\pi_{\tc{B}}$.
\end{lemma}
\begin{proof}
  We assume $(i,j) = (0,1)$ without loss of generality. Suppose
  $\bar{f} \colon x \to y$ lies $f \colon b \to b'$ in $\tB$ (and over
  $\id_{a}$ in $\tA$). Given $z \in \tE_{(a',b'')}$, we consider the commutative diagram
  \[
    \begin{tikzcd}[column sep=large]
      \tE(z,x) \arrow[r,"\bar{f}_{*}"] \arrow[d] & \tE(z,y) \arrow[d] \\
      \tc{A}(a',a) \times \tc{B}(b'', b) \arrow[r,"\id \times f_{*}"] \arrow[d] & \tc{A}(a',a) \times \tc{B}(b'',b') \arrow[d] \\
       \tc{B}(b'', b) \arrow[r,"f_{*}"] & \tc{B}(b'', b').
    \end{tikzcd}
  \]
We wish to show that the outer square is a pullback square of
\icats{}. Since the bottom square is clearly a
pullback, it suffices to show that so is the upper square. This is a commutative square of
cartesian fibrations over $\tc{A}(a',a)$, so it suffices to show that
the square on fibres over each $g \colon a' \to a$ is a pullback.
To see this, we pick a
$\pi_{\tA}$-cocartesian lift $\bar{g} \colon z \to g_{!}z$;  this lies
over an equivalence in $\tB$, which we take to be $\id_{b''}$ to
simplify the notation. Using
\cref{obs:cocartmormapfib} we can then identify the square of fibres over $g$ with the square
\[
  \begin{tikzcd}
    \tE_{a}(g_{!}z,x) \arrow[r] \arrow[d] & \tE_a(g_{!}z,y) \arrow[d] \\
    \tc{B}(b'', b) \arrow[r] & \tc{B}(b'', b'),
  \end{tikzcd}
\]
which is a pullback diagram since $\bar{f}$ is by assumption
$\pi_{a}$-cartesian.
\end{proof}

\begin{observation}\label{obs:mors of curved ortho}
  Recall from \cite{HHLN1}*{Proposition 2.3.3} that a functor of
  \icats{} $\pi \colon \oE \to \oA \times \oB$ is a \emph{curved
    orthofibration}, \ie{} a Gray $1$-bifibration among \icats{}, if
  either of the following equivalent conditions hold:
  \begin{itemize}
  \item $\pi_{\oA}$ is a cartesian fibration, $\pi_{\oA}$-cartesian
    morphisms lie over equivalences in $\oB$, and for every
    $a \in \oA$ the functor on fibres $\pi_{a} \colon \oE_{a} \to \oB$
    is a cocartesian fibration.
  \item $\pi_{\oB}$ is a cocartesian fibration,
    $\pi_{\oB}$-cocartesian morphisms lie over equivalences in $\oA$,
    and for every $b \in \oB$ the functor on fibres
    $\pi_{b} \colon \oE_{b} \to \oA$ is a cartesian fibration.
  \end{itemize}
  Moreover:
  \begin{itemize}
  \item a morphism in $\oE$ is $\pi_{\oA}$-cartesian \IFF{} it lies
    over an equivalence $b \isoto b'$ in $\oB$ and corresponds to a
    $\pi_{b}$-cartesian morphism in $\oE_{b}$,
  \item a morphism in $\oE$ is $\pi_{\oB}$-cocartesian \IFF{} it lies
    over an equivalence $a \isoto a'$ in $\oA$ and corresponds to a
    $\pi_{a}$-cocartesian morphism in $\oE_{a}$.
  \end{itemize}
  It follows that given a commutative square
  \[
    \begin{tikzcd}
     \oE \ar[r, "h"] \ar[d, "p"'] & \oE'  \ar[d, "p'"] \\
     \oA \times \oB \ar[r, "f \times g"'] & \oA' \times \oB'
    \end{tikzcd}
  \]
  where $p$ and $p'$ are curved orthofibrations, we have that
  \begin{itemize}
  \item $h$ takes $p_{\oA}$-cartesian morphisms to
    $p'_{\oA'}$-cartesian ones \IFF{} for every $b \in \oB$, $h_{b}
    \colon \oE_{b} \to \oE'_{g(b)}$ takes $p_{b}$-cartesian morphisms
    to $p'_{g(b)}$-cartesian ones.
  \item $h$ takes $p_{\oB}$-cocartesian morphisms to
    $p'_{\oB'}$-cocartesian ones \IFF{} for every $a \in \oA$, $h_{a}
    \colon \oE_{a} \to \oE'_{f(a)}$ takes $p_{a}$-cocartesian morphisms
    to $p'_{f(a)}$-cocartesian ones.    
  \end{itemize}
\end{observation}

\begin{proof}[Proof of \cref{propn:graybifib}]
 We assume $(i,j)=(0,1)$; the remaining cases are proved in exactly the
 same way. We begin the proof by showing that $(\pi_{\tB},\pi_{\tA})$
 is a Gray $(1,0)$-bifibration. First of all, we check that
 $\pi_{\tB}\colon \tE \to \tB$ is a
 cocartesian-enriched functor. We claim that
 for every $x,y \in \tE$ over $(a,b)$ and $(a',b')$ in $\tA \times \tB$, the induced functor 
  \[
    \tE(x,y)\xlongrightarrow{} \tc{B}(b,b')
  \]
  is a cocartesian fibration. Our assumptions imply that we have a
  morphism of cartesian fibrations
   \[
     \begin{tikzcd}
      \tE(x,y) \arrow[rr] \arrow[dr] & &  \tc{A}(a,a') \times \tc{B}(b,b') \arrow[dl] \\
      & \tc{A}(a,a').
    \end{tikzcd}
  \]
  In virtue of \cite{HHLN1}*{Proposition 2.3.3}, it will be enough to
  show that for every $f \colon a \to a'$, the induced map
  $\tE(x,y)_f \to \tc{B}(b,b')$ is a cocartesian fibration. Picking a
  $\pi_{\tA}$-cocartesian lift $x \to f_{!}x$ of $f$ we can identify
  the map on fibres with the functor
  \[
    \tE_{a'}(f_{!}x,y) \simeq \tE(x,y)_{f} \to \tc{B}(b,b'),
  \]
  which is a cocartesian fibration by assumption. 
  Now, we wish to show that for every $x,y,z$, the composition functor
  \[
    \tE(x,y)\times \tE(y,z) \xlongrightarrow{} \tE(x,z)
  \]
  preserves cocartesian 2-morphisms. This fits in a commutative square
  \[
    \begin{tikzcd}
     \tE(x,y) \times \tE(y,z) \ar[r] \ar[d] & \tE(x,z)  \ar[d] \\
     \tA(a,a') \times \tB(b,b') \times \tA(a',a'') \times \tB(b',b'')
     \ar[r] & \tA(a,a'') \times \tB(b,b'')
    \end{tikzcd}
  \]
  where we know the vertical maps are curved orthofibrations (as these
  are easily seen to be closed under cartesian products). From
  \cref{obs:mors of curved ortho} we conclude that it suffices to show
  that for $f \colon a \to a'$ and $g \colon a' \to a''$ the functor
  on fibres
  \[ \tE(x,y)_{f} \times \tE(y,z)_{g} \to \tE(x,z)_{gf} \]
  preserves cocartesian morphisms over the composition functor from
  $\tB$. Now \cref{obs:fibre comp cocart} implies that if we choose
  $\pi_{\tA}$-cocartesian morphisms $x \to f_{!}x$ over $f$ and $y \to
  g_{!}y$ over $g$, then we can identify this with the composite
  \[ \tE_{a'}(f_{!}x,y) \times \tE_{a''}(g_{!}y,z) \to
    \tE_{a''}(g_{!}f_{!}x, g_{!}y) \times \tE_{a''}(g_{!}y,z) \to
    \tE_{a''}(g_{!}f_{!}x, z)\]
  where the first map comes from the cocartesian transport functor
  over $g$
  and the second is composition in $\tE_{a''}$. The desired conclusion
  then follows by combining the assumption that $g_{!} \colon \tE_{a'} \to
  \tE_{a''}$ is a morphism of cocartesian-enriched
  functors with $\tE_{a''} \to \tB$ being a $(1,0)$-fibration.

Let us next verify that $\tE$ has
$\pi_{\tB}$-cartesian lifts of all morphisms in $\tB$: Given 
$f\colon b \to b'$ in $\tB$ and $y$ in $\tE$ over $(a,b')$, we know
that there is a $\pi_{a}$-cartesian lift of $f$ at $y$, and this is
then also $\pi_{\tB}$-cartesian by \cref{lem:fibrecocartcrit}.

It remains to show that for every $b \in \tB$ 
the induced functor $\tE_{b} \to \tA$ is a $(0,1)$-fibration, and that 
for every $u \colon b \to b'$ the transport functor $\tE_{b'} \to \tE_{b}$
is a morphism of cartesian-enriched functors. The first claim follows from 
the assumption that $\pi_{\tA} \colon \tE \to \tA$ is a
$(0,1)$-fibration, since we have a pullback square
\[
  \begin{tikzcd}
   \tE_{b} \ar[r] \ar[d] & \tA \times \{b\}  \ar[d] \\
   \tE \ar[r] & \tA \times \tB
  \end{tikzcd}
\]
of $(0,1)$-fibrations over $\tA$, where the two maps to $\tA \times
\tB$ are morphisms of $(0,1)$-fibrations. For the final claim
we consider $x,y$ in $\tE_{b'}$ over $a,a' \in \tA$, and pick $\pi_{\tB}$-cartesian
morphisms $\bar{u}_{x} \colon u^{*}x \to x$ and $\bar{u}_{y} \colon
u^{*}y \to y$ over $u$. Then we have a commutative diagram
\[\begin{tikzcd}
  {\tE_{b'}(x,y)} & {\tE(u^{*}x,y)_{u}} & {\tE_{b}(u^{*}x,u^{*}y)} \\
  & {\tA(a,a')}
  \arrow["\bar{u}_{x}^{*}", from=1-1, to=1-2]
  \arrow[from=1-1, to=2-2]
  \arrow[from=1-2, to=2-2]
  \arrow["\bar{u}_{y,*}"', "\sim", from=1-3, to=1-2]
  \arrow[from=1-3, to=2-2]
\end{tikzcd}\]
where the right horizontal functor is an equivalence since
$\bar{u}_{y}$ is $\pi_{\tB}$-cartesian.
Note that both of these functors are  morphisms of cartesian fibrations over $\tA(a,a')$, so 
we conclude that so is $\bar{u}_{y,*}^{-1} \bar{u}_{x}^{*}$.
\end{proof}

\begin{defn}
  We say that a commutative triangle
  \[
    \begin{tikzcd}
     \tE \ar[rr, "F"] \ar[dr, "\pi"'] & & \tE' \ar[dl, "\pi'"] \\
      & \tA \times \tB
    \end{tikzcd}
  \]
  where $\pi$ and $\pi'$ are Gray $(i,j)$-bifibrations, is a
  \emph{strong morphism of Gray (i,j)-bifibrations} if
  \begin{itemize}
   \item $F$ is a morphism of $(i,j)$-fibrations over $\tA$,
   \item for every $a \in \tA$, the functor $\tE_{a} \to \tE'_{a}$
     is a morphism of $(1-i,1-j)$-fibrations over $\tB$.
   \end{itemize}
   We write $\GBIFIB_{(i,j)}^{\mathrm{s}}(\tA, \tB)$ for the wide
   locally full subcategory of $\GBIFIB_{(i,j)}(\tA, \tB)$ containing
   the strong morphisms of Gray $(i,j)$-bifibrations.
\end{defn}

\begin{propn}\label{propn:graybifibmor}
  A commutative triangle
\[\begin{tikzcd}
  \tE && \tE' \\
  & {\tA \times \tB},
  \arrow["F", from=1-1, to=1-3]
  \arrow["\pi"', from=1-1, to=2-2]
  \arrow["\pi'", from=1-3, to=2-2]
\end{tikzcd}\] where $\pi$ and $\pi'$ are Gray $(i,j)$-bifibrations,
is a strong morphism of Gray $(i,j)$-bifibrations over $\tA \times
\tB$ \IFF{} it is a strong morphism of Gray $(1-i,1-j)$-bifibrations
over $\tB \times \tA$.
In other words,
we have an equivalence of \itcats{}
\[ \GBIFIB^{\mathrm{s}}_{(i,j)}(\tA,\tB) \simeq
  \GBIFIB^{\mathrm{s}}_{(1-i,1-j)}(\tB,\tA),\]
as sub-\itcats{} of $\CATITsl{\tA \times \tB}$, and so an equivalence
of \igpds{}
\[ \GBIFIB_{(i,j)}(\tA,\tB)^{\simeq} \simeq
  \GBIFIB_{(1-i,1-j)}(\tB,\tA)^{\simeq},\]
since every equivalence is a strong morphism.
\end{propn}
\begin{proof}
  We assume without loss of generality that $(i,j) = (0,1)$, and that
  $F$ is a strong morphism of Gray $(0,1)$-bifibrations;
  the other cases are proved similarly.  It is clear that the induced
  functors on fibres over $\tB$ are morphisms of
  $(0,1)$-fibrations,since these are pullbacks of $F$. Moreover, it
  follows from \cref{lem:fibrecocartcrit} that $F$ preserves cartesian
  morphisms over $\tB$, so it remains to show that $F$ is a morphism
  of cocartesian-enriched functors over $\tB$. For $x,y$ in $\tE$ over
  $(a,b)$ and $(a',b')$, consider the commutative triangle
\[
  \begin{tikzcd}
   \tE(x,y) \ar[rr] \ar[dr] & & \tE'(Fx,Fy) \ar[dl] \\
    & \tA(a,a') \times \tB(b,b');
  \end{tikzcd}
\]
we want to show that the horizontal functor preserves cocartesian
morphisms over $\tB(b,b')$. Here both downward functors are
curved orthofibrations, so by \cref{obs:mors of curved ortho} this is
equivalent to the map
\[
  \begin{tikzcd}
   \tE(x,y)_{f} \ar[rr] \ar[dr] & & \tE'(Fx,Fy)_{f} \ar[dl] \\
    &  \tB(b,b');
  \end{tikzcd}
\]
on fibres over $f \colon a \to a'$ preserving cocartesian
morphisms. Now choose a $\pi_{\tA}$-cocartesian morphism $\bar{f}
\colon x \to
f_{!}x$ over $f$; then $F(\bar{f})$ is $\pi'_{\tA}$-cocartesian, so we can identify this triangle with 
\[
  \begin{tikzcd}
   \tE_{a'}(f_{!}x,y) \ar[rr] \ar[dr] & & \tE'_{a'}(F(f_{!}x),Fy)
   \ar[dl] \\
    &  \tB(b,b'),
  \end{tikzcd}
\]
where the horizontal functor preserves cocartesian morphisms by
assumption. 
\end{proof}

\begin{propn}\label{prop:mapspacegbifib}
  Let $\tc{X}$ and $\tc{A}$ be $(\infty,2)$-categories. Then there exist natural equivalences
  \begin{enumerate}
    \item $\Map(\tc{X}, \FUN(\tc{A},\CAT_{(\infty,2)})^{\lax}) \simeq \GBIFIB_{(1,0)}(\tX^{\op},\tA)^{\simeq} \simeq
    \GBIFIB_{(0,1)}(\tA,\tX^{\op})^{\simeq}$.
    \item $\Map(\tc{X}, \FUN(\tc{A},\CATIT)^{\oplax}) \simeq
    \GBIFIB_{(0,0)}(\tX^{\co},\tA^{\coop})^{\simeq} \simeq
    \GBIFIB_{(1,1)}(\tA^{\coop},\tX^{\co})^{\simeq}$.
  \end{enumerate}
\end{propn}
\begin{proof}
  To show $(1)$ we note that we have natural equivalences
  \[
    \begin{split}
\Map(\tc{X}, \FUN(\tc{A},\CATI)^{\lax}) & \simeq \Map(\tc{X},
    \FIB^{\lax}_{0,1}(\tA)) \\ &  \simeq
    \GBIFIB_{(1,0)}(\tX^{\op},\tA)^{\simeq} \\ &  \simeq
    \GBIFIB_{(0,1)}(\tA,\tX^{\op})^{\simeq},
    \end{split}
  \]
    where the first two equivalences follow from by
    \cref{propn:gbifibEMst} and the last one is a consequence of
    \cref{propn:graybifibmor}. To show $(2)$, we observe that we similarly have equivalences
    \[
      \begin{split}
\Map(\tc{X}, \FUN(\tc{A},\CATI)^{\oplax}) & \simeq \Map(\tc{X},
    \FIB^{\lax}_{1,1}(\tA^{\coop})) \\ & \simeq
                                      \GBIFIB_{(0,0)}(\tX^{\co},\tA^{\coop})^{\simeq} \\
        & \simeq
    \GBIFIB_{(1,1)}(\tA^{\coop},\tX^{\co})^{\simeq},
      \end{split}
    \]
    which concludes the proof.
\end{proof}

\begin{lemma}\label{lem:2conservativegray}
  Let $p \colon \tE \to \tA \times \tB$ be a Gray $(i,j)$-bifibration
  such that for every $a \in \tA$ and $b \in \tB$ the fibre
  $\tE_{(a,b)}$ is an $\infty$-category. Then  $p$ is conservative on 2-morphisms.
\end{lemma}
\begin{proof}
  We assume without loss of generality that $(i,j)=(0,1)$.
  By \cref{lem:cons on 2mor}, we need to show that for any $x,y \in \tE$ over $(a,b)$ and $(a',b')$, the fibres of
  \[ \tE(x,y) \to \tA(a,a') \times \tB(b,b')\] are \igpds{}.
  For $f \colon a \to a'$ and $g \colon b \to b'$ we can first pick a
  $\pi_{\tA}$-cocartesian morphism $x \to x'$ over $f$, and identify
  $\tE(x,y)_{f}$ with
  \[ \tE_{a'}(x',y) \to \tB(b,b').\] Next we can pick a
  $\pi_{a'}$-cartesian morphism $y' \to y$ over $g$, and get the
  identification
  \[ \tE(x,y)_{(f,g)} \simeq \tE_{a'}(x',y)_{g} \simeq
    \tE_{(a',b)}(x',y');\]
  this is an \igpd{} since $\tE_{(a',b)}$ is an \icat{}.
\end{proof}



\section{Adjunctions and lax transformations}\label{sec:adjs}
In this section we will apply our results on lax transformations to
study adjunctions in \itcats{}. We first see that we can recognize
adjunctions in $\CATIT$ in terms of fibrations in \S\ref{sec:adjfib},
and then generalize some results of Lurie on relative adjunctions to
\itcats{} in \S\ref{sec:reladj}. There we also study adjunctions in
\itcats{} of (op)lax transformations and prove
\cref{thm:laxtradj}. Finally, we prove \cref{thm:matecorr}, the mate
correspondence for \itcats{}, in \S\ref{sec:mates}.

\subsection{Adjunctions and fibrations}\label{sec:adjfib}
Recall that a functor of \icats{} $F \colon \oC \to \oD$ is a left
adjoint \IFF{} the corresponding cocartesian fibration $\oF \to [1]$
is also a cartesian fibration, with the cartesian straightening
corresponding to the right adjoint. Our goal in this subsection is to
prove the analogous statement for \itcats{}, as well as some other
basic results on adjunctions of \itcats{}.

\begin{defn}\label{defn:adjunction}
  We say that a morphism $f \colon a \to b$ in some \itcat{} $\tC$ is a
  \emph{left adjoint} if there exist a 1-morphism $g \colon b \to a$
  and 2-morphisms $\eta \colon \id_{a} \to gf$,
  $\epsilon \colon fg \to \id_{b}$ that satisfy the usual triangular
  identities.
\end{defn}
In other words, a left adjoint in $\tC$ is a left adjoint in the
homotopy 2-category of $\tC$; this definition is known to be
equivalent to a fully coherent notion of adjunction in an \itcat{}
\cite{RiehlVerityAdj}. In the case where $\tC$ is the \itcat{} of
\itcats{}, we can also find equivalent characterizations in terms of
mapping \icats{}:
\begin{propn}\label{propn:adj charns}
  The following are equivalent for a functor of \itcats{} $F \colon \tA \to \tB$:
  \begin{enumerate}[(1)]
  \item $F$ is a left adjoint in $\CATIT$.
  \item There exist a functor $G \colon \tB \to \tA$ and a \emph{counit
      transformation}
    $\epsilon \colon FG \to \id_{\tB}$, meaning a natural
    transformation 
    such that for all $a \in \tA$ and $b \in \tB$, the composite
    \[ \tA(a, Gb) \to \tB(Fa, FGb) \xto{\epsilon_{b,*}} \tB(Fa,b)
    \]
    is an equivalence.
  \item The presheaf $\tB(F(\blank), b) \colon \tA^{\op} \to \CATI$
    is representable for every $b \in \tB$.
  \end{enumerate}
\end{propn}
\begin{proof}
  This is shown more generally for enriched \icats{} in
  \cite{Stefanich}: The equivalence of (1) and (3) is a special case
  of \cite{Stefanich}*{Proposition 4.2.10} and that between (2) and
  (3) of \cite{Stefanich}*{Corollary 4.1.8}.
\end{proof}

\begin{cor}\label{cor:ladjfib}
  Suppose a functor $F \colon \tc{A} \to \tc{B}$ of \itcats{} is
  classified by the $0$-fibration $\pi \colon \tc{F} \to [1]$. Then
  $F$ is a left adjoint \IFF{} $\pi$ is also a $1$-fibration.
\end{cor}
\begin{proof}
  Suppose $\pi$ is also a 1-fibration. Then for $b \in \tc{B}$ we can choose a $\pi$-cartesian morphism $\gamma_{b} \colon Gb \to b$ over the non-invertible morphism in $[1]$. Picking a $\pi$-cocartesian morphism $Gb \to FGb$ we obtain a map $FGb \to b$ which yields a natural transformation of functors $\tA(-,Gb) \to \tB(F-,b)$ by the Yoneda lemma. 
  For the object $Gb$ in $\tc{A}$, composition with $\gamma_{b}$ gives a natural equivalence
  \[ \tc{A}(a, Gb) \simeq \tc{F}(a, Gb)_{\id_{0}} \to \tc{F}(a,b)_{0\to 1} \simeq \tc{B}(Fa,b).
  \]
  In other words, the presheaf $\tc{B}(F\blank,b)$ on $\tc{A}$ is represented by $Gb$. Since this holds for every object $b$, it follows from \cref{propn:adj charns} that $F$ is a left adjoint.

  Conversely, if $F$ has a right adjoint $G$ with a counit
  transformation $\epsilon \colon FG \to \id_{\tc{B}}$, then we want
  to show that $\pi$ is a 1-fibration. Given $b \in \tc{B}$, we define
  a morphism $e \colon Gb \to b$ in $\tF$ over $0 \to 1$ as the
  composite of the cocartesian morphism $u \colon Gb \to FGb$ and the counit $\epsilon_{b} \colon FGb \to b$. We claim that this is then a $\pi$-cartesian morphism. To see this, we must show that postcomposition with $e$ induces an equivalence of $\infty$-categories
 \[ \tc{A}(a, Gb) \simeq \tc{F}(a, b). 
  \]
  Let $\omega:a \to F(a)$ be a cocartesian morphism. We look at the commutative diagram
  \[
    \begin{tikzcd}
      \tc{A}(a,Gb) \arrow[dr,"e_*"] \arrow[d,"u_*"] & \\
      \tc{F}(a,FGb) \arrow[r,"(\epsilon_b)_* "]  & \tc{F}(a,b) \\
      \tc{B}(Fa,FGb) \arrow[u," \omega^*"] \arrow[u,swap,"\simeq"] \arrow[r,swap,"(\epsilon_b)_* "] & \tc{B}(Fa,b) \arrow[u,"\omega^*"] \arrow[u,swap,"\simeq"]
    \end{tikzcd}
  \]
  and note that the composite $\psi=(\omega^*)^{-1} \circ u_*$ can be identified with the action of the functor $F$ on mapping $\infty$-categories. Finally, we see that $(\epsilon_b)_* \circ \psi$ is an equivalence since $\epsilon$ is the counit of an adjunction, so that the map given by postcomposition with $e$ is an equivalence. The result now follows.
\end{proof}

\begin{remark}
  See \cite{AMGR}*{\S B.5} for another proof of
  \cref{cor:ladjfib} and the representability criterion for left
  adjoints from \cref{propn:adj charns} --- there the latter is proved first, as Lemma
  B.5.5, and then the representability criterion is deduced from this in
  Corollary B.5.6.
\end{remark}

\begin{propn}\label{lem:adjnotlocal}\ 
  \begin{enumerate}[(i)]
  \item   Suppose $p \colon \tc{E} \to \tc{B}$ is a $(0,1)$-fibration and a local $(1,1)$-fibration. Then $p$ is also a $(1,1)$-fibration.
  \item   Suppose $p \colon \tc{E} \to \tc{B}$ is a $(0,0)$-fibration and a local $(1,0)$-fibration. Then $p$ is also a $(1,0)$-fibration.
  \end{enumerate}
\end{propn}
\begin{proof}
  We prove (i); the proof of (ii) is the same. Suppose $\overline{f} \colon x \to y$ and $\overline{g} \colon y \to z$ are locally $p$-cartesian morphisms over $f \colon a \to b$ and $g \colon b \to c$. Then we are required to prove that the composite $\overline{g}\circ \overline{f}$ is again a locally $p$-cartesian morphism. Given $u \in \tc{E}_{a}$, let $\phi\colon u \to u'$ be a $p$-cocartesian morphism over $f$. We can then consider the following commutative diagram:
  \[
    \begin{tikzcd}
      \tc{E}_{a}(u,x) \arrow{r}{\overline{f}_{*}} & \tc{E}(u,y)_{f} \arrow{r}{\overline{g}_{*}} & \tc{E}(u,z)_{gf} \\
       & \tc{E}_{b}(u',y) \arrow{r}{\overline{g}_{*}} \arrow{u}{\phi^{*}} & \tc{E}(u',z)_{g}. \arrow{u}{\phi^{*}}
    \end{tikzcd}
  \]
  Here the top left and bottom right horizontal morphisms are equivalences since $\overline{f}$ and $\overline{g}$ are locally cartesian, while the vertical morphisms are equivalences since $\phi$ is cocartesian. It follows that the top right horizontal morphism is also an equivalence, hence so is the composite in the top row. Since this holds for any $u \in \tc{E}_{a}$, this shows that $\overline{g}\circ \overline{f}$ is also locally cartesian, as required.
\end{proof}

Combining \cref{cor:ladjfib} and \cref{lem:adjnotlocal}, we get:
\begin{cor}\label{cor:ladj+01is11}\
  \begin{enumerate}[(i)]
  \item Suppose $p \colon \tE \to \tB$ is the $(0,1)$-fibration for a
    functor $F \colon \tB \to \CATIT$. Then $p$ is also a
    $(1,1)$-fibration \IFF{} the functor $F(f)$ is a left adjoint for
    every morphism $f$ in $\tB$.
  \item Suppose $p \colon \tE \to \tB$ is the $(0,0)$-fibration for a
    functor $F \colon \tB^{\co} \to \CATIT$. Then $p$ is also a
    $(1,0)$-fibration \IFF{} the functor $F(f)$ is a left adjoint for
    every morphism $f$ in $\tB$. \qed
  \end{enumerate}
\end{cor}

\subsection{Relative adjunctions and adjoints in (op)lax transformations}\label{sec:reladj}

In this section we will extend the results on relative adjunctions
from \cite{HA}*{\S 7.3.2} to the setting of \itcats{}. We then
use this to identify adjunctions in \itcats{} of the form
$\FUN(\tc{A},\tc{B})^{\plax}$.

\begin{defn}
  A \emph{relative adjunction} over an \itcat{} $\tc{B}$ is an adjunction in the slice \itcat{} $\CATITsl{\tc{B}}$.
\end{defn}

\begin{propn}
  A commutative triangle
  \begin{equation}
    \label{eq:reladjtr}
    \begin{tikzcd}
      \tc{E} \arrow{rr}{G} \arrow{dr}[swap]{p} & & \tc{F} \arrow{dl}{q} \\
       & \tc{B}
    \end{tikzcd}
  \end{equation}
  is a relative right adjoint \IFF{} the functor $G$ has a left adjoint $F$ and for every object $x \in \tc{F}$ the unit map $x \to GFx$ is taken by $q$ to an equivalence in $\tc{B}$.
\end{propn}
\begin{proof}
  It is clear that the conditions are necessary. We observe that the
  projection functor $\CATITsl{\tc{B}} \to \CATIT$ is conservative on
  1- and 2-morphisms. This means that if we succesfully manage to lift
  the data defining the adjunction $F\dashv G$ in $\CATIT$ to
  $\CATITsl{\tc{B}}$, then it will follow that $G$ is a relative right adjoint.

  To view $F$ as a morphism over $\tB$, we use the equivalence $q
  \simeq qGF \simeq pF$ provided by our hypothesis. In particular,
  this means that the unit defines a 2-morphism over $\tB$. To finish
  the proof we need to show that the counit is also a 2-morphism over $\tB$. This amounts to showing that the identity 2-morphism on $p$ factors as
  \[
    p \to pFG \to p
  \]
  where the first map is the invertible 2-morphism used to view $FG$ as a morphism in $(\CATIT)_{/\tc{B}}$ and the last 2-morphism is obtained by applying $p$ to the counit. Unpacking the definitions we obtain a commutative diagram
  \[\begin{tikzcd}
  && pFG & p \\
  p & qG & qGFG & qG
  \arrow[from=2-1, to=2-2]
  \arrow[from=2-2, to=2-3]
  \arrow[from=2-3, to=1-3]
  \arrow[from=1-3, to=1-4]
  \arrow[from=2-3, to=2-4]
  \arrow[from=2-4, to=1-4]
\end{tikzcd}\]
which shows that our claim follows from the triangular identities.
\end{proof}

\begin{observation}\label{obs:pbreladj}
  If $\Phi \colon \tc{E}\to \tc{F}$ is a relative left or right adjoint
  over $\tc{B}$, then so is the pullback
  $\Phi \times_{\tc{B}} \tc{A} \colon \tc{E} \times_{\tc{B}} \tc{A}
  \to \tc{F} \times_{\tc{B}} \tc{A}$ over $\tA$ for any functor
  $\tc{A} \to \tc{B}$; this is because pullback is a functor of
  \itcats{} and so preserves adjunctions. In particular, for all
  $b \in \tc{B}$, the map on fibres $\tc{E}_{b} \to \tc{F}_{b}$ is
  also a left/right adjoint.
\end{observation}

\begin{lemma}\label{lem:cartenradj}
  Suppose we have a commutative triangle \cref{eq:reladjtr} such that
  the functors $p$ and $q$ are cartesian-enriched. If $G$ has either a
  left or a right relative adjoint over $\tB$, then $G$ is a morphism
  of cartesian-enriched functors.
\end{lemma}
\begin{proof}
  We must show that for objects $x,y \in \tE$, the functor
  $\tc{E}(x,y) \to \tc{F}(Gx,Gy)$ preserves cartesian morphisms over
  $\tB(px,py)$.  
  
  First suppose $G$ has a relative left adjoint $F$. Then we can
  identify our functor on mapping \icats{} with the composite
  \[ \tE(x,y) \to \tE(FGx,y) \isoto \tF(Gx,Gy)\]
  over $\tB(px,py)$. Here the first map is given by composition with the counit $FGx \to x$, and so preserves cartesian morphisms since $p$ is a cartesian-enriched functor, and the second is an equivalence.
  Similarly, if $G$ has a relative right adjoint $H$, we can identify the same functor as the composite
  \[ \tE(x,y) \to \tE(x,HGy) \isoto \tF(Gx,Gy) \] where the first map
  is given by composition with the unit $y \to HGy$, and so again
  preserves cartesian morphisms since $p$ is a cartesian-enriched
  functor.
\end{proof}

\begin{propn}\label{propn:11fibradj}
  Suppose we have a commutative triangle \cref{eq:reladjtr} such that the functors $p$ and $q$ are local $(1,1)$-fibrations. Then this triangle is a relative right adjoint \IFF{} the following conditions hold:
  \begin{enumerate}[(1)]
  \item For every $b \in \tc{B}$, the functor on fibres $\tc{E}_{b} \to \tc{F}_{b}$ is a right 
    adjoint.
  \item $G$ is a morphism of cartesian-enriched functors.    
  \item $G$ preserves locally cartesian morphisms.
  \end{enumerate}
\end{propn}
\begin{proof}
  We first show the conditions are necessary. The first follows from \cref{obs:pbreladj} and the second from \cref{lem:cartenradj}. For the third we suppose that $\overline{f} \colon x \to y$ is a locally cartesian morphism in $\tc{E}$ over $f \colon a \to b$; we must show that $G(\overline{f})$ is locally cartesian, \ie{} that for $z \in \tc{F}_{a}$ the map
  \[ \tc{F}_{a}(z, G(x)) \xto{G(\overline{f})_{*}} \tc{F}(z, G(y))_{f} \]
  is an equivalence. But if $F$ is the left adjoint of $G$, we can identify this with the map
  \[ \tc{E}_{a}(Fz,x) \xto{\overline{f}_{*}} \tc{E}(Fz, y)_{f},\]
  which is indeed an equivalence since $\overline{f}$ is a locally cartesian morphism.

  We now prove the converse. To show that $G$ has a left adjoint, it
  suffices to show that $\tc{F}(x, G(\blank))$ is corepresentable for
  all $x \in \tc{F}$. If $x$ lies over $b \in \tc{B}$ we will show
  that this functor is corepresented by $F_{b}(x)$ and the unit map
  $x \to G_{b}F_{b}x$, where $F_{b}$ is the fibrewise left adjoint. Given an object $y\in \tc{E}$ over $b' \in \tc{B}$, we have a commutative diagram
  \[
    \begin{tikzcd}
      \tc{E}(F_{b}x, y) \arrow{r} \arrow{dr} & \tc{F}(GF_{b}x, Gy) \arrow{r} \arrow{d} & \tc{F}(x, Gy) \arrow{dl} \\
      & \tc{B}(b, b').
    \end{tikzcd}
  \]
  We want to show the horizontal composite is an equivalence, and
  since the functors to $\tc{B}(b,b')$ are cartesian fibrations and
  the horizontal maps preserve cartesian morphisms it suffices to show that
  we get an equivalence on the fibre over each $f \colon b \to b'$. If
  $\overline{f} \colon y' \to y$ is a locally cartesian morphism over
  $f$, then we have a commutative diagram
  \[
    \begin{tikzcd}
      \tc{E}_{b}(F_{b}x, y') \arrow{d}{\overline{f}_{*}} \arrow{r} & \tc{F}_{b}(GF_{b}x,Gy') \arrow{d}{G\overline{f}_{*}} \arrow{r} & \tc{F}_{b}(x, Gy') \arrow{d}{G\overline{f}_{*}} \\
      \tc{E}(F_{b}x, y)_{f} \arrow{r} & \tc{F}(GF_{b}x, Gy)_{f} \arrow{r} & \tc{F}(x,Gy)_{f}.
    \end{tikzcd}
  \]
  Here the composite in the top row is an equivalence since $F_{b}$ is left adjoint to $G$ over $b$, and the vertical maps are equivalences since $\overline{f}$ and $G\overline{f}$ are locally cartesian.
\end{proof}

\begin{propn}\label{propn:fib01radj}
  Suppose we have a commutative triangle \cref{eq:reladjtr} such that the functors $p$ and $q$ are local $(0,1)$-fibrations. Then this triangle is a relative right adjoint \IFF{} the following conditions hold:
  \begin{enumerate}[(1)]
  \item For every $b \in \tc{B}$, the functor on fibres $G_{b} \colon \tc{E}_{b} \to \tc{F}_{b}$ has a left adjoint $F_{b}$.
  \item $G$ is a morphism of cartesian-enriched functors.
  \item The fibrewise left adjoints are compatible with locally
    cocartesian morphisms, in the following sense: Given
    $x \in \tc{F}$ over $b \in \tc{B}$ and a morphism
    $f \colon b \to b'$, let $\overline{f}_{\tE} \colon F_{b}(x) \to f_{!}F_{b}(x)$ be
    locally cocartesian over $f$ in $\tE$. We can factor the composite
    \[ x \to GF_{b}(x) \xto{G\overline{f}_{\tE}} Gf_{!}F_{b}(x),\]
    where the first map is the unit of the fibrewise adjunction, through the locally cocartesian morphism $\overline{f}_{\tF} \colon x \to f_{!}x$ over $f$ in $\tF$, giving a morphism $f_{!}x \to Gf_{!}F_{b}(x)$. We require the adjoint map
    \[ F_{b'}f_{!}x \to f_{!}F_{b}(x)\]
    to be an equivalence.
  \end{enumerate}
\end{propn}
\begin{proof}
  Suppose first that $F$ is a relative left adjoint of $G$; we want to
  show that the three conditions are necessary. The first condition
  follows from \cref{obs:pbreladj} and the second from
  \cref{lem:cartenradj}. We note that an analogous argument to that of \cref{propn:11fibradj} shows that $F$ preserves locally cocartesian morphisms. The third condition follows from this: we can identify the locally cocartesian morphism $\overline{f}_{\tE} \colon F_{b}(x) \to f_{!}F_{b}(x)$  with $F(\overline{f}_{\tF}) \colon F(x) \to F(f_!x)$, and the condition then follows from the triangular identities.

  We now prove the converse. To show that $G$ has a left adjoint, it
  suffices to show that $\tc{F}(x, G(\blank))$ is representable for
  all $x \in \tc{F}$. If $x$ lies over $b \in \tc{B}$ we will show
  that this functor is represented by $F_{b}(x)$ and the unit map
  $x \to G_{b}F_{b}x$, where $F_{b}$ is the fibrewise left
  adjoint. Given an object $y\in \tc{E}$ over $b' \in \tc{B}$, we have
  a commutative diagram
  \[
    \begin{tikzcd}
      \tc{E}(F_{b}x, y) \arrow{r} \arrow{dr} & \tc{F}(GF_{b}x, Gy) \arrow{r} \arrow{d} & \tc{F}(x, Gy) \arrow{dl} \\
      & \tc{B}(b, b').
    \end{tikzcd}
  \]
  We want to show that the horizontal composite is an equivalence; since the functors to $\tc{B}(b,b')$ are cartesian fibrations and
  the horizontal maps preserve cartesian morphisms, it suffices to
  show that we get an equivalence on the fibre over each
  $f \colon b \to b'$. If $\overline{f} \colon x \to x'$ is a locally
  cocartesian morphism over $f$, then condition (3) guarantees that we
  have a locally cocartesian edge
  $g \colon F_b(x) \to f_!F_b(x) \simeq F_{b'}x'$ over $f$ and a
  commutative diagram
  \[
    \begin{tikzcd}
      \tc{E}_{b'}(F_{b'}x', y) \arrow{r} \arrow{d}{g^{*}} & \tc{F}_{b'}(GF_{b'}x', Gy) \arrow{r} & \tc{F}_{b'}(x', Gy) \arrow{d}{\overline{f}^{*}} \\
      \tc{E}(F_{b}x, y)_{f} \arrow{r} & \tc{F}(GF_{b}x, Gy)_{f} \arrow{r} & \tc{F}(x, Gy)_{f}.
    \end{tikzcd}
  \]
  Here the composite in the top row is an equivalence since $F_{b'}$ is left adjoint to $G$ over $b'$, and the vertical maps are equivalences since $\overline{f}$ and $g$ are locally cocartesian.
\end{proof}

\begin{defn}\label{def:matesquare}
  Let us consider a lax square in an $(\infty,2)$-category $\tc{E}$,
   \begin{equation}
    \label{eq:matesqualeftadj}
      \begin{tikzcd}
      a_x \arrow[d,"f_a",swap] \arrow[r,"R_x"] & b_x \arrow[d,"f_b"] \arrow[dl,Rightarrow,shorten >=1.5ex,shorten <=1.5ex] \\
      a_{y} \arrow[r,swap,"R_{y}"] & b_y,
    \end{tikzcd}
   \end{equation}
  where the functors $R_{x}, R_{y}$ admit left adjoints $L_x \dashv R_x$ and $L_{y} \dashv R_{y}$. Then we can construct a natural transformation
  \[
    L_{y} \circ f_b \to   L_{y} \circ f_b \circ R_{x}\circ L_x \to  L_{y} \circ R_{y} \circ f_a\circ L_x \to f_a \circ L_x,
  \]
  which yields another lax square,
   \begin{equation}
      \label{eq:matesquarerightadj}
      \begin{tikzcd}
      b_x \arrow[d,"f_b",swap] \arrow[r,"L_x"] & a_x \arrow[d,"f_a"]  \\
      b_y \arrow[r,swap,"L_{y}"] \arrow[ur,Rightarrow,shorten >=1.5ex, shorten <=1.5ex] & a_y,
    \end{tikzcd}
   \end{equation}
   which we call the \emph{mate} of \cref{eq:matesqualeftadj}. Note that we can similarly start with a diagram of the form \cref{eq:matesquarerightadj} and construct its associated mate of the form \cref{eq:matesqualeftadj} by a dual procedure.
\end{defn}

\begin{remark}
  The third condition in the statement of \cref{propn:fib01radj} can be rephrased as follows: Given an edge $f: b \to b'$, it follows from \cref{thm:fib01laxstr} that we have laxly commuting diagram in $\CAT_{(\infty,2)}$ of the form
  \[
    \begin{tikzcd}
      \tc{E}_b \arrow[d,"\tc{E}(f)",swap] \arrow[r,"G_b"] & \tc{F}_b \arrow[d,"\tc{F}(f)"] \arrow[dl,Rightarrow,shorten >=1.5ex,shorten <=1.5ex] \\
      \tc{E}_{b'} \arrow[r,swap,"G_{b'}"] & \tc{F}_b'.
    \end{tikzcd}    
  \]
   Then condition (3) says precisely that the associated mate square of this diagram is commutative.
\end{remark}

Combining our results on relative adjunctions with our straightening
equivalences for (op)lax transformations, we deduce the following:
\begin{cor}\label{cor:adjinlaxfuncat}
  Let $\tA$ be an \itcat{} and consider functors $F,G \colon \tA \to \CATI$.
  \begin{enumerate}[(1)]
  \item A lax transformation $R \colon F \to G$ has a left adjoint in $\FUN(\tc{A}, \CATI)^{\lax}$ \IFF{} $R_{x} \colon F(x) \to G(x)$ has a left adjoint for every $x \in \tc{A}$ and the mate of the lax naturality square
  \[
    \begin{tikzcd}
      F(x) \arrow[d,"F(\phi)",swap] \arrow[r,"R_{x}"] & G(x) \arrow[d,"G(\phi)"] \arrow[dl,Rightarrow,shorten >=1.5ex,shorten <=1.5ex] \\
      F(x') \arrow[r,swap,"R_{x'}"] & G(x')
    \end{tikzcd}    
  \]
  commutes for every morphism $\phi \colon x \to x'$ in $\tc{A}$.
\item A lax transformation $L \colon F \to G$ has a right adjoint in
  $\FUN(\tc{A}, \CATI)^{\lax}$ \IFF{} $L_{x} \colon F(x) \to G(x)$ has
  a right adjoint for every $x \in \tc{A}$ and the lax transformation $L$ is strong
  (\ie{}, its lax naturality squares all commute).
  \item An oplax transformation $L \colon F \to G$ has a right adjoint
    in $\FUN(\tc{A}, \CATI)^{\oplax}$ \IFF{}
    $L_{x} \colon F(x) \to G(x)$ has a right adjoint for every
    $x \in \tc{A}$ and the mate of the oplax naturality square
  \[
    \begin{tikzcd}
      F(x) \arrow[d,"F(\phi)",swap] \arrow[r,"L_{x}"] & G(x) \arrow[d,"G(\phi)"] \arrow[dl,Leftarrow,shorten >=1.5ex,shorten <=1.5ex] \\
      F(x') \arrow[r,swap,"L_{x'}"] & G(x')
    \end{tikzcd}    
  \]
    commutes for every morphism
    $\phi \colon x \to x'$ in $\tc{A}$.
  \item An oplax transformation $R \colon F \to G$ has a left adjoint
    in $\FUN(\tc{A}, \CATI)^{\oplax}$ \IFF{}
    $R_{x} \colon F(x) \to G(x)$ has a left adjoint for every
    $x \in \tc{A}$ and the oplax transformation $R$ is strong (\ie{},
    its oplax naturality squares all commute). 
  \end{enumerate}
\end{cor}
\begin{proof}
  The result follows by combining \cref{propn:11fibradj} and \cref{propn:fib01radj} together with \cref{thm:fib01laxstr}.
\end{proof}

We can extend this description of adjoints to more general targets
using the Yoneda embedding for (op)lax transformations from
\S\ref{sec:laxyoneda}:
\begin{cor}\label{cor:adjinlaxtr}
  Let $\tc{A}, \tc{B}$ be \itcats{} and consider functors $F,G \colon \tc{A} \to \tc{B}$.
  \begin{enumerate}[(1)]
  \item A lax transformation $R \colon F \to G$ has a left adjoint in $\FUN(\tc{A}, \tc{B})^{\lax}$ \IFF{} $R_{x} \colon F(x) \to G(x)$ has a left adjoint for every $x \in \tc{A}$ and the mate of the lax naturality square commutes for every morphism in $\tA$.
  \item A lax transformation $L \colon F \to G$ has a right adjoint in
    $\FUN(\tc{A}, \tc{B})^{\lax}$ \IFF{} $L_{x} \colon F(x) \to G(x)$
    has a right adjoint for every $x \in \tc{A}$ and $L$ is strong.
  \item An oplax transformation $L \colon F \to G$ has a right adjoint
    in $\FUN(\tc{A}, \tc{B})^{\oplax}$ \IFF{}
    $L_{x} \colon F(x) \to G(x)$ has a right adjoint for every
    $x \in \tc{A}$ and the mate of the oplax naturality square
    commutes for every morphism in $\tc{A}$.
  \item An oplax transformation $R \colon F \to G$ has a left adjoint in $\FUN(\tc{A}, \tc{B})^{\oplax}$ \IFF{} $R_{x} \colon F(x) \to G(x)$ has a left adjoint for every $x \in \tc{A}$ and $R$ is strong.
  \end{enumerate}
\end{cor}
\begin{proof}
  By \cref{funlaxemb}, we have an embedding
  \[ \iota:\FUN(\tc{A}, \tc{B})^{\lax} \to \FUN(\tA \times \tB^{\op},\CATI)^{\lax}\] as a locally full sub-\itcat{}. It follows that a lax transformation has an adjoint
  in $\FUN(\tc{A}, \tc{B})^{\lax}$ \IFF{} it has an adjoint in
  $\FUN(\tA \times \tB^{\op},
  \CATI)^{\lax}$, and this adjoint also lies in this sub-\itcat{}. 

  We prove the first two cases, since the statements in the oplax case are strictly dual. Given a map $R \colon F \to G$ in $\FUN(\tc{A}, \tc{B})^{\lax}$ we will fix the notation $R' := \iota(R)$.
  \begin{enumerate}[(1)]
    \item Let us suppose that $R \colon F \to G$ has a left adjoint in $\FUN(\tc{A}, \tc{B})^{\lax}$ given by $L \colon G \to F$. Then it is clear that $R$ is a pointwise right adjoint. Given a map $f \colon x \to y$ in $\tc{A}$,  the image of the mate square of $f$ under $\tc{B}(b,-)$ can be identified with the mate square of the map $(b,x) \to (b,y)$ in $ \FUN(\tA \times \tB^{\op},\CATI)^{\lax}$. It then follows from \cref{propn:fib01radj} that the latter mate must be commutative. The Yoneda Lemma then implies that the original mate square is commutative as well.

      To show the converse we note that the map $\iota$ preserves
      pointwise adjoints and mate squares. Therefore the map $R'$
      admits a left adjoint, which is strong by
      \cref{propn:fib01radj}. This left adjoint must then factor
      through the full sub-\itcat{} of
      $\FUN(\tA \times \tB^{\op}, \CATI)^{\lax}$ described in
      \cref{funlaxemb}, so the claim follows.

    \item Let us suppose that $L \colon G \to F$ has a right adjoint
      in $\FUN(\tc{A}, \tc{B})^{\lax}$ given by $R \colon F \to
      G$. The image $L'$ of $L$ under $\iota$ is again a left adjoint,
      so that $L'$ is strong by \cref{cor:adjinlaxfuncat}. The Yoneda
      Lemma implies that $L$ is then also strong.

      To show the converse, we observe that $L'$ satisfies the
      conditions to have a right adjoint $R'$ in
      $\FUN(\tA \times \tB^{\op},\CATI)^{\lax}$ from
      \cref{cor:adjinlaxfuncat}. We must show that $R'$ is also in the
      image of $\iota$, which by \cref{funlaxemb} means that we need
      to show that $R'(a,\blank)$ is strong for every $a \in
      \tc{A}$. To see this, we observe that $R'(a,\blank)$ is right
      adjoint to $L'(a,\blank)$ in $\FUN(\tB^{\op},
      \CATI)^{\lax}$. Here $L'(a,\blank)$ is the image of $L(a)$ under
      the Yoneda embedding; by assumption, $L(a)$ has a right adjoint
      $R(a)$ in $\tB$, so by uniqueness of adjoints $R'(a,\blank)$ is
      the image of $R(a)$ under the Yoneda embedding, and so is
      necessarily strong. \qedhere
  \end{enumerate}
\end{proof}

\begin{remark}\label{rmk:adjinlaxtrother}
  The description of adjunctions from \cref{cor:adjinlaxtr} also
  appears as \cite{adjmnd}*{Theorem 4.6}, but the proof there is
  incomplete (and completing it would require working with coherence
  data in 4-categories).
\end{remark}

We can also immediately deduce a description of adjoints in ordinary
functor \itcats{}, since adjunctions there must necessarily also be
adjunctions in the larger \itcats{} of (op)lax transformations:
\begin{cor}
  Let $\tA$ and $\tB$ be \itcats{}, and consider functors $F,G \colon \tA \to \tB$.
  \begin{enumerate}[(1)]
  \item   A natural transformation $R \colon F \to G$ has a left
    adjoint in $\FUN(\tc{A},\tB)$ \IFF{} $R_{x} \colon F(x) \to
    G(x)$ has a left adjoint for every $x \in \tA$ and the mates of all the naturality squares commute.
  \item   A natural transformation $L \colon F \to G$ has a right
    adjoint in $\FUN(\tc{A},\tB)$ \IFF{} $L_{x} \colon F(x) \to
    G(x)$ has a right adjoint for every $x \in \tA$ and the mates of all the naturality squares
    commute. \qed
  \end{enumerate}
\end{cor}

\subsection{The mate correspondence}\label{sec:mates}
In this section we will show that taking mates gives an equivalence
between lax transformations whose components are right adjoints and
oplax transformations whose components are left adjoints, in the following sense:
\begin{defn}
  Let $\FUN(\tc{A},\tc{B})^{\lax,\radj}$ be the wide and locally full
  sub-\itcat{} of $\FUN(\tc{A},\tc{B})^{\lax}$ containing those lax
  transformations whose components are all right adjoints. We define
  $\FUN(\tc{A},\tc{B})^{\oplax,\ladj}$ similarly.
\end{defn}
Our more precise goal in this section is then to prove that there is an equivalence
\[ \FUN(\tc{A},\tc{B})^{\lax,\radj} \simeq
  (\FUN(\tc{A},\tc{B})^{\oplax,\ladj})^{\coop},\] given by taking the
mates of the naturality squares. Our strategy will be to first prove
this for functors to $\CATI$ using fibrations and then reduce the
general case to this using the Yoneda embedding. We start by giving an informal   
description of the relevant \itcats{} from the fibrational perspective:
\begin{enumerate}[(1)]
\item Given a functor $\tX \to \FUN(\tA, \CATI)^{\lax,\radj} \simeq
  \FIB^{(1),\lax}_{(0,1)}(\tA)^{\radj}$,  we can unstraighten 
  over $\tX^{\op}$:
  \[
    \begin{tikzcd}
      \tE \arrow{rr} \arrow{dr} & & \tX^{\op} \times \tA \arrow{dl} \\
        & \tX^{\op},
    \end{tikzcd}
  \]
  where this is a morphism of $(1,0)$-fibrations over $\tX^{\op}$ that is fibrewise
  a $(0,1)$-fibration over $\tA$.
\item Then by \cref{propn:graybifibmor} we get that
    \[
    \begin{tikzcd}
      \tE \arrow{rr} \arrow{dr} & & \tX^{\op} \times \tA \arrow{dl} \\
        & \tA,
    \end{tikzcd}
  \]
  is a morphism of $(0,1)$-fibrations over $\tA$ that is fibrewise a
  $(1,0)$-fibration \emph{and} a $(0,0)$-fibration (encoding the right
  adjoints).
\item Now we can dualize over $\tA$ to get
    \[
    \begin{tikzcd}
      \tE^{\vee} \arrow{rr} \arrow{dr} & & \tX^{\op} \times \tA^{\coop} \arrow{dl} \\
        & \tA^{\coop},
    \end{tikzcd}
  \]
  a morphism of $(1,1)$-fibrations over $\tA^{\coop}$ that is
  fibrewise a $(0,0)$-fibration.
\item Using \cref{propn:graybifibmor} again, we obtain
      \[
    \begin{tikzcd}
      \tE^{\vee} \arrow{rr} \arrow{dr} & & \tX^{\op} \times \tA^{\coop} \arrow{dl} \\
        & \tX^{\op}
    \end{tikzcd}
  \]
  a morphism of $(0,0)$-fibrations that is fibrewise a
  $(1,1)$-fibration.
\item Now we can unstraighten over $\tX^{\op}$ to get a functor
  \[\tX^{\coop} \to \FIB^{(1),\lax}_{(1,1)}(\tA^{\coop}) \simeq \FUN(\tA,
  \CATI)^{\oplax}\]
  landing in the sub-\itcat{} $\FUN(\tA, \CATI)^{\oplax,\ladj}$, or
  \[ \tX \to \FUN(\tA,\CATI)^{\oplax,\ladj,\coop}.\]
\end{enumerate}

Next, we make these maneuvres more precise:
\begin{propn}
  Under the equivalences given in \cref{prop:mapspacegbifib} we have that:
  \begin{enumerate}
  \item The subspace
    \[ \Map(\tX,\FUN(\tc{A},\CATI)^{\lax,\radj}) \subseteq
      \Map(\tX,\FUN(\tc{A},\CATI)^{\lax}) \]
    corresponds to the subspace
    \[ \GFIB^{(1)}_{(0,1);(1,0)+(0,0)}(\tA,\tX^{\op})^{\simeq}
      \subseteq 
      \GBIFIB^{(1)}_{(0,1)}(\tA,\tX^{\op})^{\simeq}\]
   consisting of functors that are both 1-fibred Gray $(0,1)$-bifibrations and Gray $(0,1)$-$(0,0)$-fibrations.
 \item The subspace 
   \[ \Map(\tX, \FUN(\tA, \CATI)^{\oplax,\ladj})
     \subseteq \Map(\tX, \FUN(\tA, \CATI)^{\oplax})
   \]
   corresponds to the subspace
   \[
     \GFIB^{(1)}_{(1,1);(1,0)+(0,0)}(\tX^{\coop},\tA^{\co})^{\simeq}
     \subseteq
     \GBIFIB^{(1)}_{(1,1)}(\tX^{\coop},\tA^{\co})^{\simeq}
     \]
     consisting of functors that are both 1-fibred Gray $(1,1)$-bifibrations and Gray $(1,1)$-$(1,0)$-fibrations.
  \end{enumerate}
\end{propn}
\begin{proof}
  We prove the first case; the proof of the remaining case is
  essentially the same.  By definition, a functor
  $F \colon \tX \to \FUN(\tA, \CATI)^{\lax}$ factors through
  $\FUN(\tc{A},\CATI)^{\lax,\radj}$ \IFF{} for every morphism
  $f \colon x \to y$ in $\tX$ and every object $a \in \tA$, the
  functor $F(f,a) \colon F(x,a) \to F(y,a)$ is a right adjoint. If
  $p \colon \tE \to \tX^{\op} \times \tA$ is the corresponding Gray
  $(1,0)$-bifibration, then this amounts to the pullback
  $f^{*}\tE_{a} \to [1]^{\op}$ being the $(1,0)$-fibration for a
  functor that has a left adjoint. \cref{cor:ladj+01is11} then implies
  that this corresponds precisely to the condition that the
  $(1,0)$-fibration $\tE_{a} \to \tX^{\op}$ is also a
  $(0,0)$-fibration for every $a \in \tA$, as required.
\end{proof}

\begin{cor}\label{cor:catvaluedmates}
  For any \itcat{} $\tA$, there is a natural equivalence
  \[ \FUN(\tA, \CATI)^{\lax,\radj} \simeq \FUN(\tA, \CATI)^{\oplax,\ladj,\coop}\]
  given by taking mates of the naturality squares of lax transformations.
\end{cor}

\begin{proof}
  There are natural equivalences
  \[
    \begin{split}
      \Map(\tX, \FUN(\tA, \CATI)^{\lax,\radj}) & \simeq
                                                 \GFIB^{(1)}_{(0,1);(1,0)+(0,0)}(\tA,\tX^{\op})^{\simeq}
      \\
                                               & \simeq \GFIB^{(1)}_{(1,1);(1,0)+(0,0)}(\tA^{\coop},\tX^{\op})^{\simeq} \\
                                               & \simeq \Map(\tX^{\coop}, \FUN(\tA, \CATI)^{\oplax,\ladj}) \\
       & \simeq \Map(\tX, \FUN(\tA, \CATI)^{\oplax,\ladj,\coop}),
    \end{split}
  \]
  where the second is given by dualizing from $(0,1)$-fibrations over $\tA$  to 
  $(1,1)$-fibrations over $\tA^{\coop}$.
  By the Yoneda lemma, it follows that we have the required
  equivalence of \itcats{}.

  To show the final claim we can assume without loss of generality that $\tA=[1]$. In this case our equivalence coincides with that of \cite{HHLN1}, so the result follows from \cite{HHLN1}*{Proposition 3.2.7}.
\end{proof}

To extend this result to a general equivalence
\[ \FUN(\tA, \tB)^{\lax,\radj} \simeq \FUN(\tA,
  \tB)^{\oplax,\ladj,\coop},\]
we will first identify the images of these \itcats{} under the embedding of \cref{funlaxemb}. We will then show that the equivalence of \cref{cor:catvaluedmates} restricts appropriately to yield a commutative diagram
  \[
    \begin{tikzcd}
      \FUN(\tA, \tB)^{\lax,\radj} \arrow[r,"\simeq"] \arrow[d] & \FUN(\tA,\tB)^{\oplax,\ladj,\coop} \arrow[d] \\
      \FUN(\tA\times \tB^{\op}, \CATI)^{\lax,\radj} \arrow[r,"\simeq"] & \FUN(\tA \times \tB^{\op}, \CATI)^{\oplax,\ladj,\coop}.
    \end{tikzcd}
  \]

\begin{lemma}\label{lem:emblaxradj}
  Under the locally full embedding of \cref{funlaxemb},
  $\FUN(\tA,\tB)^{\lax,\radj}$ corresponds to a locally full sub-\itcat{}
  of $\FUN(\tA \times \tB^{\op},\CATI)^{\lax,\radj}$, whose morphisms are the lax transformations $\Phi
  \colon [1] \otimes (\tA \times \tB^{\op}) \to \CATI$ such that
  \begin{itemize}
  \item For every $a \in \tA$, $\Phi(\blank,a,\blank)$ is an ordinary
    natural transformation, \ie{} for every morphism $b \to b'$ in
    $\tB$ the square
    \[
      \begin{tikzcd}
        \Phi(0,a,b) \arrow{r}{\Phi(\blank,a,b)} \arrow{d} &
        \Phi(1,a,b) \arrow[Rightarrow,dl,shorten >=1.5ex,shorten <=1.5ex] \arrow{d} \\
        \Phi(0,a,b') \arrow[r,swap,"{\Phi(\blank,a,b')}"]  & \Phi(1,a,b')
      \end{tikzcd}
    \]
    commutes.
  \item If $L_{a,b}$ denotes the left adjoint of $\Phi(\blank,a,b)$,
    then for every morphism $b \to b'$ the mate square
    \[
      \begin{tikzcd}
        \Phi(1,a,b) \arrow{r}{L_{a,b}} \arrow{d} & \Phi(0,a,b) \arrow{d} \\
        \Phi(1,a,b') \arrow[r,swap,"L_{a,b'}"] \arrow[Rightarrow,ur,shorten >=1.5ex,shorten <=1.5ex]  & \Phi(0,a,b')
      \end{tikzcd}
    \]
    commutes.
  \end{itemize}
\end{lemma}
\begin{proof}
  The first condition is the one for $\Phi$ to live in
  $\FUN(\tA, \tB)^{\lax}$, and the second condition says by
  \cref{cor:adjinlaxfuncat} that for every $a \in \tA$ we get a left adjoint in
  $\FUN(\tB^{\op}, \CATI)$ and so in $\tB$, as required.
\end{proof}

\begin{lemma}\label{lem:emboplaxladj}
 Under the locally full embedding of \cref{funlaxemb},
 $\FUN(\tA,\tB)^{\oplax,\ladj}$ corresponds to a locally full sub-\itcat{}
  of $\FUN(\tA \times \tB^{\op},\CATI)^{\oplax,\ladj}$, whose morphisms are the oplax transformations $\Phi
  \colon  (\tA \times \tB^{\op})\otimes  [1] \to \CATI$ such that
  \begin{itemize}
  \item For every $a \in \tA$, $\Phi(\blank,a,\blank)$ is an ordinary
    natural transformation, \ie{} for every morphism $b \to b'$ in
    $\tB$ the square
    \[
      \begin{tikzcd}
        \Phi(0,a,b) \arrow{r}{\Phi(\blank,a,b)} \arrow{d} &
        \Phi(1,a,b) \arrow{d} \\
        \Phi(0,a,b') \arrow[r,swap,"{\Phi(\blank,a,b')}"]  \arrow[Rightarrow,ur,shorten >=1.5ex,shorten <=1.5ex]  & \Phi(1,a,b')
      \end{tikzcd}
    \]
    commutes.
  \item If $R_{a,b}$ denotes the right adjoint of $\Phi(\blank,a,b)$,
    then for every $b \to b'$ the mate square
    \[
      \begin{tikzcd}
        \Phi(1,a,b) \arrow{r}{R_{a,b}} \arrow{d} &
        \Phi(0,a,b) \arrow[Rightarrow,dl,shorten >=1.5ex,shorten <=1.5ex] \arrow{d} \\
        \Phi(1,a,b') \arrow[r,swap,"R_{a,b'}"]  & \Phi(0,a,b')
      \end{tikzcd}
    \]
    commutes.
  \end{itemize}
\end{lemma}
\begin{proof}
  This is dual to \cref{lem:emblaxradj} (cf. \cref{rem:funoplaxemb}).
\end{proof}

\begin{thm}\label{thm:mategeneral}
  Let $\tA,\tB$ be $(\infty,2)$-categories. Then there exists a natural equivalence of $(\infty,2)$-categories
  \[
    \FUN(\tA,\tB)^{\lax,\radj} \simeq (\FUN(\tA,\tB)^{\oplax,\ladj})^{\coop}
  \]
  given by taking mates of the naturality squares of lax transformations.
\end{thm}
\begin{proof}
  Observe that by \cref{lem:emblaxradj} and \cref{lem:emboplaxladj} we only need to check two things:
  \begin{itemize}
    \item The equivalence \[\FUN(\tA \times \tB^{\op},\CATI)^{\lax,\radj} \simeq (\FUN(\tA \times \tB^{\op},\CATI)^{\oplax,\ladj})^{\coop}\] given in \cref{cor:catvaluedmates} restricts to an equivalence \[\FUN(\tA,\tB)^{\lax,\radj} \simeq (\FUN(\tA,\tB)^{\oplax,\ladj})^{\coop}.\]
    \item The equivalence above is given by taking mates.
  \end{itemize}
  We start by showing the first assertion. Let
  $\Phi \colon [1] \otimes (\tA \times \tB^{\op}) \to \CATI$ be a morphism in the image of $\FUN(\tA,\tB)^{\lax,\radj}$,
  and denote by $\widetilde{\Phi}$ its image in the right-hand side of the equivalence. To show that $\widetilde{\Phi}(\blank,a,\blank)$ is a
  strong natural transformation for every $a \in \tA$, we observe that for
  every $b \to b'$ the associated square is the mate of the square
  associated to $\Phi(\blank,a,\blank)$, which is commutative by
  assumption. The remaining condition follows after observing that the
  mate of the square associated to $b \to b'$ and
  $\widetilde{\Phi}(\blank,a,\blank)$ is again commutative since
  $\Phi(\blank,a,\blank)$ is by assumption a strong natural transformation.

  For the final part, we can assume without loss of generality that
  $\tA=[1]$. Let $\eta \colon[1] \to \FUN([1],\tB)^{\lax,\radj}$ and
  denote by $\widetilde{\eta}$ the image of $\eta$ under the
  equivalence constructed above. We further define $\psi$ as the mate
  of the square determined by $\eta$. To finish the proof we must show
  that $\widetilde{\eta} \simeq \psi$. Now, if we denote by $\Phi$,
  $\widetilde{\Phi}$ and $\Psi$ the images of
  $\eta,\widetilde{\eta}$ and $\psi$ in
  $\FUN(\tA \times \tB^{\op},\CATI)^{\lax,\radj}$ and
  $\FUN(\tA \times \tB^{\op},\CATI)^{\oplax,\ladj})^{\coop}$, as appropriate, it
  follows from \cref{cor:catvaluedmates} that
  $\widetilde{\Phi} \simeq \Psi$ which then implies that
  $\widetilde{\eta}$ is indeed the same as $\psi$.
\end{proof}
Taking $\tA=[0]$ in the previous theorem, we obtain the following:
\begin{cor}
  Let $\tB$ be an \itcat{}, and define $\tB^{\radj}$ and $\tB^{\ladj}$
  as the locally full sub-\itcats{} on those morphisms that are right
  and left adjoints, respectively. Then there is an equivalence
  $\tB^{\radj} \simeq (\tB^{\ladj})^{\coop}$, given on morphisms by
  taking adjoints. \qed
\end{cor}

As a variant of \cref{thm:mategeneral}, we can also obtain a mate equivalence between functors that take values in left and right adjoints:
\begin{defn}\label{defn:leftadjfuncat}
  Let $\tA$, $\tB$ be $(\infty,2)$-categories. We define the full
  sub-\itcat{} $\LFUN(\tA,\tB)^{\lax}$ of $\FUN(\tA,\tB)^{\lax}$ to
  consist of the functors that send each morphism of $\tA$ to a
  left adjoint in $\tB$. We similarly denote by
  $\RFUN(\tA,\tB)^{\oplax}$ the analogous construction using right
  adjoints.
\end{defn}

\begin{thm}\label{thm:adjointtoafunctor}
   Let $\tA,\tB$ be $(\infty,2)$-categories. Then there exists a natural equivalence of $(\infty,2)$-categories
   \[
     \LFUN(\tA,\tB)^{\lax} \simeq \RFUN(\tA^{\coop},\tB)^{\oplax}
   \]
   which sends a functor $F \colon \tA \to \tB$ to a functor $F^{\vee} \colon \tA^{\coop} \to \tB$ that takes each 1-morphism to the 1-morphism right adjoint to its image under $F$.
\end{thm}
\begin{proof}
  Let $\tX$ be an $(\infty,2)$-category and observe that we have natural equivalences
  \[
    \Map(\tX, \LFUN(\tA,\tB)^{\lax}) \simeq \Map(\tA,\FUN(\tX,\tB)^{\oplax,\ladj}) \simeq \Map(\tA,(\FUN(\tX,\tB)^{\lax,\radj})^{\coop})
  \]
  where the first equivalence is given by the universal property of the Gray tensor product and the last one is given by \cref{thm:mategeneral}. Dually, we have
  \[
    \Map(\tX,\RFUN(\tA^\coop,\tB)^{\oplax})
    \simeq 
    \Map(\tA,(\FUN(\tX,\tB)^{\lax,\radj})^{\coop}),
  \]
  and so the Yoneda lemma implies that $ \LFUN(\tA,\tB)^{\lax} \simeq \RFUN(\tA^{\coop},\tB)^{\oplax}$, as desired.
\end{proof}


\end{document}